\documentclass[10pt]{amsart}
\usepackage{latexsym}
\usepackage{amsmath}
\usepackage{amssymb}
\usepackage{mathrsfs}
\usepackage{graphicx}
\usepackage{color}
\usepackage{pgfpages}
\usepackage{ifthen}
\usepackage{leftidx,tensor}
\usepackage[T1]{fontenc}
\usepackage[latin1]{inputenc}
\usepackage{mathtools}
\usepackage{comment}
\usepackage{dsfont}
\usepackage[nocompress]{cite}
\usepackage[shortlabels]{enumitem}
\usepackage{aliascnt}
\usepackage[bookmarks=true,pdfstartview=FitH, pdfborder={0 0 0}, colorlinks=true,citecolor=red, linkcolor=blue]{hyperref}
\usepackage{bbm}
\usepackage{nicefrac}
\usepackage{pgfplots}
\pgfplotsset{compat=1.18}
\usetikzlibrary{arrows.meta}
\theoremstyle{plain}
\newtheorem{thm}{Theorem}[section]
\newaliascnt{cor}{thm}
\newaliascnt{prop}{thm}
\newaliascnt{lem}{thm}
\newaliascnt{remark}{thm}
\newtheorem{cor}[cor]{Corollary}

\newtheorem{prop}[prop]{Proposition}
\newtheorem{lem}[lem]{Lemma}
\newtheorem{remark}[remark]{Remark}
\aliascntresetthe{cor}
\aliascntresetthe{prop}
\aliascntresetthe{lem} 
\aliascntresetthe{remark}
\theoremstyle{definition}
\newaliascnt{defn}{thm}
\newaliascnt{asu}{thm}
\newaliascnt{con}{thm}
\newtheorem{defn}[defn]{Definition}

\aliascntresetthe{defn}
\aliascntresetthe{asu}
\aliascntresetthe{con}
\newcounter{stp}
\newcounter{stpi}
\newcounter{stpci}
\newcounter{stpiii}

\numberwithin{equation}{section}

\setlist[enumerate]{font = \normalfont}
\newcommand{\eps}{\varepsilon}
\newcommand{\sol}{\mathrm{sol}}
\newcommand{\test}{\mathrm{test}}

\newcommand{\rf}{\mathrm{f}}
\newcommand{\rp}{\mathrm{p}}
\newcommand{\rb}{\mathrm{b}}
\newcommand{\rfb}{\rf_\rb}

\newcommand{\mre}{\mathrm{e}}
\newcommand{\rv}{\mathrm{v}}

\newcommand{\hx}{\widehat{x}}
\newcommand{\hy}{\widehat{y}}
\newcommand{\hz}{\widehat{z}}
\newcommand{\hOm}{\widehat{\Omega}}
\newcommand{\hGam}{\widehat{\Gamma}}
\newcommand{\hOmf}{\hOm_\rf}
\newcommand{\hOmb}{\hOm_\rb}

\newcommand{\Omf}{\Omega_\rf}

\newcommand{\Omb}{\Omega_\rb}

\newcommand{\tOm}{\Tilde{\Omega}}
\newcommand{\tOmb}{\tOm_\rb}
\newcommand{\Ombd}{\Omb^\delta}
\newcommand{\Gad}{\Gamma^\delta}

\newcommand{\bfeta}{\boldsymbol{\eta}}
\newcommand{\bfsigma}{\boldsymbol{\sigma}}
\newcommand{\hbfsigma}{\widehat{\bfsigma}}
\newcommand{\bfD}{\boldsymbol{D}}
\newcommand{\bfS}{\boldsymbol{S}}
\newcommand{\bfu}{\boldsymbol{u}}
\newcommand{\bfxi}{\boldsymbol{\xi}}
\newcommand{\bfn}{\boldsymbol{n}}
\newcommand{\hbfn}{\widehat{\bfn}}
\newcommand{\bfe}{\boldsymbol{e}}
\newcommand{\bftau}{\boldsymbol{\tau}}
\newcommand{\bfw}{\boldsymbol{w}}

\newcommand{\hbftau}{\widehat{\bftau}}
\newcommand{\rhob}{\rho_\rb}
\newcommand{\bfsigmaE}{\bfsigma^{\mathrm{E}}}
\newcommand{\bfI}{\boldsymbol{I}}
\newcommand{\bfq}{\boldsymbol{q}}
\newcommand{\hbfeta}{\widehat{\bfeta}}
\newcommand{\hp}{\widehat{p}}
\newcommand{\hr}{\widehat{r}}
\newcommand{\bfPhi}{\boldsymbol{\Phi}}

\newcommand{\hbSb}{\widehat{\bfS}_\rb}
\newcommand{\hS}{\widehat{S}}
\newcommand{\hbS}{\widehat{\bfS}}
\newcommand{\lambde}{\lambda_\mre}
\newcommand{\lambdv}{\lambda_\rv}
\newcommand{\mue}{\mu_\mre}
\newcommand{\muv}{\mu_\rv}
\newcommand{\hbfD}{\widehat{\bfD}}
\newcommand{\hnabla}{\widehat{\nabla}}
\newcommand{\hbfPhi}{\widehat{\bfPhi}}
\newcommand{\hbfxi}{\widehat{\bfxi}}
\newcommand{\hzeta}{\widehat{\zeta}}
\newcommand{\bfpsi}{\boldsymbol{\psi}}
\newcommand{\hbfpsi}{\widehat{\bfpsi}}
\newcommand{\hphi}{\widehat{\varphi}}
\newcommand{\etad}{\eta^\delta}
\newcommand{\heta}{\widehat{\eta}}
\newcommand{\bfphi}{\boldsymbol{\phi}}

\newcommand{\rhop}{\rho_\rp}
\newcommand{\homega}{\widehat{\omega}}
\newcommand{\hFp}{\widehat{F}_\rp}
\newcommand{\hDelta}{\widehat{\Delta}}

\newcommand{\bfsigmaf}{\bfsigma_\rf}
\newcommand{\hbfu}{\widehat{\bfu}}
\newcommand{\hg}{\widehat{g}}
\newcommand{\hbfw}{\widehat{\bfw}}
\newcommand{\hbfv}{\widehat{\bfv}}

\newcommand{\df}{\Dot{f}}
\newcommand{\dbfu}{\Dot{\bfu}}
\newcommand{\dbfeta}{\Dot{\bfeta}}

\newcommand{\bbfeta}{\overline{\bfeta}}
\newcommand{\bp}{\overline{p}}
\newcommand{\bomega}{\overline{\omega}}
\newcommand{\bbfu}{\overline{\bfu}}

\newcommand{\tbfv}{\Tilde{\bfv}}

\newcommand{\bfvphi}{\boldsymbol{\varphi}}

\newcommand{\Vf}{V_\rf}
\newcommand{\cV}{\mathcal{V}}
\newcommand{\cVf}{\cV_\rf}
\newcommand{\Vfom}{\Vf^\omega}
\newcommand{\cVfom}{\cVf^\omega}
\newcommand{\cVom}{\cV_\omega}
\newcommand{\Vd}{V_\rd}
\newcommand{\cVb}{\cV_\rb}
\newcommand{\Vp}{V_\rp}
\newcommand{\cQ}{\mathcal{Q}}
\newcommand{\cQb}{\cQ_\rb}
\newcommand{\cVsol}{\cV_\sol}
\newcommand{\cVsolom}{\cVsol^\omega}
\newcommand{\cVtest}{\cV_\test}
\newcommand{\cVtestom}{\cVtest^\omega}

\newcommand{\bfx}{\boldsymbol{x}}
\newcommand{\bfy}{\boldsymbol{y}}
\newcommand{\hbfx}{\widehat{\bfx}}
\newcommand{\hbfy}{\widehat{\bfy}}

\newcommand{\bfv}{\boldsymbol{v}}
\newcommand{\bfA}{\boldsymbol{A}}

\newcommand{\bfG}{\boldsymbol{G}}

\newcommand{\bff}{\boldsymbol{f}}
\newcommand{\bR}{\mathbb{R}}

\newcommand{\bN}{\mathbb{N}}

\newcommand{\rd}{\mathrm{d}}

\newcommand{\srd}{\, \mathrm{d}}

\newcommand{\rc}{\mathrm{c}}
\newcommand{\rC}{\mathrm{C}}
\newcommand{\rL}{\mathrm{L}}
\newcommand{\rW}{\mathrm{W}}
\newcommand{\rH}{\mathrm{H}}

\newcommand{\rD}{\mathrm{D}}

\newcommand{\cJ}{\mathcal{J}}
\newcommand{\cX}{\mathcal{X}}

\newcommand{\hcJ}{\widehat{\cJ}}

\DeclareMathOperator{\Id}{Id}

\DeclareMathOperator{\tr}{tr}

\newcommand{\sigmad}{\sigma_\delta}

\newcommand{\dOmega}{\del \Omega}
\newcommand{\oOmega}{\overline{\Omega}}

\newcommand{\del}{\partial}
\newcommand{\dk}[1]{\partial_{#1}}
\newcommand{\dt}{\dk{t}} 
\newcommand{\dz}{\dk{z}} 
\newcommand{\tin}{\enspace \text{in} \enspace}
\newcommand{\ton}{\enspace \text{on} \enspace}
\newcommand{\tfor}{\enspace \text{for} \enspace}
\newcommand{\tforall}{\enspace \text{for all} \enspace}
\newcommand{\tand}{\enspace \text{and} \enspace}

\newcommand{\twith}{\enspace \text{with} \enspace}
\newcommand{\tso}{\enspace \text{so} \enspace}
\newcommand{\tas}{\enspace \text{as} \enspace}
\newcommand{\taswellas}{\enspace \text{as well as} \enspace}
\newcommand{\twhere}{\enspace \text{where} \enspace}

\begin{document}

\title[3D FPSI]{Three-Dimensional Navier--Stokes--Biot Coupling via a Moving Reticular Plate Interface: Existence of Weak Solutions}

\author{Felix Brandt}
\address{Department of Mathematics, University of California at Berkeley, Berkeley, 94720, CA, USA.}
\email{fbrandt@berkeley.edu}
\author{Sun\v{c}ica \v{C}ani\'c}
\address{Department of Mathematics, University of California at Berkeley, Berkeley, 94720, CA, USA.}
\email{canics@berkeley.edu}
\author{Boris Muha}
\address{Department of Mathematics, Faculty of Science, University of Zagreb, Zagreb, Croatia}
\email{borism@math.hr}

\begin{abstract}
We prove the existence of finite-energy weak solutions to a regularized three-dimensional fluid-structure interaction (FSI) problem involving an incompressible, viscous, Newtonian fluid and a multilayered poro(visco)elastic structure. The structure consists of a thick layer modeled by the Biot equations and a thin reticular plate with inertia and elastic energy, transparent to fluid flow. The coupling is nonlinear in the sense that it takes place on a moving interface that is not known a priori but is defined by the solution itself, making the problem a moving-boundary problem. This nonlinear free-boundary coupling, combined with the limited regularity of the Biot displacement, renders the classical weak formulation ill-defined at finite energy. To address this, we introduce a minimally invasive regularization based on a suitable extension and convolution of the Biot displacement, chosen so that the regularized problem remains consistent with the original model. We then construct approximate solutions to the regularized problem via a Lie operator-splitting scheme and derive uniform energy bounds. While these bounds ensure weak and weak* convergence, passing to the limit in the nonlinear terms requires refined compactness arguments, including variants of the Aubin--Lions lemma and tools adapted to moving non-Lipschitz interfaces. The result applies in particular to the purely elastic case (without structural damping) as well as the poroviscoelastic case. This work extends the two-dimensional analysis of \cite{KCM:24} to the fully three-dimensional setting and, to our knowledge, provides the first existence result for a nonlinearly coupled, multilayer 3D Navier--Stokes--Biot FSI system with a permeable interface.
\end{abstract}

\maketitle

\section{Introduction}\label{sec:intro}

Fluid-structure interaction (FSI) problems involving multilayered poroelastic structures naturally arise from a wide range of applications.
We refer here e.g.\ to encapsulation of bioartificial organs \cite{WCBBR:22} or blood flow in arteries \cite{BQQ:09, BZY:15, CWB:21} modeled as poro(visco)elastic media to investigate drug transport through vascular walls.

In this article, we prove the existence of a finite-energy weak solution to a regularized, nonlinearly coupled fluid--poro(visco)elastic structure interaction (FPSI) problem.
More precisely, we study the interaction problem of an incompressible, viscous, Newtonian fluid, represented by the Navier--Stokes equations, as well as bulk poro(visco)elasticity, given by the Biot equations for porous media.
Here, bulk poro(visco)elasticity refers to the fact that the fluid and poro(visco)elastic domains have the same dimension--both are three-dimensional in this work.
We further assume that the free fluid flow and the Biot poro(visco)elastic medium are coupled across a moving interface.
This interface is modeled by a reticular plate with inertia and elastic energy that is transparent to fluid flow.
The geometric setup is sketched in \autoref{fig:3d_fsi_flat}.

A similar problem, though with two-dimensional fluid and Biot domains and a one-dimensional interface, has been analyzed recently in \cite{KCM:24}.
In the present manuscript, we extend the existence result of a finite-energy weak solution of Leray--Hopf type to the 3D case.
Note that the three-dimensional character of the present problem has a clear meaning in biomedical applications \cite{MC:13b}, and the extension of the result from \cite{KCM:24} to the 3D case requires some non-trivial modifications, as outlined below after the literature review.

The field of FSI involving purely elastic structures located at a part of the fluid boundary has been a very active one in the last 20 years.
The linearly coupled case has been investigated in the situation of linear \cite{DGHL:03} and nonlinear models \cite{BGLT:07, KTZ:10}.
Strong solutions to FSI problems with the domain of the structure displacement being of a lower dimension than the fluid domain were first studied in \cite{BdV:04}, and the analysis was later extended e.g.\ in \cite{Leq:11, GH:16, MRR:20, DS:20, MT:21}.
For investigations of the situation of the fluid and the structure domain being of the same dimension, we refer for example to \cite{CS:05, CCS:07, KT:12, IKLT:14, RV:14}.
The existence of a strong time-periodic solution to an FSI problem with a multilayered structure has been analyzed recently in \cite{BMR:25}.

The study of weak solutions to nonlinearly coupled FSI problems was pioneered in \cite{CDEG:05} and further developed e.g.\ in \cite{Gra:08, MC:13a, LR:14}.
In particular, in \cite{MC:13a}, a constructive approach to FSI problems based on an operator splitting scheme first introduced in \cite{GGCC:09} was used to prove the existence of a finite-energy weak solution to an FSI problem between a 1D elastic Koiter shell and a 2D incompressible, viscous, Newtonian fluid. 
The underlying idea is to semidiscretize the full nonlinearly coupled problem in time, and to reduce the analysis to fluid and structure subproblems.
This scheme has been successfully applied to different settings including the situation of 3D-2D-FSI problems \cite{MC:13b}, multilayered structures \cite{MC:14} or elastic shells with nonlinear Koiter energy \cite{MC:15}.
Our approach likewise employs time discretization combined with Lie operator splitting to construct approximate solutions to the 3D FPSI problem.

Biot \cite{Biot:41, Biot:55} introduced the equations modeling poroelastic media in the context of soil consolidation, and they are referred to as Biot equations nowadays.
Concerning their mathematical analysis, we mention for instance \cite{Sho:00, BGSW:16, BMW:23}.
Linearly coupled FSI problems involving poroelastic structures were e.g.\ considered in \cite{Sho:05, Ces:17, AGW:25, BHR:25,AW:25}.
In \cite{BCMW:21}, a linearly coupled multilayered FPSI problem of a Stokes flow with a poroelastic plate as well as a poroelastic Biot medium was studied.
The aforementioned work \cite{KCM:24} addresses the nonlinearly coupled FPSI problem of a Navier--Stokes fluid with a multilayered structure consisting of a thin reticular plate and a thick Biot layer in the 2D-1D-setting.

The main mathematical difficulties in the present work stem from the nonlinear coupling of the Navier--Stokes equations with the bulk poro(visco)elasticity.
Due to the same dimension of the thick structural layer and the fluid domain, the regularity of the finite-energy weak solutions is insufficient for defining the moving domain and associated traces. In addition, geometric nonlinearities from the moving domain mean that certain integrals in the weak formulation are not well-defined. To address this, we introduce a minimally invasive regularization: a suitable extension and convolution of the Biot displacement that renders all integrals well-defined.
We then show the existence of a finite-energy weak solution to the emerging regularized FPSI problem.

The next step in this work will be to prove that the regularized problem studied here is consistent with the original problem in the sense that its solution converges to the classical solution of the non-regularized problem when the regularization parameter tends to zero. 
Note that the approach from \cite{KCM:24} to prove this type of a weak-classical consistency result cannot be used in the present three-dimensional setting, since the convergence rate of the convolution kernel is not compatible with the loss of regularity following from the odd extension procedure.
In a forthcoming work, a conditional weak-classical consistency result with minimal requirements on the underlying extension and regularization procedure will be provided.

Compared to \cite{KCM:24}, there are additional difficulties arising from the three-dimensional character of the FPSI problem.
In fact, by Sobolev embeddings, the plate displacements in the finite-energy space are no longer Lipschitz continuous, see \eqref{eq:reg of omega in fin energy space}.
The consequences of this observation are twofold:
First, the fluid domain can no longer be represented as the subgraph of a Lipschitz function, leading to complications in the definition of the solution and test spaces.
As a remedy, we adapt the strategy from \cite{MC:13b, MC:15} in the case of 3D FSI problems and invoke the so-called Lagrangian trace, see also \cite{CDEG:05, Gra:08,LR:14, Muh:14}.
Second, we face a loss of regularity when switching between the moving and fixed fluid domain configurations.
This manifests itself in the uniform boundedness estimates of the approximate fluid velocities, providing estimates in Sobolev spaces with lower integrability, which in turn complicates the subsequent compactness arguments.

Here is a summary of the methodology. We prove existence constructively. After specifying the regularization, we time-discretize $[0,T]$ into $N$ subintervals of length $\Delta t = \frac{T}{N}$ and apply Lie splitting to decouple, at each step, a reticular-plate subproblem from the Biot--fluid subproblem. Solving these subproblems sequentially yields the approximate solutions.

We then derive energy estimates uniform in the time step, giving uniform bounds for the approximate solutions.
At this stage, the loss of regularity due to the lower regularity of the plate displacements enters, and it is only possible to show the uniform boundedness of the approximate fluid velocities in $\rL^2(0,T;\rW^{1,p}(\Omf)^3)$ for $p < 2$, where $\Omf$ denotes the fluid domain that will be made precise in the following section.
Because of the strong coupling, the approaches in \cite[Lemma~6]{CDEG:05} and \cite[Proposition~6.5]{MC:14} to obtain a Korn equality are not applicable, and the non-Lipschitz geometry rules out the Korn inequality of \cite[Lemma~1]{Vel:12}. We therefore use a Korn-type inequality for non-Lipschitz domains \cite[Proposition~2.9]{Len:14}, at the cost of the aforementioned loss of integrability.

The highly nonlinear nature of the present 3D FPSI problem implies that the weak and weak* convergences resulting from the uniform boundedness of the approximate solutions are insufficient for the limit passage.
We invoke suitable compactness criteria to upgrade to strong convergences.
More precisely, we use the classical Aubin--Lions compactness result for the Biot displacement, the Arzela--Ascoli theorem for the plate displacement, the compactness result of Dreher and J\"ungel \cite{DJ:12} for piecewise constant functions in the context of the Biot and plate velocity as well as a generalized Aubin--Lions--Simon compactness result due to Muha and \v{C}ani\'c \cite{MC:19} for the fluid velocity.
The loss of regularity of the plate displacements compared to the 2D case significantly complicates the proof of the compactness of the fluid velocities.

Finally, we pass to the limit in the semidiscretized problem.
Here we first need to show the convergence of the gradients of the fluid velocities, which is due to the loss of regularity compared to the 2D case.
Moreover, we need to construct suitable test functions and discuss their convergence.
The limit passage is based on the above weak, weak* and strong convergences.
We observe that the lack of Lipschitz regularity of the plate displacements requires considerable adjustments in the limit passage.

Note that the local-in-time character of the solution results from the possibility of degeneration of the fluid domain.
To address the maximal time of existence, the procedure from \cite[Section~5]{CDEG:05} can be adopted to the present setting.
It yields that the maximal time of existence can be extended up to the occurrence of one of the following events:
the moving fluid or Biot domain degenerates, the Lagrangian mapping is no longer injective, or the solution exists globally, i.e., $T = \infty$.

    This work presents the first existence result for a nonlinearly coupled, multilayer 3D Navier-Stokes--Biot fluid--structure interaction (FSI) system with a permeable interface.

    The remainder of the article is organized as follows. In \autoref{sec:FPSI problem}, we introduce the multilayered FPSI problem consisting of the Biot subproblem, the reticular plate subproblem, the fluid subproblem as well as suitable coupling and boundary conditions.
\autoref{sec:def of a weak sol} is dedicated to the introduction of the maps between the reference and physical domains as well as the formal derivation of a weak formulation.
In \autoref{sec:reg probl, weak sol & ex result}, we formulate the regularized problem and state the main existence result for a finite-energy weak solution.
In \autoref{sec:splitting scheme}, we present the splitting scheme comprising the plate subproblem as well as the fluid and Biot subproblem.
We also make precise the coupled semidiscrete problem and establish discrete energy equalities.
The construction of the approximate solutions, the proof of their uniform boundedness and the derivation of weak and weak* convergences from there is done in \autoref{sec:approx sols on the complete time int}.
In \autoref{sec:comp args}, we perform the compactness arguments for the Biot displacement, the plate displacement, the Biot and plate velocity, the Biot pore pressure and the fluid velocity.
\autoref{sec:limit passage} is concerned with the limit passage, for which the strong convergence of the fluid velocity traces, the convergence of the test functions on approximate fluid domains and the convergence of the gradients of the fluid velocity are required.
In the final \autoref{sec:conclusions}, we provide a summary and give a brief outlook on possible future work.

\section{The fluid-poro(visco)elastic structure interaction problem}\label{sec:FPSI problem}

We consider a 3D FPSI problem on the time-dependent domain $\Omega(t) = \Omega_\rb(t) \cup \Gamma(t) \cup \Omega_\rf(t)$ (\autoref{fig:3d_fsi_flat}).
The 3D reference configuration is $\hOm= \widehat{\Omega}_\rb \cup \widehat{\Gamma} \cup \widehat{\Omega}_\rf$, where
\begin{equation*}
    \begin{aligned}
      \hOmb = (0,L)^2 \times (0,R), \enspace \hGam = (0,L)^2 \times \{0\} \tand \hOmf = (0,L)^2 \times (-R,0).
    \end{aligned}
\end{equation*}
The mathematical models defined on each of the moving sub-domains are described as follows.

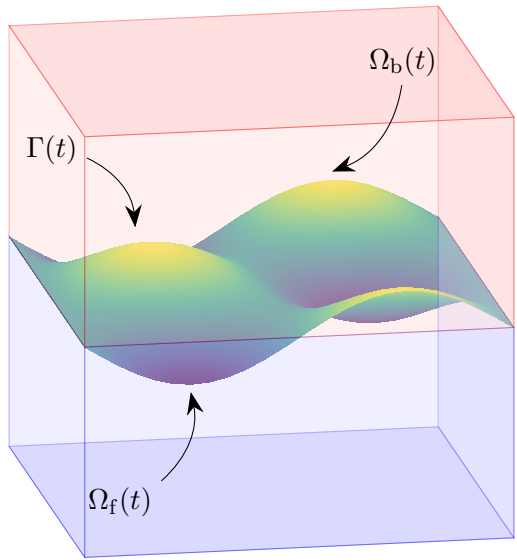
\begin{figure}[h!]
\centering
\begin{tikzpicture}
\begin{axis}[
axis lines=none, 
xlabel=$x$, ylabel=$z$, zlabel=$y$,
xmin=-1.5, xmax=1.5,
ymin=-1.5, ymax=1.5,
zmin=-1.2, zmax=1.2,
view={170}{15}, 
ticks=none,
enlargelimits=false,
axis equal,
width=0.8\textwidth,
height=0.7\textwidth,
]

\def\Zmax{1}
\def\Zmin{-1}
\def\Xmax{1}
\def\Xmin{-1}
\def\Ymax{1}
\def\Ymin{-1}

\addplot3[
surf,
shader=interp,
colormap/viridis,
opacity=0.8, 
domain=\Xmin:\Xmax,
y domain=\Ymin:\Ymax,
samples=25,
variable=\u,
variable y=\v,
] (u, v, {0.2 * sin(deg(u*pi)) * cos(deg(v*pi))});

\draw[blue, fill=blue!30, opacity=0.4] (axis cs:\Xmin, \Ymin, \Zmin) -- (axis cs:\Xmax, \Ymin, \Zmin) -- (axis cs:\Xmax, \Ymax, \Zmin) -- (axis cs:\Xmin, \Ymax, \Zmin) -- cycle; 
\draw[blue, fill=blue!20, opacity=0.3] (axis cs:\Xmin, \Ymin, \Zmin) -- (axis cs:\Xmin, \Ymax, \Zmin) -- (axis cs:\Xmin, \Ymax, 0) -- (axis cs:\Xmin, \Ymin, 0) -- cycle; 
\draw[blue, fill=blue!20, opacity=0.3] (axis cs:\Xmax, \Ymin, \Zmin) -- (axis cs:\Xmax, \Ymax, \Zmin) -- (axis cs:\Xmax, \Ymax, 0) -- (axis cs:\Xmax, \Ymin, 0) -- cycle; 
\draw[blue, fill=blue!20, opacity=0.3] (axis cs:\Xmin, \Ymax, \Zmin) -- (axis cs:\Xmax, \Ymax, \Zmin) -- (axis cs:\Xmax, \Ymax, 0) -- (axis cs:\Xmin, \Ymax, 0) -- cycle; 

\draw[red, fill=red!30, opacity=0.4] (axis cs:\Xmin, \Ymin, \Zmax) -- (axis cs:\Xmax, \Ymin, \Zmax) -- (axis cs:\Xmax, \Ymax, \Zmax) -- (axis cs:\Xmin, \Ymax, \Zmax) -- cycle; 
\draw[red, fill=red!20, opacity=0.3] (axis cs:\Xmin, \Ymin, \Zmax) -- (axis cs:\Xmin, \Ymax, \Zmax) -- (axis cs:\Xmin, \Ymax, 0) -- (axis cs:\Xmin, \Ymin, 0) -- cycle; 
\draw[red, fill=red!20, opacity=0.3] (axis cs:\Xmax, \Ymin, \Zmax) -- (axis cs:\Xmax, \Ymax, \Zmax) -- (axis cs:\Xmax, \Ymax, 0) -- (axis cs:\Xmax, \Ymin, 0) -- cycle; 
\draw[red, fill=red!20, opacity=0.3] (axis cs:\Xmin, \Ymax, \Zmax) -- (axis cs:\Xmax, \Ymax, \Zmax) -- (axis cs:\Xmax, \Ymax, 0) -- (axis cs:\Xmin, \Ymax, 0) -- cycle;

\node (gamma_label) at (axis cs: 0.8, -1, 0.4) {\large $\Gamma(t)$};
\draw[-{Stealth[length=3mm]}] (gamma_label) to[bend left] (axis cs:0.5, -0.5, 0.15);

\node (fluid_label) at (axis cs: 0.8, 0.8, -0.8) {\large $\Omega_\rf(t)$};
\draw[-{Stealth[length=3mm]}] (fluid_label) to[bend right] (axis cs:0.4, 0.4, -0.4);

\node (biot_label) at (axis cs: -0.8, -0.8, 0.8) {\large $\Omega_\rb(t)$};
\draw[-{Stealth[length=3mm]}] (biot_label) to[bend left] (axis cs:-0.4, -0.4, 0.4);

\end{axis}
\end{tikzpicture}
\caption{
A sketch of the $3D$ fluid domain $\Omega_\rf(t)$, $3D$ Biot domain $\Omega_\rb(t)$, and the interface $\Gamma(t)$ modeled by the reticular plate equation.}
\label{fig:3d_fsi_flat}
\end{figure}

\subsection{The Biot subproblem and the Lagrangian mapping}\label{ssec:Biot subproblem}
We begin by introducing the governing equations defined on a fixed domain $\Omb$, originally formulated by M.A.~Biot in \cite{Biot:41, Biot:55}, and then extend this formulation to the case of time-dependent, moving domains.

For a positive time $T > 0$, the model variables are the displacement $\bfeta \colon [0,T] \times \Omb \to \bR^3$ of a poroelastic material from its reference configuration $\Omb$ and the fluid pore pressure $p \colon [0,T] \times \Omb \to \bR$.
The system of coupled PDEs then reads as
\begin{equation}\label{eq:Biot eqs on fixed dom}
    \left\{
    \begin{aligned}
        \rhob \del_{tt} \bfeta - \nabla \cdot \bfsigma_\rb(\nabla \bfeta,p)
        &= 0, &&\tin [0,T] \times \Omb,\\
        c_0 \dt p + \alpha (\nabla \cdot \bfeta)_t -\nabla \cdot (\kappa \nabla p)
        &= 0, &&\tin [0,T] \times \Omb.
    \end{aligned}
    \right.
\end{equation}
Here, $\rhob$ is the density of the poro(visco)elastic matrix, $c_0$ is the storage coefficient, $\alpha$ is the Biot-Willis coefficient, and $\kappa$ is the permeability of the poro(visco)elastic material, all strictly greater than zero.

In \eqref{eq:Biot eqs on fixed dom}$_1$, $\bfsigma_\rb(\nabla \bfeta,p)$ represents the stress tensor, which is given by
\begin{equation}\label{sigma_fixed}
    \bfsigma_\rb(\nabla \bfeta,p) = \bfsigmaE(\nabla \bfeta,p) - \alpha p \bfI,
\end{equation}
where $\bfsigmaE(\nabla \bfeta,p)$ denotes the elastic part. Details about $\bfsigmaE(\nabla \bfeta,p)$ will be discussed below. 

The first equation \eqref{eq:Biot eqs on fixed dom}$_1$ models the elastodynamics of the poroelastic material.
Note that the effect of the pore pressure enters via the term $\alpha p \bfI$ in the stress tensor.
On the other hand, \eqref{eq:Biot eqs on fixed dom}$_2$ models the conservation of fluid mass.
The term $\kappa \nabla p$ in \eqref{eq:Biot eqs on fixed dom}$_2$ is related to the filtration velocity $\bfq$ via Darcy's law
\begin{equation*}
    \bfq = -\kappa \nabla p, \tfor \text{$\kappa > 0$ constant}. 
\end{equation*}

To study the Biot equations defined on a {\em moving domain}, we will consider 
the elastodynamics equation~\eqref{eq:Biot eqs on fixed dom}$_1$ defined in Lagrangian  formulation on the reference domain $\hOmb$, while \eqref{eq:Biot eqs on fixed dom}$_2$ will be considered on the Eulerian domain $\Omb(t)$.
We introduce the following {\em Lagrangian mapping} between the two domains 
\begin{equation}\label{eq:Lagrangian map}
    \bfPhi_\rb^\eta(t,\cdot) = \Id + \hbfeta(t,\cdot) \colon \hOmb \to \Omb(t), 
\end{equation}
where $\hbfeta \colon [0,T] \times \hOmb \to \bR^3$ denotes the displacement of the Biot poro(visco)elastic matrix, defined on the reference domain. The inverse of the Lagrangian map will be denoted by
$(\bfPhi_\rb^\eta)^{-1}(t,\cdot) \colon \Omb(t) \to \hOmb$.

Using the Lagrangian mapping we can define the stress tensor \eqref{sigma_fixed}, now denoted by $\widehat\bfsigma_\rb(\widehat\nabla \widehat\bfeta,\widehat{p})$, as
\begin{equation*}
\widehat\bfsigma_\rb(\widehat\nabla \widehat\bfeta,\widehat{p}) :=\hbSb(\hnabla \hbfeta,\hp)- \alpha \det\bigl(\hnabla \hbfPhi_\rb^\eta\bigr) \hp (\hnabla \hbfPhi_\rb^\eta)^{-\top}, 
\end{equation*}
where 
\begin{equation*}
    \hbSb(\hnabla \hbfeta,\hp) = 2 \mue \hbfD(\hbfeta) + \lambde(\hnabla \cdot \hbfeta) \bfI + 2 \muv \hbfD(\hbfeta_t) + \lambdv(\hnabla \cdot \hbfeta_t) \bfI
\end{equation*}
is the Piola--Kirchhoff stress tensor.
Here $\bfD$ represents the symmetric part of the gradient, the notation~$\hbfD$ and $\hnabla$ indicates that the respective differential operators act on the Lagrangian domain $\hOmb$, $\lambde$, $\mue$ and $\lambdv$, $\muv$ are Lam\'e coefficients associated to the elastic and viscoelastic stress, respectively, i.e., $\lambdv = \muv = 0$ corresponds to the purely elastic case. 

For later reference, we note that on the moving domain, the Piola--Kirchhoff stress tensor can be obtained via the Piola transform as follows:
\begin{equation*}
    \begin{aligned}
        \bfS_\rb(\nabla \bfeta,p)
        &= \bigl[\det(\hnabla \hbfPhi_\rb^\eta)^{-1} \hbS_\rb(\hnabla \hbfeta,\hp) (\hnabla \hbfPhi_\rb^\eta)^\top\bigr] \circ (\bfPhi_\rb^\eta)^{-1}\\
        &= \left(\frac{1}{\det(\hnabla \hbfPhi_\rb^\eta)}\bigl[ 2 \mue \hbfD(\hbfeta) + \lambde(\hnabla \cdot \hbfeta) + 2 \muv \hbfD(\hbfeta_t) + \lambdv (\hnabla \cdot \hbfeta_t) \bigr] (\hnabla \hbfPhi_\rb^\eta)^\top\right) \circ (\bfPhi_\rb^\eta)^{-1}.
    \end{aligned}
\end{equation*}

Recalling the material derivative
\begin{equation*}
    \frac{\rD}{\rD t} = \frac{\rd}{\rd t} + \bigl((\dt \bfeta(t,\cdot) \circ (\bfPhi_\rb^\eta)^{-1}(t,\cdot)) \cdot \nabla\bigr),
\end{equation*}
we can now express the Biot equations on the moving domain as follows:
\begin{equation}\label{eq:Biot eqs on moving dom}
    \left\{
    \begin{aligned}
        \rhob \del_{tt} \hbfeta
        &= \hnabla \cdot \widehat{\boldsymbol{\sigma}}_\rb(\hnabla \hbfeta,\hp), &&\tin [0,T] \times \hOmb,\\
        \frac{c_0}{[\det(\hnabla \hbfPhi_\rb^\eta)] \circ (\hbfPhi_\rb^\eta)^{-1}} \frac{\rD}{\rD t} p + \alpha \nabla \cdot \frac{\rD}{\rD t} \bfeta
        &= \nabla \cdot (\kappa \nabla p), &&\tin [0,T] \times \Omb(t).
    \end{aligned}
    \right.
\end{equation}
Here, the Biot material displacement $\bfeta$ and the fluid pore pressure $p$ on the physical domain $\Omb(t)$ in~\eqref{eq:Biot eqs on moving dom}$_2$
are given by
\begin{equation*}
    \bfeta(t,\cdot) = \hbfeta\bigl(t,(\hbfPhi_\rb^\eta)^{-1}(t,\cdot)\bigr), \enspace p(t,\cdot) = \hp\bigl(t,(\hbfPhi_\rb^\eta)^{-1}(t,\cdot)\bigr), \twhere \Omb(t) = \hbfPhi_\rb^\eta(t,\hOmb).
\end{equation*}

At this stage, we observe that there is a subtlety in the definition of the moving domain $\Omb(t)$, namely it is a priori not clear that it is well-defined unless $\hbfeta$ possesses enough regularity, and the map $\hbfPhi_\rb^\eta = \Id + \hbfeta$ is injective as a map from $\hOmb$ to $\Omb(t)$.
We shall address this aspect at a later stage, namely in the context of the construction of the approximate solutions in \autoref{sec:approx sols on the complete time int}.

\subsection{The reticular plate subproblem}\label{ssec:plate subproblem}
Reticular plates are thin elastic plates with holes/periodic cells distributed in all directions \cite{CSJP:99}. The coefficients in the reticular plate equation depend on the periodic structure. 
As in \cite{BCMW:21}, we restrict ourselves to the transverse plate displacement
$\homega \colon [0,T] \times \hGam \to \bR$
under the assumption that the stress in the transverse direction dominates over the longitudinal and lateral components. Using $\rhop > 0$ to denote reticular plate density, and $\hDelta^2$ for the in-plane bi-Laplacian, the reticular plate equation can be written as
\begin{equation}\label{eq:plate eq}
    \rhop \del_{tt} \homega + \hDelta^2 \homega = \hFp, \ton [0,T] \times \hGam,
\end{equation}
where $\hFp$ denotes the source term, and $\hGam$ denotes the reference configuration of the middle surface of the plate. Without loss of generality, we took the coefficient in front of the bi-Laplacian to be equal to 1. 

We denote the time-dependent configuration of the plate by
\begin{equation*}
    \Gamma(t) = \{(x,y,z) \in \bR^3 : 0 < x,\, y < L \tand z = \homega(t,x,y) \}.
\end{equation*}
It represents the bottom boundary of the moving Biot domain $\Omb(t)$.
We assume that the lateral boundaries and the top boundary of the Biot domain are independent of time,
namely,  $\bfeta = 0$ on these boundaries.
This also allows us to write the moving Biot domain $\Omb(t)$ as
\begin{equation*}
    \Omb(t) = \{(x,y,z) \in \bR^3 : 0 < x,\, y < L \tand \homega(t,x,y) < z < R \}.
\end{equation*}

\subsection{The fluid subproblem and ALE mapping}\label{ssec:fluid subproblem}
We will be assuming that the fluid occupying $\Omega_\rf(t)$ is viscous and incompressible, and that the flow is governed  by the Navier--Stokes equations
\begin{equation}\label{eq:fluid eqs on moving dom}
    \left\{
    \begin{aligned}
        \dt \bfu + (\bfu \cdot \nabla) \bfu
        &= \nabla \cdot \bfsigmaf(\nabla \bfu,\pi), &&\tin [0,T] \times \Omf(t),\\
        \nabla \cdot \bfu
        &= 0, &&\tin [0,T] \times \Omf(t),
    \end{aligned}
    \right.
\end{equation}
where $\bfu \colon [0,T] \times \Omf(t) \to \bR^3$ denotes the fluid velocity,
and $\pi \colon [0,T] \times \Omf(t) \to \bR$ the pressure.
We further assume the fluid is Newtonian, so that the Cauchy stress tensor takes the form
\begin{equation*}
    \bfsigmaf(\nabla \bfu,\pi) = 2 \nu \bfD(\bfu) - \pi \bfI,
\end{equation*}
where $\nu > 0$ is the kinematic viscosity coefficient.

As in the case of the moving Biot domain, the fluid domain $\Omf(t)$ is not known a priori --  it can be expressed in terms of the plate displacement $\homega$ as follows:
\begin{equation*}
    \Omf(t) = \{(x,y,z) \in \bR^3 : 0 < x,\, y < L \tand -R < z < \homega(t,x,y) \}.
\end{equation*}

As we will see below, it is convenient to express the weak formulation of the coupled FPSI problem on a fixed reference domain. To that end, we introduce the following Arbitrary Lagrangian-Eulerian (ALE) mapping
$\hbfPhi_\rf^\omega \colon \hOmf \to \Omf(t)$, which maps the reference domain $\hOmf$ onto the time-dependent physical domain~$\Omf(t)$:
\begin{equation}\label{eq:ALE mapping}
    \hbfPhi_\rf^\omega(\hx,\hy,\hz) = \Bigl(\hx,\hy,\hz + \Bigl(1 + \frac{\hz}{R}\Bigr) \homega\Bigr), \tfor (\hx,\hy,\hz) \in \hOmf.
\end{equation}

\subsection{Coupling conditions}\label{ssec:coupling conds}

The coupling between the Biot subproblem  \eqref{eq:Biot eqs on moving dom}, the reticular plate subproblem \eqref{eq:plate eq} and the fluid subproblem \eqref{eq:fluid eqs on moving dom} occurs across the moving reticular plate interface $\Gamma(t)$ and is governed by two sets of coupling conditions: kinematic and dynamic.

To state the coupling conditions we introduce the following two sets of notations:
\begin{itemize}
\item  The Eulerian structure velocity of the Biot poro(visco)elastic matrix defined at each point of the physical domain $\Omb(t)$:
\begin{equation*}
    \bfxi(t,\cdot) = \dt \hbfeta\bigl(t,(\bfPhi_\rb^\eta)^{-1}(t,\cdot)\bigr);
    \end{equation*}
\item Normal unit vectors:  $\bfn(t)$ denotes the outward unit vector to the boundary of the fluid domain~$\Omega_\rf(t)$ at $\Gamma(t)$, and $\hbfn$ denotes the normal unit vector to the reference fluid domain $\widehat\Omega_\rf$ at~$\hGam$.
\end{itemize}

\noindent{\bf{The kinematic coupling conditions.}}
We assume the following three kinematic coupling conditions at the reticular plate interface:
\begin{enumerate}[label=\textbf{K\arabic*}:]
\item Continuity of normal components of the velocity, or penetration condition, describing conservation of mass of the fluid:
\begin{equation}\label{eq:cont of normal comps of vel}
    \bfu \cdot \bfn(t) = (\bfq + \bfxi) \cdot \bfn(t), \ton (0,T) \times \Gamma(t).
\end{equation}
\item The Beavers--Joseph--Saffman condition capturing 
slip in the tangential component of the fluid velocity:
\begin{equation}\label{eq:Beavers--Joseph--Saffman}
    \beta (\bfxi - \bfu) \cdot \bftau_i(t) = \bfsigmaf \bfn(t) \cdot \bftau_i(t), \ton (0,T) \times \Gamma(t), \tfor i=1,2,
\end{equation}
where $\beta \ge 0$ is a constant, and $\bftau_i(t)$, $i=1,2$, represent the unit tangent vectors to $\Gamma(t)$.
\item The continuity of displacements:
\begin{equation}\label{eq:cont of displacement}
    \hbfeta = \homega \bfe_3, \ton (0,T) \times \hGam.
\end{equation}
\end{enumerate}

{\bf{The dynamic coupling conditions.}}
The following two dynamic coupling conditions are assumed at the reticular plate interface:
\begin{enumerate}[label=\textbf{D\arabic*}:]
\item The interface elastodynamics is governed by Newton's second law of motion, where the net force~$\hFp$ acting on the interface is given by the jump in the normal component of the stress across the interface -- with contributions from the fluid on one side and the Biot medium on the other:
\begin{equation}\label{eq:bal of forces on the plate}
    \hFp = -\det\bigl(\nabla \hbfPhi_\rf^\omega\bigr)\bigl[\bfsigmaf(\nabla \bfu,\pi) \circ \hbfPhi_\rf^\omega\bigr]\bigl(\nabla \hbfPhi_\rf^\omega\bigr)^{-\top} \bfe_3 \cdot \bfe_3 + \left.\hbfsigma_\rb(\hnabla \hbfeta,\hp) \bfe_3 \cdot \bfe_3 \right|_{\hGam}, \ton (0,T) \times \hGam.
\end{equation}
\item Continuity of pressure at the interface, i.e., we assume that
\begin{equation}\label{eq:bal of pressure}
    -\bfsigmaf(\nabla \bfu,\pi) \bfn(t) \cdot \bfn(t) + \frac{1}{2} |\bfu|^2 = p, \ton (0,T) \times \Gamma(t).
\end{equation}
This condition implies that the reticular plate behaves as a permeable interface, allowing fluid to pass without resistance.
\end{enumerate}

\subsection{The boundary and initial conditions}\label{ssec:bdry and init conds}

The FPSI problem studied in this work is supplemented by appropriate boundary and initial conditions.

For the fluid, we assume that the boundary $\partial \Omf(t) \setminus \Gamma(t)$ consists of rigid walls, and thus impose a no-slip condition on the velocity field $\bfu$:
\begin{equation}\label{eq:no-slip fluid}
\bfu = 0, \quad \text{on } (0,T) \times \partial\Omf(t) \setminus \Gamma(t).
\end{equation}

Similarly, we assume that the boundary of the Biot poro(visco)elastic domain, excluding the interface~$\Gamma(t)$, is rigid. Consequently, we enforce no-slip boundary conditions for both the solid displacement~$\hbfeta$ and the pore pressure $\hp$:
\begin{equation}\label{eq:no-slip dispalcement & pressure}
\hbfeta = 0 \quad \text{and} \quad \hp = 0, \quad \text{on } (0,T) \times \partial \hOmb \setminus \hGam.
\end{equation}

The initial conditions for the coupled system are given by:
\begin{equation}\label{eq:init conds}
\bfu(0) = \bfu_0 \text{ in } \Omf(0),\
\hbfeta(0) = \hbfeta_0, \, \partial_t \hbfeta(0) = \hbfxi_0 \text{ in } \hOmb,\
\homega(0) = \homega_0, \, \partial_t \homega(0) = \hzeta_0 \text{ in } \hGam,\
\hp(0) = \hp_0 \text{ in } \hOmb.
\end{equation}

In summary, the 3D FPSI problem under consideration consists of the Biot problem \eqref{eq:Biot eqs on moving dom}, the reticular plate problem \eqref{eq:plate eq}, the fluid problem \eqref{eq:fluid eqs on moving dom}, the kinematic coupling conditions \eqref{eq:cont of normal comps of vel}--\eqref{eq:cont of displacement}, the dynamic coupling conditions \eqref{eq:bal of forces on the plate}--\eqref{eq:bal of pressure}, the boundary conditions \eqref{eq:no-slip fluid}--\eqref{eq:no-slip dispalcement & pressure} and the initial conditions~\eqref{eq:init conds}.
In the table below, we collect the constants and coefficients appearing in the model described above:
\begin{equation*}
    \begin{aligned}
        &\textbf{Coefficients} &&\textbf{Name}\\
        &\rhob &&\text{Density of poro(visco)elastic matrix;}\\
        &c_0 &&\text{Storage coefficient;}\\
        &\alpha &&\text{Biot-Willis coefficient;}\\
        &\kappa &&\text{Permeability of the poro(visco)elastic matrix;}\\
        &\mue, \lambde, \muv, \lambdv &&\text{Lam\'e constants accounting for elasticity and viscosity;}\\
        &\rhop &&\text{Plate density coefficient;}\\
        &\nu &&\text{Kinematic viscosity coefficient;}\\
        &\beta &&\text{Constant in the context of the Beavers--Joseph--Saffman coupling condition.}
    \end{aligned}
\end{equation*}

\section{Definition of a weak solution}\label{sec:def of a weak sol}

After formulating the nonlinearly coupled 3D FPSI problem in \autoref{sec:FPSI problem}, we turn to the notion of a weak solution.
Due to the nonlinear coupling, both the fluid domain $\Omf(t)$ and the Biot domain $\Omb(t)$ depend on time and are therefore not known a priori in the physical setting. To address this, we introduce mappings from the respective reference domains $\hOmf$, $\hGam$, and $\hOmb$ to the corresponding time-dependent physical domains.

At this point, it is important to highlight that the 3D nature of the problem introduces additional analytical difficulties compared to the 2D case treated in \cite{KCM:24}, primarily due to the weaker Sobolev embeddings in three dimensions. As shown in \eqref{eq:reg of omega in fin energy space}, the plate displacement $\homega$ lies in a finite-energy space whose spatial regularity is limited to $\rC^{0,r}$ for any $r < 1$. Consequently, the fluid domain $\Omf(t)$ is, in general, not a Lipschitz domain, and the associated ALE mapping is not necessarily Lipschitz continuous either.

Nevertheless, the $r$-H\"{o}lder continuity of $\homega$ ensures that $\Omf(t)$ can locally be represented as the subgraph of a $\rC^{0,r}$ function. This regularity is sufficient to define traces on $\Omf(t)$, as established in \cite{CDEG:05, Gra:08, Muh:14}.

For a treatment of 3D nonlinearly coupled FSI problems without poroelasticity, we refer to the works~\cite{MC:13b, MC:15}.
We will elaborate more on the difficulties described above in this section as well as in the subsequent section.

\subsection{Maps from the reference to the physical domains}\label{ssec:maps from ref to phys dom}

We start by recalling the mappings defined in  \eqref{eq:Lagrangian map} and \eqref{eq:ALE mapping}, involving the fluid and Biot domains:
\begin{equation*}
    \hbfPhi_\rb^\eta(t,\cdot) \colon \hOmb \to \Omb(t) \tand \hbfPhi_\rf^\omega(t,\cdot) \colon \hOmf \to \Omf(t).
\end{equation*}
Additionally, we introduce the map $\hbfPhi_\Gamma^\omega(t,\cdot) \colon \hGam \to \Gamma(t)$ accounting for the fluid-structure interface, and summarize the three mappings that will be important in this work:
\begin{equation}\label{eq:trafos to phys doms}
    \begin{aligned}
        \hbfPhi_\rb^\eta(\hx,\hy,\hz)
        &= \Id + \hbfeta(\hx,\hy,\hz), &&\tfor (\hx,\hy,\hz) \in \hOmb,\\
        \hbfPhi_\Gamma^\omega(\hx,\hy,0)
        &= (\hx,\hy,\homega(\hx,\hy)), &&\tfor (\hx,\hy) \in \hGam,\\
        \hbfPhi_\rf^\omega(\hx,\hy,\hz)
        &= \Bigl(\hx,\hy,\hz + \Bigl(1 + \frac{\hz}{R}\Bigr) \homega\Bigr), &&\tfor (\hx,\hy,\hz) \in \hOmf.
    \end{aligned}
\end{equation}
The inverse of the ALE mapping in \eqref{eq:trafos to phys doms}$_3$ is given by
\begin{equation*}
    (\hbfPhi_\rf^\omega)^{-1}(x,y,z) = \Bigl(x,y,-R + \frac{R}{R + \homega}(R + z)\Bigr).
\end{equation*}
Concerning the regularity of the ALE mapping \eqref{eq:trafos to phys doms}$_3$, we make the following observation.
As also revealed later on, the natural space for the plate displacement $\homega$ in the class of finite energy is 
\begin{equation*}
    \cVom \coloneqq \rW^{1,\infty}(0,T;\rL^2(\hGam)) \cap \rL^\infty(0,T;\rH_0^2(\hGam)). 
\end{equation*}
From a mixed derivative-type embedding, see \cite[Lemma~4]{CKM:22} or also \cite[Section~1.2]{Gra:08}, it then follows that 
\begin{equation*}
    \homega \in \cVom \hookrightarrow \rW^{\theta,\infty}(0,T;\rH^{2(1-\theta)}(\hGam)), \tforall \theta \in (0,1).
\end{equation*}
In conjunction with Sobolev embeddings and the fact that $\hGam \subset \bR^2$ in the present 3D FPSI problem, the latter embedding leads to
\begin{equation}\label{eq:reg of omega in fin energy space}
    \homega \in \cVom \hookrightarrow \rC(0,T;\rC^{0,r}(\hGam)) \tfor r < 1.
\end{equation}
Because of \eqref{eq:reg of omega in fin energy space}, the ALE map $\hbfPhi_\rf^\omega$ from \eqref{eq:trafos to phys doms}$_3$ is not necessarily a Lipschitz function.
As we will see below, this also has consequences on the regularity of the fluid velocity $\hbfu$ defined on the reference domain~$\hOmf$.

Next, with regard to the transformation of the associated integrals, we provide the determinants of the Jacobians of the aforementioned maps.
Indeed, straightforward calculations reveal that 
\begin{equation}\label{eq:dets of Jacobians}
    \hcJ_\rb^\eta = \det(\bfI + \hnabla \hbfeta), \enspace \hcJ_\Gamma^\omega = \sqrt{1 + |\del_{\hx} \homega|^2 + |\del_{\hy} \homega|^2} \tand \hcJ_\rf^\omega = 1 + \frac{\homega}{R}.
\end{equation}
A few comments on the latter quantities are in order now.
Note that $\hcJ_\Gamma^\omega$ gives a measure of the arc length difference between the reference and deformed configuration of the plate.
In addition, let us observe that we do not include an absolute value sign in $\hcJ_\rf^\omega$, since we are only concerned with the situation up to degeneracy of the domain, i.e., up to $|\homega| \ge R$.

{\bf{Regularity of the transformed fluid velocity.}}
The mappings that we have just made precise allow us to describe the precise way the functions under consideration are transformed.
In particular, the fluid velocity $\bfu$ defined on the time-dependent domain $\Omf(t)$ is transformed to the reference domain by
\begin{equation*}
    \hbfu(t,\hx,\hy,\hz) = \bfu \circ \hbfPhi_\rf^\omega, \tfor (\hx,\hy,\hz) \in \hOmf.
\end{equation*}
With regard to \eqref{eq:reg of omega in fin energy space}, it does {\em not} necessarily follow from $\bfu \in \rH^1(\Omf(t))^3$ that $\hbfu \in \rH^1(\hOmf)^3$.
We shall take this into consideration when making precise the solution and test function spaces in \autoref{ssec:weak sols to the reg probl}.

{\bf{The transformed divergence-free condition.}} Recall that the fluid velocity $\bfu$ is assumed to be divergence free on the physical domain $\Omf(t)$, i.e., $\nabla \cdot \bfu = 0$.
The pull-back to the reference configuration generally does {\em not} preserve this property, so~$\hbfu$ defined on $\hOmf$ is in general not divergence free.
This calls for a reformulation of the divergence free condition on the fixed domain, which we discuss now.
Let $g$ be a function defined on $\Omf(t)$.
The definition of the pull-back together with the chain rule then leads to
\begin{equation}\label{eq:repr of gradient on moving fluid dom}
    \nabla g = \nabla \bigl(\hg \circ (\hbfPhi_\rf^\omega)^{-1}\bigr) = (\hnabla_\rf^\omega \hg) \circ (\hbfPhi_\rf^\omega)^{-1}.
\end{equation}
Here, $\hnabla_\rf^\omega$ represents the transformed gradient operator, which is of the form
\begin{equation}\label{eq:transformed fluid grad}
    \hnabla_\rf^\omega = \begin{pmatrix}
        \del_{\hx} - (R + z) \del_{\hx} \homega \frac{R}{(R+\homega)^2}\del_{\hz}\\
        \del_{\hy} - (R + z) \del_{\hy} \homega \frac{R}{(R+\homega)^2}\del_{\hz}\\
        \frac{R}{R+\homega}\del_{\hz}
    \end{pmatrix}, \twhere z = \hz + \Bigl(1 + \frac{\hz}{R}\Bigr) \homega.
\end{equation}
As a result, the divergence free condition and the symmetric part of the gradient can be rephrased as
\begin{equation*}
    \hnabla_\rf^\omega \cdot \hbfu = 0 \tand \hbfD_\rf^\omega(\hbfu) = \frac{1}{2}\Bigl(\hnabla_\rf^\omega \hbfu + (\hnabla_\rf^\omega \hbfu)^\top\Bigr).
\end{equation*}

{\bf{The transformed time derivative.}} For the transformed time derivative under the map $\hbfPhi_\rf^\omega$, we obtain
\begin{equation}\label{eq:transformed time der}
    \dt \bfu = \dt \hbfu - (\hbfw \cdot \hnabla_\rf^\omega) \hbfu, \twhere \hbfw = \frac{R + \hz}{R} \dt \homega \bfe_3.
\end{equation}

{\bf{Transformed functions and differential operators on the Biot domain.}}
We now examine the effect of the transformation $\hbfPhi_\rb^\eta$ on functions and differential operators defined on the Biot domain.
Let~$g$ be a scalar function defined on the time-dependent physical domain $\Omb(t)$.
The pull-back of $g$ to the fixed reference domain $\hOmb$ is given by
\begin{equation*}
    \hg = g \circ \hbfPhi_\rb^\eta.
\end{equation*}
As in the case of the fluid domain, we aim to express differential operators on the physical domain in terms of corresponding operators acting on the reference domain.
In particular, 
\begin{equation*}
    \nabla g = \nabla(\hg \circ (\hbfPhi_\rb^\eta)^{-1}) = (\hnabla_\rb^\eta \hg) \circ (\hbfPhi_\rb^\eta)^{-1},
\end{equation*}
for $\nabla$ representing the gradient on the (time-dependent) physical domain, and $\hnabla_\rb^\eta$ denoting a differential operator on the (fixed) reference domain.

Next, we want to calculate the expression for the gradient operator on the reference domain, which we denote by  $\hnabla$.
Using the chain rule together with the definition of the map $\hbfPhi_\rb^\eta$, we obtain
\begin{equation*}
    \hnabla (g \circ \hbfPhi_\rb^\eta) = [(\nabla g) \circ \hbfPhi_\rb^\eta] \cdot (\bfI + \hnabla \hbfeta).
\end{equation*}
Therefore the transformed gradient operator $\hnabla_\rb^\eta$ on $\hOmb$ is of the form
\begin{equation}\label{eq:transformed Biot grad}
    \hnabla_\rb^\eta \hg = (\del_{\hx} \hg,\del_{\hy} \hg,\del_{\hz} \hg) \cdot (\bfI + \hnabla \hbfeta)^{-1}.
\end{equation}
Let us observe that the invertibility of the matrix $\bfI + \hnabla \hbfeta$ is directly linked to the question of the map 
\begin{equation*}
    (\hx,\hy,\hz) \mapsto (\hx,\hy,\hz) + \hbfeta(\hx,\hy,\hz)
\end{equation*}
being a bijection between the domains $\hOmb$ and $\Omb(t)$.

\subsection{Definition of a weak solution}\label{ssec:def of a weak sol}

After precisely defining the mappings between the various domains, we now turn to the weak formulation of the 3D nonlinearly coupled FPSI problem under consideration. This formulation will be derived through formal calculations, beginning with the fluid equations.

Before proceeding, we comment on an important issue related to the regularity of the fluid domain and its implications for the trace. In the current setting, the fluid domain can be (locally) represented as the subgraph of a $\rC^{0,r}$-function. This allows for the definition of the so-called Lagrangian trace:
\begin{equation}\label{eq:Lagrangian trace}
\gamma_{\Gamma(t)} \colon \rC^1\bigl(\overline{\Omf(t)}\bigr) \to \rC(\Gamma(t)), \enspace \gamma_{\Gamma(t)} u \coloneqq u(t,x,y,\homega(t,x,y)).
\end{equation}
As established in \cite{CDEG:05, Gra:08, Muh:14}, this trace operator $\gamma_{\Gamma(t)}$ admits a continuous linear extension from $\rH^1(\Omf(t))$ to $\rH^s(\Gamma(t))$ for $s \in [0,\nicefrac{1}{2})$. The precise result can be found in \cite[Theorem~3.1]{Muh:14}.
We note that in the 3D FPSI setting the fluid domain generally fails to be Lipschitz due to the limited regularity of the reticular-plate interface. This distinguishes the 3D, nonlinearly coupled problem from its 2D counterpart. As a result, the weak-solution spaces introduced below for the 3D problem differ  from those used in the 2D analysis in \cite{KCM:24}. 

Since the formal derivation of the weak formulation is analogous in 2D and 3D, we do not reproduce the standard steps here and instead refer the reader to~\cite[Section 4.2]{KCM:24}; we also do not delve into the detailed function-space framework or the analytical subtleties caused by the limited regularity of the plate displacement, as these technical aspects are addressed in \autoref{ssec:weak sols to the reg probl}. Furthermore, for simplicity, we formulate the weak problem with the fluid velocity defined on the moving domain, thereby avoiding complications related to the potential lack of regularity of the ALE mapping $\hbfPhi_\rf^\omega$ introduced in \eqref{eq:trafos to phys doms}$_3$. We therefore recall the transverse plate velocity, defined by:
\begin{equation}\label{eq:hzeta}
    \dt \homega = \hzeta \tand \zeta = \hzeta \circ (\hbfPhi_\Gamma^\omega)^{-1}.
\end{equation}

\begin{defn}\label{def:weak form FPSI}
We say that the tuple $(\bfu, \homega, \hbfeta, p)$ satisfies the weak formulation of the 3D nonlinearly coupled FPSI problem if, for every test function tuple $(\bfv, \hphi, \hbfpsi, r)$ that is $\rC_\rc^1$ in time on $[0,T)$ and satisfy the following conditions:
\begin{itemize}
    \item $\nabla\cdot\bfv=0$, $\bfv = 0$ on $\del \Omf(t) \setminus \Gamma(t)$,
    \item $\hbfpsi = 0$ on $\del \hOmb \setminus \hGam$,
    \item $r = 0$ on $\del \Omb(t) \setminus \Gamma(t)$,
    \item  $\hbfpsi = \hphi \bfe_3$ on $\hGam$,
\end{itemize}
the following holds:
\begin{equation*}
    \begin{aligned}
        &\quad -\int_0^T \int_{\Omf(t)} \bfu \cdot \dt \bfv \srd \bfx \srd t + \frac{1}{2}\int_0^T\int_{\Omf(t)} \bigl[((\bfu \cdot \nabla)\bfu) \cdot \bfv - ((\bfu \cdot\nabla)\bfv) \cdot \bfu\bigr] \srd \bfx \srd t\\
        &\qquad + \frac{1}{2}\int_0^T\int_{\Gamma(t)} (\bfu \cdot \bfn - 2 \zeta \bfe_3 \cdot \bfn) \bfu \cdot \bfv \srd S \srd t + 2 \nu \int_0^T \int_{\Omf(t)} \bfD(\bfu) : \bfD(\bfv) \srd \bfx \srd t\\
        &\qquad + \int_0^T \int_{\Gamma(t)} \Bigl(\frac{1}{2}|\bfu|^2 - p\Bigr)(\psi_{\bfn} - v_{\bfn}) \srd S \srd t + \beta\sum_{i=1}^2 \int_0^T \int_{\Gamma(t)} (\zeta \bfe_3 - \bfu) \cdot \bftau_i (\psi_{\bftau_i} - v_{\bftau_i}) \srd S \srd t\\
        &\qquad - \rhop \int_0^T \int_{\hGam} \dt \homega \cdot \dt \hphi \srd \hS \srd t + \int_0^T \int_{\hGam} \hDelta \homega \cdot \hDelta \hphi \srd \hS \srd t - \rhob \int_0^T \int_{\hOmb} \dt \hbfeta \cdot \dt \hbfpsi \srd \hbfx \srd t\\
        &\qquad + 2 \mue \int_0^T \int_{\hOmb} \hbfD(\hbfeta) : \hbfD(\hbfpsi) \srd \hbfx \srd t + \lambde \int_0^T \int_{\hOmb} (\hnabla \cdot \hbfeta)(\hnabla \cdot \hbfpsi) \srd \hbfx \srd t\\
        &\qquad + 2 \muv \int_0^T \int_{\hOmb} \hbfD(\dt \hbfeta) : \hbfD(\hbfpsi) \srd \hbfx \srd t + \lambdv \int_0^T \int_{\hOmb} (\hnabla \cdot \dt \hbfeta)(\hnabla \cdot \hbfpsi) \srd \hbfx \srd t - \alpha \int_0^T \int_{\Omb(t)} p (\nabla \cdot \bfpsi) \srd \bfx \srd t\\
        &\qquad - c_0 \int_0^T \int_{\hOmb} \hp \cdot \dt \hr \srd \hbfx \srd t - \alpha \int_0^T \int_{\Omb(t)} \frac{\rD}{\rD t} \bfeta \cdot \nabla r \srd \bfx \srd t - \alpha \int_0^T \int_{\Gamma(t)} (\zeta \bfe_3 \cdot \bfn) r \srd S \srd t\\
        &\qquad + \kappa \int_0^T \int_{\Omb(t)} \nabla p \cdot \nabla r \srd \bfx \srd t - \int_0^T \int_{\Gamma(t)} ((\bfu - \zeta \bfe_3) \cdot \bfn) r \srd S \srd t\\
        &= \int_{\Omf(0)} \bfu(0) \cdot \bfv(0) \srd \bfx + \rhop \int_{\hGam} \dt \homega(0) \cdot \hphi(0) \srd \hS + \rhob \int_{\hOmb} \dt \hbfeta(0) \cdot \hbfpsi(0) \srd \hbfx + c_0 \int_{\hOmb} \hp(0) \cdot \hr(0) \srd \hbfx. 
    \end{aligned}
\end{equation*}
\end{defn}

\begin{remark}\label{Regularity}
We emphasize a crucial point: while any classical solution -- that is, a solution that is sufficiently smooth in both time and space -- of the 3D nonlinear FPSI problem automatically satisfies the weak formulation stated in \autoref{def:weak form FPSI}, this is not the case for finite-energy weak solutions. 

Indeed, if one formally derives an energy estimate following \autoref{def:weak form FPSI}, one obtains that the following energy estimate holds (at least for classical solutions):
\begin{equation}\label{Energy1}
    E^K(T) + E^E(T) + \int_0^T \bigl(D_\rf^V(t) + D_\rb^V(t) + D_{\rfb}^V(t) + D_\beta^V(t)\bigr) \srd t = E^K(0) + E^E(0),
\end{equation}
where
   $E^K(t) = \frac{1}{2} \int_{\Omf(t)} |\bfu(t)|^2 \srd \bfx + \frac{\rhob}{2} \int_{\hOmb} |\dt \hbfeta(t)|^2 \srd \hbfx + \frac{c_0}{2} \int_{\hOmb} |\hp(t)|^2 \srd \hbfx + \frac{\rhop}{2} \int_{\hGam} |\dt \homega(t)|^2 \srd \hS$,
is the total kinetic energy of the problem,
$
    E^E(t) = \mue \int_{\hOmb} |\hbfD(\hbfeta)(t)|^2 \srd \hbfx + \frac{\lambde}{2} \int_{\hOmb} |\hnabla \cdot \hbfeta(t)|^2 \srd \hbfx + \frac{1}{2}\int_{\hGam} |\hDelta \homega(t)|^2 \srd \hS
$
is the elastic energy of the Biot poroelastic matrix and the reticular plate, and\begin{equation*}
    \begin{aligned}
        D_\rf^V(t)
        &= 2 \nu \int_{\Omf(t)} |\bfD(\bfu)|^2 \srd \bfx, \enspace D_\rb^V(t) = 2 \muv \int_{\hOmb} |\hbfD(\dt \hbfeta)|^2 \srd \hbfx + \lambdv \int_{\hOmb} |\hnabla \cdot \dt \hbfeta|^2 \srd \hbfx,\\
        D_{\rfb}^V(t)
        &= \kappa \int_{\Omb(t)} |\nabla p|^2 \srd \bfx \tand D_\beta^V(t) = \beta \sum_{i=1}^2 \int_{\Gamma(t)} |(\bfxi - \bfu) \cdot \bftau_i|^2 \srd S,
    \end{aligned}
\end{equation*}
are dissipation terms due to fluid viscosity ($D_\rf^V$), the viscosity of the Biot poro(visco)elastic matrix~($D_\rb^V$), permeability effects ($D_{\rfb}^V$) and friction in the Beavers--Joseph--Saffman slip condition ($D_\beta^V$).
Details behind the derivation of this energy equality are presented later in \autoref{sec:EnergyEquality1}.

Indeed, this energy equality shows that the structure displacement $\hbfeta$ only belongs to $\rL^\infty(0,T;\rH^1(\hOmb)^3)$, which is insufficient to ensure that the following integrals in the weak formulation are well-defined:
\begin{equation}\label{eq:crit ints weak form}
\alpha \int_0^T \int_{\Omb(t)} p (\nabla \cdot \bfpsi) \srd \bfx \srd t, \,
\alpha \int_0^T \int_{\Omb(t)} \frac{\rD}{\rD t} \bfeta \cdot \nabla r \srd \bfx \srd t, \, \alpha \int_0^T \int_{\Gamma(t)} (\zeta \bfe_3 \cdot \bfn) r \srd S \srd t, \,
\kappa \int_0^T \int_{\Omb(t)} \nabla p \cdot \nabla r \srd \bfx \srd t.
\end{equation}
The problem arises because the test functions $\hbfpsi$ and $\hr$ live in $\rH^1(\hOmb)^3$ and $\rH^1(\hOmb)$, respectively -- on the reference (fixed) domain. Upon transforming to the physical (moving) domain, an additional factor -- the Jacobian determinant $\hcJ_\rb^\eta = \det(\bfI + \hnabla \hbfeta)$ from \eqref{eq:dets of Jacobians} -- must be taken into account. However, this Jacobian is only known to be in $\rL^\infty(0,T;\rL^1(\hOmb))$ under finite-energy regularity.

As a result, the regularity of $\hcJ_\rb^\eta$ is not sufficient to guarantee that the integrals in \eqref{eq:crit ints weak form} are finite. This observation is essential: it shows that the weak formulation from \autoref{def:weak form FPSI} cannot be interpreted in the setting of finite-energy weak solutions. A different, more suitable formulation is therefore required to handle weak solutions in this lower regularity regime.
\end{remark}

\section{The regularized problem, notion of a weak solution and existence result}\label{sec:reg probl, weak sol & ex result}

To address the regularity difficulties associated with the Biot problem discussed in \autoref{Regularity}, we introduce a {\emph{regularized problem}} that serves as an approximation of the original formulation. The regularization differs from the original problem in the smallest number of terms, and is ``consistent'' with the original problem in the sense that, whenever a classical solution to the original problem exists, the solution to the regularized problem converges to it as the regularization parameter tends to zero. A similar regularization strategy was successfully employed in \cite{KCM:24} to overcome weak regularity challenges in the 2D setting.

Since the mathematical difficulties associated with weak regularity stem primarily from the lack of smoothness of the Biot displacement $\hbfeta$ as the Biot domain evolves, we start by 
regularizing the Biot displacement $\hbfeta$ on $\hOmb$ via a convolution with a smooth and compactly supported kernel.
The convolution is chosen carefully, since the domain $\hOmb$ is bounded -- it is defined in a way that preserves the Dirichlet boundary conditions on the lateral boundaries as well as the top boundary.

For this purpose, we introduce an {\em extended domain} $\tOmb$ given by
\begin{equation*}
    \tOmb = [-L,2L] \times [-L,2L] \times [-R,2R].
\end{equation*}
For $\delta < \min\{L,R\}$, the convolution of a function defined on the extended domain $\tOmb$ with a smooth and compactly supported function in the closed ball of radius $\delta$ yields a function that is well-defined on $\hOmb$.

We introduce an odd extension of $\hbfeta$ along the lines $\hx = 0$, $\hx = L$, $\hy = 0$, $\hy = L$, $\hz = 0$ as well as $\hz = R$.
To this end, let us recall that $\hbfeta$ satisfies $\hbfeta = 0$ on $\hx = 0$, $\hx = L$, $\hy = 0$, $\hy = L$ and $\hz = R$, while it fulfills $\hbfeta = \homega \bfe_3$ on $\hz = 0$.
The odd extension of $\hbfeta$ to $\tOmb$ is then given as follows:
On the closure of the original domain -- still denoted by $\hOmb = [0,L] \times [0,L] \times [0,R]$ by a slight abuse of notation -- we keep the extension -- still denoted by $\hbfeta$ -- the same.
In contrast, outside of the closure of $\hOmb$, we successively define~$\hbfeta$ by
\begin{equation}\label{eq:odd ext}
    \left\{
    \begin{aligned}
        \hbfeta(\hx,\hy,\hz)
        &= \omega(\hx,\hy) \bfe_3 + \bigl(\homega(\hx,\hy) \bfe_3 - \hbfeta(\hx,\hy,-\hz)\bigr), &&\ton [0,L] \times [0,L] \times [-R,0],\\
        \hbfeta(\hx,\hy,\hz)
        &= -\hbfeta(\hx,\hy,2R-\hz), &&\ton [0,L] \times [0,L] \times [R,2R],\\
        \hbfeta(\hx,\hy,\hz)
        &= -\hbfeta(-\hx,\hy,\hz), &&\ton [-L,0] \times [0,L] \times [-R,2R],\\
        \hbfeta(\hx,\hy,\hz)
        &= -\hbfeta(2L-\hx,\hy,\hz), &&\ton [L,2L] \times [0,L] \times [-R,2R],\\
        \hbfeta(\hx,\hy,\hz)
        &= -\hbfeta(\hx,-\hy,\hz), &&\ton [-L,2L] \times [-L,0] \times [-R,2R],\\
        \hbfeta(\hx,\hy,\hz)
        &= -\hbfeta(\hx,2L-\hy,\hz), &&\ton [-L,2L] \times [L,2L] \times [-R,2R].
    \end{aligned}
    \right.
\end{equation}

Next, let $\sigma$ be a radially symmetric function with compact support in the closed unit ball such that
\begin{equation*}
    \int_{\bR^3} \sigma \srd \bfx = 1, \tand 
\end{equation*}
\begin{equation*}
    \sigmad(\bfx) \coloneqq \delta^{-3} \sigma(\delta^{-1}\bfx) \ton \bR^3.
\end{equation*}
Note that $\sigmad$ then has compact support in the closed ball of radius $\delta$, and we will use $\sigmad$ to convolve with the Biot displacement for the regularization.

More precisely, we introduce the following spatially smooth regularized functions on $\hOmb$.
First, we define the regularized Biot domain by invoking the odd extension to $\tOmb$ from \eqref{eq:odd ext} and setting
\begin{equation*}
    \hbfeta^\delta(t,\hbfx) = (\hbfeta \ast \sigma_\delta)(t,\hbfx) = \int_{\bR^3} \hbfeta(t,\hbfx - \hbfy) \sigmad(\hbfy) \srd \hbfy, \ton \hOmb. 
\end{equation*}
Accordingly, we define the regularized Lagrangian mapping
\begin{equation*}
    \hbfPhi_\rb^{\eta^\delta}(t,\cdot) = \Id + \hbfeta^\delta(t,\cdot),
\end{equation*}
as well as the regularized moving Biot domain
\begin{equation*}
    \Ombd(t) = \hbfPhi_\rb^{\etad}(t,\hOmb).
\end{equation*}
Let us observe that in spite of the coupling condition $\left. \hbfeta \right|_{\hGam} = \homega \bfe_3$, we can generally {\em not} guarantee that~$\left. \hbfeta^\delta \right|_{\hGam} = \homega \bfe_3$.
Therefore, we also define the regularized moving interface
\begin{equation*}
    \Gad(t) = \hbfPhi_\rb^{\etad}(t,\hGam).
\end{equation*}
An alternative interpretation is that $\hGam^\delta$ is the plate interface when displaced from the reference configuration $\hGam$ in the direction $\left.\hbfeta^\delta \right|_{\hGam}$.
Note that this is a purely transverse displacement.

Another important observation is that by the odd extension described in \eqref{eq:odd ext}, we have
\begin{equation*}
    \hbfeta^\delta = 0 \ton \del \hOmb \setminus \hGam.
\end{equation*}

We can now introduce a {\emph{regularized weak formulation}} of the $3D$ nonlinearly coupled FPSI with regularization parameter $\delta$.
First, we introduce the solution and test space.
These are motivated by the subsequent energy estimates.
After that, we state the regularized weak formulation in the moving domain framework.

\subsection{Weak solutions to the regularized problem}\label{ssec:weak sols to the reg probl}
In this subsection, we define a weak solution to the {\emph{regularized 3D nonlinearly coupled FPSI problem}}.
We emphasize that, as already touched upon in \autoref{sec:def of a weak sol}, the transition from the 2D setting studied in \cite{KCM:24} to the present 3D framework introduces several additional mathematical challenges -- most notably due to the lower spatial regularity of the plate displacement $\homega$.

Specifically, the regularity result in \eqref{eq:reg of omega in fin energy space} necessitates a more refined construction of the solution and test spaces, including an extension of the Lagrangian trace operator from \eqref{eq:Lagrangian trace} to accommodate reduced spatial regularity.
Furthermore, the lack of Lipschitz continuity of the ALE map in 3D complicates the transformation between the reference and physical domains, affecting the formulation of the fluid equations.
Our approach builds upon and generalizes techniques developed in \cite{MC:13b, MC:15} to handle the unique analytical obstacles posed by the 3D setting.

In the definition below, we present the solution and test spaces, taking into account the reduced regularity of the plate displacement and the associated difficulties with the trace as well as the ALE mapping.

\begin{defn}[Solution and test spaces]\label{def:sol & test space}
\ 

\begin{enumerate}[(a)]
    \item For the fluid part on the moving domain, i.e., in the Eulerian formulation, we first define
    \begin{equation*}
        \prescript{}{0}{\Vf(t)} \coloneqq \bigl\{\bfu \in \rC^1\bigl(\overline{\Omf(t)}\bigr)^3 : \nabla \cdot \bfu = 0 \tand \bfu = 0 \tfor x, y \in \{0,L\} \tand z = -R\bigr\}.
    \end{equation*}
    Considering the discussion at the beginning of \autoref{ssec:def of a weak sol}, we set 
    \begin{equation*}
        \Vf(t) \coloneqq \overline{\prescript{}{0}{\Vf(t)}}^{\rH^1(\Omf(t))^3} \tand \cVf \coloneqq \rL^\infty(0,T;\rL^2(\Omf(t))^3) \cap \rL^2(0,T;\Vf(t)).
    \end{equation*}
    Similarly as in \cite{CDEG:05, Gra:08}, one can verify that $\Vf(t)$ can be characterized as
    \begin{equation*}
        \Vf(t) \coloneqq \bigl\{\bfu \in \rH^1(\Omf(t))^3 : \nabla \cdot \bfu = 0 \tand \bfu = 0 \tfor x, y \in \{0,L\} \tand z = -R\bigr\}. 
    \end{equation*}
    \item With regard to the fact that the ALE mapping $\hbfPhi_\rf^\omega$ from \eqref{eq:trafos to phys doms}$_3$ is in general not Lipschitz continuous, implying that $\hbfu$ is not necessarily in $\rH^1(\hOmf)^3$, we set
    \begin{equation}\label{eq:spat fluid sol space}
        \Vfom \coloneqq \bigl\{\hbfu = \bfu \circ \hbfPhi_\rf^\omega : \bfu \in \Vf(t)\bigr\}.
    \end{equation}
    The associated fluid part of the finite-energy solution space is then defined by
    \begin{equation*}
        \cVfom \coloneqq \rL^\infty(0,T;\rL^2(\hOmf)^3) \cap \rL^2(0,T;\Vfom).
    \end{equation*}
    \item Concerning the plate subproblem, we introduce the function space
    \begin{equation*}
        \cVom \coloneqq \rW^{1,\infty}(0,T;\rL^2(\hGam)) \cap \rL^\infty(0,T;\rH_0^2(\hGam)).
    \end{equation*}
    \item In the context of the Biot displacement, we invoke
    \begin{equation}\label{eq:Biot displacement sol space}
        \begin{aligned}
            \Vd
            &\coloneqq \bigl\{\hbfeta \in \rH^1(\hOmb)^3 : \hbfeta = 0 \tfor \hx, \hy \in \{0,L\} \tand \hz = R \taswellas \heta_x = \heta_y = 0 \ton \hGam\bigr\} \tand\\
            \cVb
            &\coloneqq \rW^{1,\infty}(0,T;\rL^2(\hOmb)^3) \cap \rL^\infty(0,T;\Vd) \cap \rH^1(0,T;\Vd).
        \end{aligned}
    \end{equation}
    \item For the Biot pore pressure, we define
    \begin{equation}\label{eq:Biot pore pressure sol space}
        \begin{aligned}
            \Vp
            &\coloneqq \bigl\{\hp \in \rH^1(\hOmb) : \hp = 0 \tfor \hx, \hy \in \{0,L\} \tand \hz = R\bigr\} \tand\\
            \cQb 
            &\coloneqq \rL^\infty(0,T;\rL^2(\hOmb)) \cap \rL^2(0,T;\Vp).
        \end{aligned}
    \end{equation}
    \item The weak solution space with the fluid velocity defined on the moving domain is given by
    \begin{equation*}
        \cVsol \coloneqq \bigl\{(\bfu,\homega,\hbfeta,\hp) \in \cVf \times \cVom \times \cVb \times \cQb : \hbfeta = \homega \bfe_3 \ton \hGam\bigr\}.
    \end{equation*}
    \item The weak solution space with the fluid velocity defined on the fixed domain takes the form
    \begin{equation*}
        \cVsolom \coloneqq \bigl\{(\hbfu,\homega,\hbfeta,\hp) \in \cVfom \times \cVom \times \cVb \times \cQb : \hbfeta = \homega \bfe_3 \ton \hGam\bigr\}.
    \end{equation*}
    \item We define the test space for the fluid on the moving domain by 
    \begin{equation*}
        \cVtest \coloneqq \bigl\{(\bfv,\hphi,\hbfpsi,\hr) \in \rC_\rc^1([0,T); \Vf(t) \times \rH_0^2(\hGam) \times \Vd \times \Vp) : \hbfpsi = \hphi \bfe_3 \ton \hGam\bigr\}.
    \end{equation*}
    \item The test space for the fluid defined on the fixed domain is given by
    \begin{equation*}
        \cVtestom \coloneqq \bigl\{(\hbfv,\hphi,\hbfpsi,\hr) \in \rC_\rc^1([0,T); \Vfom \times \rH_0^2(\hGam) \times \Vd \times \Vp) : \hbfpsi = \hphi \bfe_3 \ton \hGam\bigr\}.
    \end{equation*}
\end{enumerate}
\end{defn}

Below, we comment on the Hilbert space structure of the space $\Vfom$ introduced in \autoref{def:sol & test space}(b).

\begin{remark}
Under the non-degeneracy assumption $|\homega| < R$ on the domain, we can define a scalar product on $\Vfom$ as follows: 
\begin{equation*}
    (\hbfu,\hbfv)_{\Vfom} \coloneqq \int_{\hOmf} \Bigl(1 + \frac{\homega}{R}\Bigr) \bigl(\hbfu \cdot \hbfv + \hnabla_\rf^\omega \hbfu : \hnabla_\rf^\omega \hbfv\bigr) \srd \hbfx,
\end{equation*}
where $\hnabla_\rf^\omega$ is the transformed gradient defined in \eqref{eq:transformed fluid grad}.
Recalling the Jacobian $\hcJ_\rf^\omega = 1 + \frac{\homega}{R}$ from \eqref{eq:dets of Jacobians} and exploiting the definition of $\hnabla_\rf^\omega$, we find that
\begin{equation*}
    (\hbfu,\hbfv)_{\Vfom} = \int_{\hOmf} \Bigl(1 + \frac{\homega}{R}\Bigr) \bigl(\hbfu \cdot \hbfv + \hnabla_\rf^\omega \hbfu : \hnabla_\rf^\omega \hbfv\bigr) \srd \hbfx = \int_{\Omf(t)}  \bfu \cdot \bfv + \nabla \bfu : \nabla \bfv \srd \bfx = (\bfu,\bfv)_{\rH^1(\Omf(t))^3}.
\end{equation*}
As a result, $\bfu \mapsto \hbfu$ is an isometric isomorphism between $\Vf(t)$ and $\Vfom$, rendering $\Vfom$ a Hilbert space.
\end{remark}

After introducing the appropriate solution and test spaces in \autoref{def:sol & test space}, we now discuss the weak formulation of the regularized problem involving the {\em moving} fluid domain.

\begin{defn}[Weak formulation of the regularized problem with moving fluid domain]\label{def:reg weak form FPSI moving fluid dom}
\ 

\noindent
A tuple $(\bfu,\homega,\hbfeta,\hp) \in \cVsol$ is a weak solution to the regularized 3D nonlinearly coupled FPSI problem with regularization parameter $\delta$ if for any test function $(\bfv,\hphi,\hbfpsi,\hr) \in \cVtest$, the following  weak formulation is satisfied:
\begin{equation*}
    \begin{aligned}
        &\quad -\int_0^T \int_{\Omf(t)} \bfu \cdot \dt \bfv \srd \bfx \srd t + \frac{1}{2}\int_0^T\int_{\Omf(t)} \bigl[((\bfu \cdot \nabla)\bfu) \cdot \bfv - ((\bfu \cdot\nabla)\bfv) \cdot \bfu\bigr] \srd \bfx \srd t\\
        &\qquad + \frac{1}{2}\int_0^T\int_{\Gamma(t)} (\bfu \cdot \bfn - 2 \zeta \bfe_3 \cdot \bfn) \bfu \cdot \bfv \srd S \srd t + 2 \nu \int_0^T \int_{\Omf(t)} \bfD(\bfu) : \bfD(\bfv) \srd \bfx \srd t\\
        &\qquad + \int_0^T \int_{\Gamma(t)} \Bigl(\frac{1}{2}|\bfu|^2 - p\Bigr)(\psi_{\bfn} - v_{\bfn}) \srd S \srd t + \beta\sum_{i=1}^2 \int_0^T \int_{\Gamma(t)} (\zeta \bfe_3 - \bfu) \cdot \bftau_i (\psi_{\bftau_i} - v_{\bftau_i}) \srd S \srd t\\
        &\qquad - \rhop \int_0^T \int_{\hGam} \dt \homega \cdot \dt \hphi \srd \hS \srd t + \int_0^T \int_{\hGam} \hDelta \homega \cdot \hDelta \hphi \srd \hS \srd t - \rhob \int_0^T \int_{\hOmb} \dt \hbfeta \cdot \dt \hbfpsi \srd \hbfx \srd t\\
        &\qquad + 2 \mue \int_0^T \int_{\hOmb} \hbfD(\hbfeta) : \hbfD(\hbfpsi) \srd \hbfx \srd t + \lambde \int_0^T \int_{\hOmb} (\hnabla \cdot \hbfeta)(\hnabla \cdot \hbfpsi) \srd \hbfx \srd t\\
        &\qquad + 2 \muv \int_0^T \int_{\hOmb} \hbfD(\dt \hbfeta) : \hbfD(\hbfpsi) \srd \hbfx \srd t + \lambdv \int_0^T \int_{\hOmb} (\hnabla \cdot \dt \hbfeta)(\hnabla \cdot \hbfpsi) \srd \hbfx \srd t - \alpha \int_0^T \int_{\Ombd(t)} p (\nabla \cdot \bfpsi) \srd \bfx \srd t\\
        &\qquad - c_0 \int_0^T \int_{\hOmb} \hp \cdot \dt \hr \srd \hbfx \srd t - \alpha \int_0^T \int_{\Ombd(t)} \frac{\rD^\delta}{\rD t} \bfeta \cdot \nabla r \srd \bfx \srd t - \alpha \int_0^T \int_{\Gad(t)} (\zeta \bfe_3 \cdot \bfn^\delta) r \srd S \srd t\\
        &\qquad + \kappa \int_0^T \int_{\Ombd(t)} \nabla p \cdot \nabla r \srd \bfx \srd t - \int_0^T \int_{\Gamma(t)} ((\bfu - \zeta \bfe_3) \cdot \bfn) r \srd S \srd t\\
        &= \int_{\Omf(0)} \bfu(0) \cdot \bfv(0) \srd \bfx + \rhop \int_{\hGam} \dt \homega(0) \cdot \hphi(0) \srd \hS + \rhob \int_{\hOmb} \dt \hbfeta(0) \cdot \hbfpsi(0) \srd \hbfx + c_0 \int_{\hOmb} \hp(0) \cdot \hr(0) \srd \hbfx, 
    \end{aligned}
\end{equation*}
where the material derivative with respect to the regularized displacement is given by
\begin{equation*}
    \frac{\rD^\delta}{\rD t} = \frac{\rd}{\rd t}  + (\bfxi^\delta \cdot \nabla), \twith \bfxi^\delta(t,\cdot) = \dt \hbfeta^\delta\bigl(t,\hbfPhi_\rb^{\etad})^{-1}(t,\cdot)\bigr),
\end{equation*}
the vector $\bfn$ represents the upward pointing normal vector to $\Gamma(t)$, and $\bfn^\delta$ is the upward pointing normal vector to $\Gamma^\delta(t)$.
\end{defn}

As mentioned earlier, we introduced a minimal regularization, modifying only the four terms in the weak formulation that require it, while keeping all others unchanged, as in \cite{KCM:24}.

\begin{remark}\label{rem:rem on not reg weak sol}{\bf{Notation.}}
To ease notation, we will use the following conventions throughout the rest of this work:
\begin{enumerate}[(a)]
    \item Even though the solution to the above regularized problem depends on the regularization parameter~$\delta$ implicitly, we will omit writing $\delta$ explicitly if it is clear from the context that we are dealing with the solution to the regularized problem.
    \item We also omit writing the compositions with the maps $\hbfPhi_\rf^\omega$, $\hbfPhi_\Gamma^\omega$, $\hbfPhi_\rb^\eta$ as well as $\hbfPhi_\rb^{\etad}$ and their inverses explicitly.
    As a result, we simply write
    \begin{equation*}
        -\alpha \int_0^T \int_{\Ombd(t)} p \nabla \cdot \bfpsi \srd \bfx \srd t \enspace \text{instead of} \enspace -\alpha \int_0^T  \int_{\Ombd(t)} \Bigl(\hp \circ (\hbfPhi_\rb^{\etad})^{-1}\Bigr) \nabla \cdot \Bigl(\hbfpsi \circ (\hbfPhi_\rb^{\etad})^{-1}\Bigr) \srd \bfx \srd t
    \end{equation*}
    or
    \begin{equation*}
        - \int_0^T \int_{\Gamma(t)} ((\bfu - \zeta \bfe_3) \cdot \bfn) r \srd S \srd t \enspace \text{instead of} \enspace - \int_0^T \int_{\Gamma(t)} ((\bfu - (\zeta \circ (\hbfPhi_\Gamma^\omega)^{-1}) \bfe_3) \cdot \bfn) (\hr \circ (\hbfPhi_\Gamma^\omega)^{-1}) \srd S \srd t.
    \end{equation*}
\end{enumerate}
\end{remark}

We note that the above considerations implicitly assume the invertibility of the map $\hbfPhi_\rb^{\etad}$. This assumption will be addressed in more detail later, as it leads to a constraint on the time interval of existence. Specifically, we will show that invertibility can be ensured on a time interval $[0, T_\delta]$, where the length of the interval $T_\delta$ may depend on the regularization parameter $\delta$.

Next, we present the weak formulation of the regularized problem with the {\em fluid velocity defined on the fixed domain}. To this end, recall that the Jacobians $\hcJ_\rb^\eta$, $\hcJ_\Gamma^\omega$, and $\hcJ_\rf^\omega$, as defined in \eqref{eq:dets of Jacobians}, naturally arise in the transformation from the moving domains to the corresponding fixed reference domains.

We start with the first term in \autoref{def:reg weak form FPSI moving fluid dom}.
Using of the non-degeneracy assumption $|\homega| < R$, invoking the transformed time derivative from \eqref{eq:transformed time der}, also recalling $\hbfw = \frac{R + \hz}{R} \dt \homega \bfe_3$, employing the transformed fluid gradient from \eqref{eq:transformed fluid grad}, and integrating by parts in the $z$-direction, we find that
\begin{equation}\label{eq:der weak form fixed fluid dom first term}
    \begin{aligned}
        \int_{\Omf(t)} \bfu \cdot \dt \bfv \srd \bfx
        &= \int_{\hOmf} \Bigl(1 + \frac{\homega}{R}\Bigr) \hbfu \cdot \dt \hbfv \srd \hbfx - \int_{\hOmf} \Bigl(1 + \frac{\homega}{R}\Bigr) \hbfu \cdot \bigl[(\hbfw \cdot \hnabla_\rf^\omega) \hbfv\bigr] \srd \hbfx\\
        &= \int_{\hOmf} \Bigl(1 + \frac{\homega}{R}\Bigr) \hbfu \cdot \dt \hbfv \srd \hbfx - \frac{1}{R} \int_{\hOmf} \hbfu \cdot \bigl[(R + \hz) \dt \homega \del_{\hz} \hbfv\bigr] \srd \hbfx\\
        &= \int_{\hOmf} \Bigl(1 + \frac{\homega}{R}\Bigr) \hbfu \cdot \dt \hbfv \srd \hbfx - \frac{1}{2R} \int_{\hOmf} \hbfu \cdot \bigl[(R + \hz) \dt \homega \del_{\hz} \hbfv\bigr] \srd \hbfx\\
        &\quad + \frac{1}{2R} \int_{\hOmf} (\dt \homega) \hbfu \cdot \hbfv \srd \hbfx + \frac{1}{2R} \int_{\hOmf} \bigl[(R + \hz) \dt \homega \del_{\hz} \hbfu\bigr] \cdot \hbfv \srd \hbfx - \frac{1}{2} \int_{\hGam} (\hbfu \cdot \hbfv) \dt \homega \srd \hS\\
        &= \int_{\hOmf} \Bigl(1 + \frac{\homega}{R}\Bigr) \hbfu \cdot \dt \hbfv \srd \hbfx - \frac{1}{2} \int_{\hOmf} \Bigl(1 + \frac{\homega}{R}\Bigr) \bigl[((\hbfw \cdot \hnabla_\rf^\omega) \hbfv) \cdot \hbfu - ((\hbfw \cdot \hnabla_\rf^\omega) \hbfu) \cdot \hbfv \bigr] \srd \hbfx\\
        &\quad + \frac{1}{2R} \int_{\hOmf} (\dt \homega) \hbfu \cdot \hbfv \srd \hbfx - \frac{1}{2} \int_{\hGam} (\hbfu \cdot \hbfv) \dt \homega \srd \hS.
    \end{aligned}
\end{equation}
Using the expression for the normal vector to the interface $\bfn = (-\del_{\hx} \homega,-\del_{\hy} \homega,1) / \hcJ_\Gamma^\omega$, as well as the interface condition $\left.\zeta \bfe_3\right|_{\Gamma(t)} = \dt \homega \bfe_3$, which is a result of \eqref{eq:cont of displacement}, we find that
\begin{equation*}
    \int_0^T \int_{\Gamma(t)} (\zeta \bfe_3 \cdot \bfn) \bfu \cdot \bfv \srd S \srd t = \int_0^T \int_{\hGam} (\hbfu \cdot \hbfv) \dt \homega \srd \hS \srd t.
\end{equation*}
Therefore, the last term in \eqref{eq:der weak form fixed fluid dom first term} can be combined with one of the terms from \autoref{def:reg weak form FPSI moving fluid dom}.

\begin{remark}{\bf{Notation for scaled normal and tangent vectors.}}
Since the factor $\hcJ_\Gamma^\omega$ in the unit normal vector cancels out with the Jacobian of the transformation from~$\Gamma(t)$ to $\hGam$, we introduce the following normal and tangent vectors: 
\begin{equation}\label{eq:normal & tangential vectors}
    \hbfn^\omega = (-\del_{\hx} \homega,-\del_{\hy} \homega,1), \enspace \hbftau_1^\omega = (1,0,\del_{\hx} \homega) \cdot \frac{1}{\sqrt{1 + |\del_{\hx}\homega|^2}} \tand \hbftau_2^\omega = (0,1,\del_{\hy} \homega) \cdot \frac{1}{\sqrt{1 + |\del_{\hy}\homega|^2}}.
\end{equation}
Note that we include the normalization factors in the definition of the tangential vectors, since in contrast to the normal vector, this factor does not cancel out with the Jacobian $\hcJ_\Gamma^\omega$ when transforming to the fixed domain.
Moreover, we set
\begin{equation*}
    \hbfn^{\etad} = \Bigl(-\del_{\hx} \Bigl( \hbfeta^\delta \Big|_{\hGam}\Bigr),-\del_{\hy} \Bigl( \hbfeta^\delta \Big|_{\hGam}\Bigr),1\Bigr).
\end{equation*}
\end{remark}
The weak formulation of the regularized FPSI problem, involving a {\emph{fixed fluid domain}}, now reads:

\begin{defn}[Weak formulation of the regularized problem with fixed fluid domain]\label{def:reg weak form FPSI fixed fluid dom}
\ 

\noindent
A tuple $(\hbfu,\homega,\hbfeta,\hp) \in \cVsolom$ is a weak solution to the regularized 3D nonlinearly coupled FPSI problem with regularization parameter $\delta$ if for any test function $(\hbfv,\hphi,\hbfpsi,\hr) \in \cVtestom$, the following weak formulation holds:
\begin{equation*}
    \begin{aligned}
        &\enspace -\int_0^T \int_{\hOmf} \Bigl(1 + \frac{\homega}{R}\Bigr) \hbfu \cdot \dt \hbfv \srd \hbfx \srd t + \frac{1}{2}\int_0^T\int_{\hOmf} \Bigl(1 + \frac{\homega}{R}\Bigr) \bigl[\bigl(((\hbfu - \hbfw) \cdot \hnabla_\rf^\omega)\hbfu\bigr) \cdot \hbfv\\
        &\qquad-\bigl(((\hbfu - \hbfw) \cdot \hnabla_\rf^\omega)\hbfv\bigr) \cdot \hbfu\bigr] \srd \hbfx \srd t -\frac{1}{2R} \int_0^T \int_{\hOmf} (\dt \homega) \hbfu \cdot \hbfv \srd \hbfx \srd t\\
        &\quad + \frac{1}{2}\int_0^T\int_{\hGam} (\hbfu \cdot \hbfn^\omega - \hzeta \bfe_3 \cdot \hbfn^\omega) \hbfu \cdot \hbfv \srd \hS \srd t + 2 \nu \int_0^T \int_{\hOmf} \Bigl(1 + \frac{\homega}{R}\Bigr) \hbfD(\hbfu) : \hbfD(\hbfv) \srd \hbfx \srd t\\
        &\quad +\int_0^T \int_{\hGam} \Bigl(\frac{1}{2}|\hbfu|^2 - \hp\Bigr)(\hbfpsi - \hbfv) \cdot \hbfn^\omega \srd \hS \srd t + \beta\sum_{i=1}^2 \int_0^T \int_{\hGam} \hcJ_\Gamma^\omega (\hzeta \bfe_3 - \hbfu) \cdot \hbftau_i^\omega(\bfpsi - \hbfv) \cdot \hbftau_i^\omega \srd \hS \srd t\\
        &\quad - \rhop \int_0^T \int_{\hGam} \dt \homega \cdot \dt \hphi \srd \hS \srd t + \int_0^T \int_{\hGam} \hDelta \homega \cdot \hDelta \hphi \srd \hS \srd t - \rhob \int_0^T \int_{\hOmb} \dt \hbfeta \cdot \dt \hbfpsi \srd \hbfx \srd t\\
        &\quad + 2 \mue \int_0^T \int_{\hOmb} \hbfD(\hbfeta) : \hbfD(\hbfpsi) \srd \hbfx \srd t + \lambde \int_0^T \int_{\hOmb} (\hnabla \cdot \hbfeta)(\hnabla \cdot \hbfpsi) \srd \hbfx \srd t\\
        &\quad + 2 \muv \int_0^T \int_{\hOmb} \hbfD(\dt \hbfeta) : \hbfD(\hbfpsi) \srd \hbfx \srd t + \lambdv \int_0^T \int_{\hOmb} (\hnabla \cdot \dt \hbfeta)(\hnabla \cdot \hbfpsi) \srd \hbfx \srd t - \alpha \int_0^T \int_{\hOmb} \hcJ_\rb^{\etad} \hp \hnabla_\rb^{\etad} \cdot \hbfpsi \srd \hbfx \srd t\\
        &\quad - c_0 \int_0^T \int_{\hOmb} \hp \cdot \dt \hr \srd \hbfx \srd t - \alpha \int_0^T \int_{\hOmb} \hcJ_\rb^{\etad} \dt \hbfeta \cdot \hnabla_\rb^{\etad} \hr \srd \hbfx \srd t - \alpha \int_0^T \int_{\hGam} (\hzeta \bfe_3 \cdot \hbfn^{\etad}) \hr \srd \hS \srd t\\
        &\quad + \kappa \int_0^T \int_{\hOmb} \hcJ_\rb^{\etad} \hnabla_\rb^{\etad} \hp \cdot \hnabla_\rb^{\etad} \hr \srd \hbfx \srd t - \int_0^T \int_{\hGam} ((\hbfu - \hzeta \bfe_3) \cdot \hbfn^\omega) \hr \srd \hS \srd t\\
        &= \int_{\Omf(0)} \bfu(0) \cdot \bfv(0) \srd \bfx + \rhop \int_{\hGam} \dt \homega(0) \cdot \hphi(0) \srd \hS + \rhob \int_{\hOmb} \dt \hbfeta(0) \cdot \hbfpsi(0) \srd \hbfx + c_0 \int_{\hOmb} \hp(0) \cdot \hr(0) \srd \hbfx, 
    \end{aligned}
\end{equation*}
where we recall the definitions of $\hcJ_\rb^{\etad}$ and $\hcJ_\Gamma^\omega$ from \eqref{eq:dets of Jacobians}, $\hbfw$ from \eqref{eq:transformed time der}, $\hnabla_\rf^\omega$ from \eqref{eq:transformed fluid grad}, $\hnabla_\rb^{\etad}$ from \eqref{eq:transformed Biot grad} as well as $\hzeta$ from \eqref{eq:hzeta}.
\end{defn}

\subsection{A formal energy inequality}\label{sec:EnergyEquality1}
Here we show that the weak formulation of the regularized FPSI problem is ``energy consistent'' with the original, non-regularized problem in the sense that the two problems satisfy the same type of energy estimate. 
The only difference is that the integrals over the moving Biot domain $\Omb(t)$ in the energy equality of the non-regularized problem \eqref{Energy1} are now given over the regularized moving Biot domain $\Ombd(t)$. 

\begin{lem}\label{lem:formal energy eq}
Suppose that a weak solution to the regularized FPSI problem given in \autoref{def:reg weak form FPSI moving fluid dom} exists and is sufficiently smooth. Then the following energy equality holds:
\begin{equation*}
    E^K(T) + E^E(T) + \int_0^T \bigl(D_\rf^V(t) + D_\rb^V(t) + D_{\rfb}^V(t) + D_\beta^V(t)\bigr) \srd t = E^K(0) + E^E(0),
\end{equation*}
where
\begin{equation*}
    E^K(t) = \frac{1}{2} \int_{\Omf(t)} |\bfu(t)|^2 \srd \bfx + \frac{\rhob}{2} \int_{\hOmb} |\dt \hbfeta(t)|^2 \srd \hbfx + \frac{c_0}{2} \int_{\hOmb} |\hp(t)|^2 \srd \hbfx + \frac{\rhop}{2} \int_{\hGam} |\dt \homega(t)|^2 \srd \hS
\end{equation*}
is the total kinetic energy of the problem,
\begin{equation*}
    E^E(t) = \mue \int_{\hOmb} |\hbfD(\hbfeta)(t)|^2 \srd \hbfx + \frac{\lambde}{2} \int_{\hOmb} |\hnabla \cdot \hbfeta(t)|^2 \srd \hbfx + \frac{1}{2}\int_{\hGam} |\hDelta \homega(t)|^2 \srd \hS
\end{equation*}
is the elastic energy of the Biot poroelastic matrix and the reticular plate, and
\begin{equation*}
    \begin{aligned}
        D_\rf^V(t)
        &= 2 \nu \int_{\Omf(t)} |\bfD(\bfu)(t)|^2 \srd \bfx, \enspace D_\rb^V(t) = 2 \muv \int_{\hOmb} |\hbfD(\dt \hbfeta)(t)|^2 \srd \hbfx + \lambdv \int_{\hOmb} |\hnabla \cdot \dt \hbfeta(t)|^2 \srd \hbfx,\\
        D_{\rfb}^V(t)
        &= \kappa \int_{\Ombd(t)} |\nabla p(t)|^2 \srd \bfx \tand D_\beta^V(t) = \beta \sum_{i=1}^2 \int_{\Gamma(t)} |(\bfxi(t) - \bfu(t)) \cdot \bftau_i|^2 \srd S,
    \end{aligned}
\end{equation*}
are the dissipation terms due to the fluid viscosity $D_\rf^V$, the viscosity of the Biot poro(visco)elastic matrix~$D_\rb^V$, permeability effects $D_{\rfb}^V$, and friction in the Beavers--Joseph--Saffman slip condition $D_\beta^V$.
\end{lem}

The proof of this (formal) energy equality is based on using the test functions $(\hbfv,\hphi,\hbfpsi,\hr) = (\hbfu,\homega,\hbfeta,\hp)$, switching between the formulations on the fixed and moving domains and canceling appropriate terms. For more details one can also see the proof of Lemma 5.1 in  \cite{KCM:24}.

\subsection{Main result on the existence of a weak solution to the regularized problem}
The main result of this work is the proof of the existence of a weak solution to the regularized nonlinearly coupled FPSI problem in 3D, made precise in \autoref{ssec:weak sols to the reg probl}.

\begin{thm}\label{thm:ex of a weak sol}
Suppose that the problem parameters satisfy $\rhob$, $\mue$, $\lambde$, $\alpha$, $\rhop$, $\nu > 0$, and assume that the viscosity parameters of the Biot layer fulfill $\muv$, $\lambdv \ge 0$.
Moreover, suppose that the initial data satisfy~$\homega_0 \in \rH_0^2(\hGam)$ for the plate displacement, $\hzeta_0 \in \rL^2(\hGam)$ for the plate velocity, $\hp_0 \in \rL^2(\hOmb)$ for the Biot pore pressure, $\bfu_0 \in \rH^1(\Omf(0))^3$ for the fluid velocity with the additional property that $\bfu_0$ is divergence free, $\hbfeta_0 \in \rH^1(\hOmb)^3$ for the Biot displacement and
\begin{enumerate}[(a)]
    \item $\hbfxi_0 \in \rH^1(\hOmb)^3$ for the Biot velocity in the viscoelastic case $\muv$, $\lambdv > 0$, and
    \item $\hbfxi_0 \in \rL^2(\hOmb)^3$ for the Biot velocity in the purely elastic case $\muv = \lambdv = 0$.
\end{enumerate}
Furthermore, assume that there is $R_0 \in (0,R)$ with $|\homega_0| \le R_0$, the compatibility conditions $\left. \hbfeta_0 \right|_{\hGam} = \homega_0 \bfe_3$ as well as $\left. \hbfxi_0 \right|_{\hGam} = \hzeta_0 \bfe_3$ are satisfied, and that for some arbitrary but fixed regularization parameter $\delta > 0$, we have that $\det(\bfI + \nabla \hbfeta_0^\delta) > 0$.

Then {\bf{there exists $T > 0$}} such that the regularized 3D nonlinearly coupled FPSI problem, made precise in \autoref{ssec:weak sols to the reg probl}, {\bf{admits a weak solution}} $(\bfu,\homega,\hbfeta,\hp)$ on $[0,T]$ in the sense of \autoref{def:reg weak form FPSI moving fluid dom}.

Moreover, the weak solution satisfies the following {\bf{energy estimate}}:
\begin{equation*}
    \begin{aligned}
        &\quad \frac{1}{2} \int_{\Omf(t)} \bigl|\bfu(t)\bigr|^2 \srd \bfx + \frac{\rhob}{2} \int_{\hOmb} \bigl|\hbfxi(t) \bigr|^2 \srd \hbfx + \frac{c_0}{2} \int_{\hOmb} \bigl|\hp(t)\bigr|^2 \srd \hbfx + \mue \int_{\hOmb} |\hbfD(\hbfeta)(t)|^2 \srd \hbfx + \frac{\lambde}{2} \int_{\hOmb} |\hnabla \cdot \hbfeta(t)|^2 \srd \hbfx\\
        &\quad +\frac{\rhop}{2} \int_{\hGam}\bigl|\hzeta(t)\bigr|^2 \srd \hS + \frac{1}{2}\int_{\hGam} |\hDelta \homega(t)|^2 \srd \hS + 2 \nu \int_0^t \int_{\Omega_{\rf,N}(s)} \bigl|\bfD(\bfu)(s)\bigr|^2 \srd \bfx \srd s + 2\muv \int_0^t \int_{\hOmb} |\hbfD(\hbfxi)(s)|^2 \srd \hbfx \srd s\\
        &\quad + \lambdv \int_0^t \int_{\hOmb} |\hnabla \cdot \hbfxi(s)|^2 \srd \hbfx \srd s + \kappa \int_0^t \int_{\Omega_{\rb}^\delta(s)} \bigl|\nabla p(s)\bigr|^2 \srd \bfx \srd s + \beta \sum_{i=1}^2 \int_0^t \int_{\Gamma(s)} \bigl|(\zeta \bfe_3 - \bfu) \cdot \bftau_i\bigr|^2 \srd S \srd t \le E_0.
    \end{aligned}
\end{equation*}
\end{thm}

The remainder of this paper is devoted to proving \autoref{thm:ex of a weak sol}. The proof will be based on a splitting approach, which was first introduced to study weak solutions to moving boundary problems in \cite{MC:13a}.
For notational simplicity, we omit below the $\widehat{}$ notation distinguishing quantities defined on the reference domain from those on the physical domain, including the domains themselves, as the distinction between the two will be clear from the context.

\section{The splitting scheme}\label{sec:splitting scheme}
Similar to the strategies employed in \cite{MC:13a,KCM:24}, the general approach is to semidiscretize the problem in time by introducing time steps $\Delta t = \nicefrac{T}{N}$ and partitioning the interval $[0, T]$ into $N$ subintervals of length~$\Delta t$.
On each subinterval, we consider two subproblems:
\begin{enumerate}
\item a plate subproblem, which governs the elastodynamics of the reticular plate and is used to update the plate displacement and velocity; and
\item a fluid/Biot subproblem, which updates the fluid velocity, Biot displacement, Biot pore pressure, and plate velocity.
\end{enumerate}

The solution from the first subproblem is used as input data for the second.

The approximations of the fluid velocity, plate displacement and velocity, and Biot poro(visco)elastic material displacement and pressure, will be denoted by
\begin{equation*}
    (\bfu_N^{n + \frac{i}{2}},\omega_N^{n + \frac{i}{2}},\zeta_N^{n + \frac{i}{2}},\bfeta_N^{n + \frac{i}{2}},p_N^{n + \frac{i}{2}}), \tfor n=0,1,\dots,N \tand i=0,1.
\end{equation*}
They are defined on the given time subinterval $[n \Delta t,(n+1) \Delta t]$.
Moreover, the quantities with superscript $n + \nicefrac{1}{2}$ indicate the approximate solutions after solving the plate subproblem, while the quantities with superscript $n+1$ represent the approximate solutions after solving the fluid/Biot subproblem.

We remark that we will work on the weak formulation given in terms of the fixed reference domain, defined in \autoref{def:reg weak form FPSI fixed fluid dom}.
Consequently, we will semi-discretize the regularized weak formulation on the fixed reference domain.
Concerning the approximation of time derivatives, we will use backward Euler definition and employ the notation
\begin{equation}\label{eq:time discr}
    \df_N^{n + \frac{i}{2}} = \frac{f_N^{n+\frac{i}{2}}-f_N^{n+\frac{i}{2}-1}}{\Delta t}.
\end{equation}
Let us remark at this stage that due to the backward Euler discretization as described in \eqref{eq:time discr}, there can be negative subscripts.
We will explicitly define the quantities $f_{-1}$ and $f_{-\nicefrac{1}{2}}$ when necessary.

\subsection{The plate subproblem}\label{ssec:plate subprobl}
In this step, only the plate quantities $\omega_N^{n+\frac{1}{2}}$ and $\zeta_N^{n+\frac{1}{2}}$ are updated, while the fluid and Biot quantities remain unchanged:
\begin{equation*}
    \bfu_N^{n+\frac{1}{2}} = \bfu_N^n, \enspace \bfeta_N^{n+\frac{1}{2}} = \bfeta_N^n \tand p_N^{n+\frac{1}{2}} = p_N^n.
\end{equation*}
The plate displacement $\omega_N^{n+\frac{1}{2}}$ and velocity $\zeta_N^{n+\frac{1}{2}}$ are determined by using the weak formulation of the plate subproblems, i.e., find $\omega_N^{n+\frac{1}{2}} \in \rH_0^2(\Gamma)$ and $\zeta_N^{n+\frac{1}{2}} \in \rH_0^2(\Gamma)$
such that
\begin{equation}\label{eq:weak form plate subprobl}
    \begin{aligned}
        \int_\Gamma \biggl(\frac{\omega_N^{n+\frac{1}{2}}-\omega_N^{n-\frac{1}{2}}}{\Delta t}\biggr) \cdot \phi \srd S
        &= \int_\Gamma \zeta_N^{n+\frac{1}{2}} \cdot \phi \srd S, &&\tforall \phi \in \rL^2(\Gamma),\\
        \rhop \int_\Gamma \biggl(\frac{\zeta_N^{n+\frac{1}{2}}-\zeta_N^{n}}{\Delta t}\biggr) \cdot \varphi \srd S + \int_\Gamma \Delta \omega_N^{n+\frac{1}{2}} \cdot \Delta \varphi \srd S
        &= 0, &&\tforall \varphi \in \rH_0^2(\Gamma).
    \end{aligned}
\end{equation}
For the first subinterval, i.e., for $n=0$, we set $\omega_N^{-\frac{1}{2}} = \omega(0)$ and $\zeta_N^0 = \zeta(0)$, and these quantities satisfy $\omega(0) \bfe_3 = \left. \bfeta(0) \right|_{\Gamma}$ as well as $\zeta(0) \bfe_3 = \bfxi(0)$.

\begin{lem}\label{lem:ex of weak sol to plate subprobl}
There exists a unique solution $\bigl(\omega_N^{n+\frac{1}{2}},\zeta_N^{n+\frac{1}{2}}\bigr)$ to the semidiscretized plate subproblem \eqref{eq:weak form plate subprobl}, and $\bigl(\omega_N^{n+\frac{1}{2}},\zeta_N^{n+\frac{1}{2}}\bigr)$ satisfies the energy equality
\begin{equation*}
    \begin{aligned}
        &\quad \frac{\rhop}{2} \int_\Gamma |\zeta_N^{n+\frac{1}{2}}|^2 \srd S + \frac{\rhop}{2} \int_\Gamma |\zeta_N^{n+\frac{1}{2}} - \zeta_N^{n}|^2 \srd S + \frac{1}{2} \int_\Gamma |\Delta \omega_N^{n+\frac{1}{2}}|^2 \srd S + \frac{1}{2} \int_\Gamma |\Delta(\omega_N^{n+\frac{1}{2}} - \omega_N^{n-\frac{1}{2}})|^2 \srd S\\
        &= \frac{\rhop}{2} \int_\Gamma |\zeta_N^{n}|^2 \srd S + \frac{1}{2} \int_\Gamma |\Delta \omega_N^{n-\frac{1}{2}}|^2 \srd S.
    \end{aligned}
\end{equation*}
\end{lem}

The proof is a consequence of Lax--Milgram lemma, similar to the proof of Lemma 6.1 in \cite{KCM:24}.

\subsection{The fluid and Biot subproblem}\label{ssec:fluid & Biot subprobl}

In this subproblem the fluid and Biot quantities are updated.
However, by the kinematic coupling condition~\eqref{eq:cont of displacement} involving the Biot medium and plate displacement, it is necessary to update the plate velocity as well.
Thus, only the plate displacement stays intact, i.e., $\omega_N^{n+1} = \omega_N^{n+\frac{1}{2}}$, while the other quantities are recovered from the corresponding weak formulation.

For the weak formulation, we define the solution and test space by
\begin{equation}\label{eq:sol & test space fluid & Biot subprobl}
    \begin{aligned}
        \cV_N^{n+1}
        &\coloneqq \bigl\{(\bfu,\zeta,\bfeta,p) \in \Vf^{\omega_N^n} \times \rH_0^2(\Gamma) \times \Vd \times \Vp\bigr\} \tand\\
        \cQ_N^{n+1}
        &\coloneqq \{(\bfv,\varphi,\bfpsi,r) \in \Vf^{\omega_N^n} \times \rH_0^2(\Gamma) \times \Vd \times \Vp : \bfpsi = \varphi \bfe_3 \ton \Gamma\},
    \end{aligned}
\end{equation}
respectively, where $\Vf^{\omega_N^n}$, $\Vd$ and $\Vp$ are given in \eqref{eq:spat fluid sol space}, \eqref{eq:Biot displacement sol space}$_1$ and \eqref{eq:Biot pore pressure sol space}$_1$.

The weak formulation of the present fluid and Biot subproblem is now given as follows.

\begin{defn}\label{def:weak sol to fluid and Biot subproblem}
We say that the fluid and Biot subproblem in the splitting scheme admits a weak solution if there are $(\bfu_N^{n+1},\zeta_N^{n+1},\bfeta_N^{n+1},p_N^{n+1}) \in \cV_N^{n+1}$ such that for all test functions $(\bfv,\varphi,\bfpsi,r) \in \cQ_N^{n+1}$ defined on the (fixed) reference domain, the following holds:
\begin{equation*}
    \begin{aligned}
        &\int_{\Omf} \Bigl(1 + \frac{\omega_N^n}{R}\Bigr) \dbfu_N^{n+1} \cdot \bfv \srd \bfx + 2 \nu \int_{\Omf} \Bigl(1 + \frac{\omega_N^n}{R}\Bigr)\bfD_\rf^{\omega_N^n}(\bfu_N^{n+1}) : \bfD_\rf^{\omega_N^n}(\bfv) \srd \bfx\\
        &+\int_{\Gamma} \Bigl(\frac{1}{2}\bfu_N^{n+1} \cdot \bfu_N^n - p_N^{n+1}\Bigr)(\bfpsi - \bfv) \cdot \bfn^{\omega_N^n} \srd S\\
        &+ \frac{1}{2}\int_{\Omf} \Bigl(1 + \frac{\omega_N^n}{R}\Bigr) \Bigl[\Bigl(\Bigl(\Bigl(\bfu_N^n - \zeta_N^{n+\frac{1}{2}}\frac{R+z}{R}\bfe_3\Bigr) \cdot \nabla_\rf^{\omega_N^n}\Bigr)\bfu_N^{n+1}\Bigr) \cdot \bfv\\
        &\qquad -\Bigl(\Bigl(\Bigl(\bfu_N^n - \zeta_N^{n+\frac{1}{2}}\frac{R+z}{R}\bfe_3\Bigr) \cdot \nabla_\rf^{\omega_N^n}\Bigr)\bfv\Bigr) \cdot \bfu_N^{n+1}\Bigr] \srd \bfx + \frac{1}{2R} \int_{\Omf} \zeta_N^{n+\frac{1}{2}} \bfu_N^{n+1} \cdot \bfv \srd \bfx\\
        &+ \frac{1}{2}\int_{\Gamma} (\bfu_N^{n+1} - \dbfeta_N^{n+1}) \cdot \bfn^{\omega_N^n} (\bfu_N^n \cdot \bfv) \srd S +\beta\sum_{i=1}^2 \int_{\Gamma}\cJ_\Gamma^{\omega_N^n} (\dbfeta_N^{n+1} - \bfu_N^{n+1}) \cdot \bftau_i^{\omega_N^n}(\bfpsi - \bfv) \cdot \bftau_i^{\omega_N^n} \srd S\\
        &+\rhob\int_{\Omb} \Bigl(\frac{\dbfeta_N^{n+1}-\dbfeta_N^n}{\Delta t}\Bigr) \cdot \bfpsi \srd \bfx + \rhop\int_{\Gamma} \Bigl(\frac{\zeta_N^{n+1}-\zeta_N^{n+\frac{1}{2}}}{\Delta t}\Bigr) \varphi \srd S + 2 \mue \int_{\Omb} \bfD(\bfeta_N^{n+1}) : \bfD(\bfpsi) \srd \bfx\\
        &+ \lambde \int_{\Omb} (\nabla \cdot \bfeta_N^{n+1})(\nabla \cdot \bfpsi) \srd \bfx + 2 \muv \int_{\Omb} \bfD(\dbfeta_N^{n+1}) : \bfD(\bfpsi) \srd \bfx + \lambdv \int_{\Omb} (\nabla \cdot \dbfeta_N^{n+1})(\nabla \cdot \bfpsi) \srd \bfx\\
        &- \alpha \int_{\Omb} \cJ_\rb^{(\bfeta_N^n)^\delta} p_N^{n+1} \nabla_\rb^{(\bfeta_N^n)^\delta} \cdot \bfpsi \srd \bfx + c_0 \int_{\Omb} \frac{p_N^{n+1}-p_N^n}{\Delta t} r \srd \bfx - \alpha \int_{\Omb} \cJ_\rb^{(\bfeta_N^n)^\delta} \dbfeta_N^{n+1} \cdot \nabla_\rb^{(\bfeta_N^n)^\delta} r \srd \bfx\\
        &- \alpha \int_{\Gamma} (\dbfeta_N^{n+1} \cdot \bfn^{(\bfeta_N^n)^\delta}) r \srd S + \kappa \int_{\Omb} \cJ_\rb^{(\bfeta_N^n)^\delta} \nabla_\rb^{(\bfeta_N^n)^\delta} p_N^{n+1} \cdot \nabla_\rb^{(\bfeta_N^n)^\delta} r \srd \bfx - \int_{\Gamma} ((\bfu_N^{n+1} - \dbfeta_N^{n+1}) \cdot \bfn^{\omega_N^n}) r \srd S = 0
    \end{aligned}
\end{equation*}
as well as $\int_\Gamma \Bigl(\frac{\bfeta_N^{n+1}-\bfeta_N^n}{\Delta t}\Bigr) \cdot \bfphi \srd S = \int_\Gamma \zeta_N^{n+1} \bfe_3 \cdot \bfphi \srd S$ for all $\bfphi \in \rL^2(\Gamma)^3$.
\end{defn}

Note that in the case $n=0$, a negative subscript appears in the definition of $\dbfeta_N^0$ in the Euler discretization.
Thus, we set $\dbfeta_N^0$ to be the initial plate velocity $\zeta_0 \bfe_3$.

Below, we state the existence result of a unique weak solution to the fluid and Biot subproblem as discussed in \autoref{def:weak sol to fluid and Biot subproblem}.

\begin{lem}\label{lem:ex of weak sol to fluid & Biot subprobl}
The semidiscretized fluid and Biot subproblem in \autoref{def:weak sol to fluid and Biot subproblem} admits a unique weak solution in the sense of \autoref{def:weak sol to fluid and Biot subproblem} under the following assumptions:
\begin{enumerate}
    \item {\bf{ Assumption~1.A (Boundedness of plate displacement away from $R$):}}
    There is a positive constant~$R_0$ so that $|\omega_N^{k+\frac{i}{2}}| \le R_0 < R$ for all $k=0,1,\dots,n$ and $i=0,1$.
    \item {\bf{Assumption~2.A (Invertibility of $\Id + (\bfeta_N^n)^{\delta}$):}}
    The map $\Id + (\bfeta_N^n)^{\delta} \colon \Omb \to (\Omb)_{N}^{n,\delta}$,
    which maps the fixed Biot domain $\Omb$ onto $(\Omb)_{N}^{n,\delta}$, is invertible.
    \item{\bf{Assumption~3.A (Positivity of $\cJ_\rb^{(\bfeta_N^n)^\delta} = \det(\bfI + \nabla (\bfeta_N^n)^\delta)$):}} There exists a constant $c_0 > 0$ such that for all $N$, we have $\cJ_\rb^{(\bfeta_N^n)^\delta} = \det\bigl(\bfI + \nabla (\bfeta_N^n)^\delta\bigr) \ge c_0 > 0$ for all $n=0,1,\dots,N$.
\end{enumerate}
Moreover, the resulting weak solution to the fluid and Biot subproblem satisfies the following energy equality:
\begin{equation*}
    \begin{aligned}
        &\quad \frac{1}{2} \int_{\Omf} \Bigl(1+\frac{\omega_N^{n+1}}{R}\Bigr) \bigl|\bfu_N^{n+1}\bigr|^2 \srd \bfx + \frac{\rhob}{2} \int_{\Omb} \bigl|\dbfeta_N^{n+1} \bigr|^2 \srd \bfx + \frac{\rhop}{2} \int_\Gamma \bigl|\zeta_N^{n+1}\bigr|^2 \srd S\\
        &\quad + \frac{c_0}{2} \int_{\Omb} \bigl|p_N^{n+1}\bigr|^2 \srd \bfx + \mue \int_{\Omb} |\bfD(\bfeta_N^{n+1})|^2 \srd \bfx + \frac{\lambde}{2} \int_{\Omb} |\nabla \cdot \bfeta_N^{n+1}|^2 \srd \bfx\\
        &\quad + 2 \nu (\Delta t)\int_{\Omf} \Bigl(1 + \frac{\omega_N^n}{R}\Bigr)\bigl|\bfD_\rf^{\omega_N^n}(\bfu_N^{n+1})\bigr|^2 \srd \bfx + 2\muv (\Delta t) \int_{\Omb} |\bfD(\dbfeta_N^{n+1})|^2 \srd \bfx + \lambdv (\Delta t) \int_{\Omb} |\nabla \cdot \dbfeta_N^{n+1}|^2 \srd \bfx\\
        &\quad + \kappa (\Delta t) \int_{\Omb} \cJ_\rb^{(\bfeta_N^n)^\delta} \bigl|\nabla_\rb^{(\bfeta_N^n)^\delta} p_N^{n+1}\bigr|^2 \srd \bfx + \beta(\Delta t)\sum_{i=1}^2 \int_{\Gamma} \cJ_\Gamma^{\omega_N^n} \bigl|(\dbfeta_N^{n+1} - \bfu_N^{n+1}) \cdot \bftau_i^{\omega_N^n}\bigr|^2 \srd S\\
        &\quad + \frac{1}{2} \int_{\Omf} \Bigl(1 + \frac{\omega_N^n}{R}\Bigr) \bigl|\bfu_N^{n+1} - \bfu_N^n\bigr|^2 \srd \bfx + \frac{\rhob}{2} \int_{\Omb} \bigl|\dbfeta_N^{n+1} - \dbfeta_N^{n} \bigr|^2 \srd \bfx + \frac{\rhop}{2} \int_\Gamma \bigl|\zeta_N^{n+1} - \zeta_N^{n}\bigr|^2 \srd S\\
        &\quad + \frac{c_0}{2} \int_{\Omb} \bigl|p_N^{n+1}-p_N^n\bigr|^2 \srd \bfx + \mue \int_{\Omb} |\bfD(\bfeta_N^{n+1} - \bfeta_N^n)|^2 \srd \bfx + \frac{\lambde}{2} \int_{\Omb} |\nabla \cdot (\bfeta_N^{n+1} - \bfeta_N^{n})|^2 \srd \bfx\\
        &= \frac{1}{2} \int_{\Omf} \Bigl(1 + \frac{\omega_N^n}{R}\Bigr) \bigl|\bfu_N^n\bigr|^2 \srd \bfx + \frac{\rhob}{2} \int_{\Omb} \bigl|\dbfeta_N^{n} \bigr|^2 \srd \bfx + \frac{\rhop}{2} \int_\Gamma \bigl|\zeta_N^{n}\bigr|^2 \srd S + \frac{c_0}{2} \int_{\Omb} \bigl|p_N^n\bigr|^2 \srd \bfx\\
        &\quad + \mue \int_{\Omb} |\bfD(\bfeta_N^n)|^2 \srd \bfx + \frac{\lambde}{2} \int_{\Omb} |\nabla \cdot \bfeta_N^{n}|^2 \srd \bfx.
    \end{aligned}
\end{equation*}
\end{lem}

The proof of \autoref{lem:ex of weak sol to fluid & Biot subprobl} follows from the Lax--Milgram lemma after addressing two main challenges. The first is the challenge associated with the fact that
the bilinear form corresponding to the weak formulation is not coercive on $V_\rf^{\omega_N^n} \times \Vd \times \Vp$ due to a mismatch in hyperbolic-parabolic scaling.
This can be remedied by rescaling the test functions as follows:
\begin{equation*}
    \bfv \mapsto (\Delta t) \bfv \tand r \mapsto (\Delta t) r.
\end{equation*}
We note that the rescaled test functions are admissible as they satisfy $((\Delta t)^{-1} \bfv,\varphi,\bfpsi,(\Delta t)^{-1} r) \in \cQ_N^{n+1}$.

The second challenge is associated with the fact that the validity of Korn's inequality for the Biot domain {\em cannot} be guaranteed a priori.
However, an explicit calculation reveals the following Korn inequality.

\begin{lem}[Korn's inequality for $\Omb$]\label{lem:Korn's ineq Biot}
For $\bfeta \in \Vd$, it holds that $\int_{\Omb} |\bfD(\bfeta)|^2 \srd \bfx \ge \frac{1}{2} \int_{\Omb}|\nabla \bfeta|^2 \srd \bfx$.
\end{lem}
The proof is analogous to the one of \cite[Prop.~6.1]{KCM:24}, so we omit the details.

The rest of the proof of \autoref{lem:ex of weak sol to fluid & Biot subprobl} follows the same calculations as the proof of Lemma 6.2 in \cite{KCM:24}.

\bigskip
Our goal is to prove that the sequences of approximate solutions defined by this splitting scheme, converge to a solution of the regularized, nonlinearly coupled FPSI problem, as the time step $\Delta t$ tends to zero. The first step in this direction is to obtain uniform estimate on approximate solution. For this purpose it is useful to present the semidiscretized scheme in monolithic form, and derive the appropriate energy estimates. We present those estimates in the next section. 

\subsection{Weak formulation and energy equality for the coupled semidiscrete problem}\label{ssec:weak form & energy eq for coupled semi-discr probl}
We derive the monolithic (coupled) formulation of the semidiscrete problem by combining the weak formulation of the semidiscretized plate subproblem \eqref{eq:weak form plate subprobl} with the weak formulation of the semidiscretized fluid and Biot subproblem in \autoref{def:weak sol to fluid and Biot subproblem}, and using the fact that 
\begin{equation*}
    \rhop \int_\Gamma \Bigl(\frac{\zeta_N^{n+\frac{1}{2}}-\zeta_N^n}{\Delta t}\Bigr) \cdot \varphi \srd S + \rhop \int_\Gamma \Bigl(\frac{\zeta_N^{n+1}-\zeta_N^{n+\frac{1}{2}}}{\Delta t}\Bigr) \cdot \varphi \srd S = \rhop \int_\Gamma \Bigl(\frac{\zeta_N^{n+1}-\zeta_N^{n}}{\Delta t}\Bigr) \cdot \varphi \srd S.
\end{equation*}
The following monolithic weak formulation holds:
\begin{defn}\label{def:weak sol to coupled problem}
For each fixed $n \in \{1,...,N\}$, the weak formulation of the semidiscretized problem above, written in monolithic form, reads:  the functions $(\bfu_N^{n+1},\zeta_N^{n+1},\bfeta_N^{n+1},p_N^{n+1}) \in \cV_N^{n+1}$ and $\omega_N^{n+\frac{1}{2}}$, $\zeta_N^{n+\frac{1}{2}} \in \rH_0^2(\Gamma)$ are a weak solution if for all test functions~$(\bfv,\varphi,\bfpsi,r) \in \cQ_N^{n+1}$, $\phi \in \rL^2(\Gamma)$ and $\bfphi \in \rL^2(\Gamma)^3$, they satisfy the following weak formulation:
\begin{equation}\label{eq:weak form coupled semidiscr probl}
    \begin{aligned}
        &\int_{\Omf} \Bigl(1 + \frac{\omega_N^n}{R}\Bigr) \dbfu_N^{n+1} \cdot \bfv \srd \bfx + 2 \nu \int_{\Omf} \Bigl(1 + \frac{\omega_N^n}{R}\Bigr)\bfD_\rf^{\omega_N^n}(\bfu_N^{n+1}) : \bfD_\rf^{\omega_N^n}(\bfv) \srd \bfx\\
        &+\int_{\Gamma} \Bigl(\frac{1}{2}\bfu_N^{n+1} \cdot \bfu_N^n - p_N^{n+1}\Bigr)(\bfpsi - \bfv) \cdot \bfn^{\omega_N^n} \srd S\\
        &+ \frac{1}{2}\int_{\Omf} \Bigl(1 + \frac{\omega_N^n}{R}\Bigr) \Bigl[\Bigl(\Bigl(\Bigl(\bfu_N^n - \zeta_N^{n+\frac{1}{2}}\frac{R+z}{R}\bfe_3\Bigr) \cdot \nabla_\rf^{\omega_N^n}\Bigr)\bfu_N^{n+1}\Bigr) \cdot \bfv\\
        &\qquad -\Bigl(\Bigl(\Bigl(\bfu_N^n - \zeta_N^{n+\frac{1}{2}}\frac{R+z}{R}\bfe_3\Bigr) \cdot \nabla_\rf^{\omega_N^n}\Bigr)\bfv\Bigr) \cdot \bfu_N^{n+1}\Bigr] \srd \bfx + \frac{1}{2R} \int_{\Omf} \zeta_N^{n+\frac{1}{2}} \bfu_N^{n+1} \cdot \bfv \srd \bfx\\
        &+ \frac{1}{2}\int_{\Gamma} (\bfu_N^{n+1} - \dbfeta_N^{n+1}) \cdot \bfn^{\omega_N^n} (\bfu_N^n \cdot \bfv) \srd S +\beta\sum_{i=1}^2 \int_{\Gamma} \cJ_\Gamma^{\omega_N^n}(\dbfeta_N^{n+1} - \bfu_N^{n+1}) \cdot \bftau_i^{\omega_N^n}(\bfpsi - \bfv) \cdot \bftau_i^{\omega_N^n} \srd S\\
        &+\rhob\int_{\Omb} \Bigl(\frac{\dbfeta_N^{n+1}-\dbfeta_N^n}{\Delta t}\Bigr) \cdot \bfpsi \srd \bfx + \rhop\int_{\Gamma} \Bigl(\frac{\zeta_N^{n+1}-\zeta_N^{n}}{\Delta t}\Bigr) \varphi \srd S + 2 \mue \int_{\Omb} \bfD(\bfeta_N^{n+1}) : \bfD(\bfpsi) \srd \bfx\\
        &+ \lambde \int_{\Omb} (\nabla \cdot \bfeta_N^{n+1})(\nabla \cdot \bfpsi) \srd \bfx + 2 \muv \int_{\Omb} \bfD(\dbfeta_N^{n+1}) : \bfD(\bfpsi) \srd \bfx + \lambdv \int_{\Omb} (\nabla \cdot \dbfeta_N^{n+1})(\nabla \cdot \bfpsi) \srd \bfx\\
        &- \alpha \int_{\Omb} \cJ_\rb^{(\bfeta_N^n)^\delta} p_N^{n+1} \nabla_\rb^{(\bfeta_N^n)^\delta} \cdot \bfpsi \srd \bfx + c_0 \int_{\Omb} \frac{p_N^{n+1}-p_N^n}{\Delta t} r \srd \bfx - \alpha \int_{\Omb} \cJ_\rb^{(\bfeta_N^n)^\delta} \dbfeta_N^{n+1} \cdot \nabla_\rb^{(\bfeta_N^n)^\delta} r \srd \bfx\\
        &- \alpha \int_{\Gamma} (\dbfeta_N^{n+1} \cdot \bfn^{(\bfeta_N^n)^\delta}) r \srd S + \kappa \int_{\Omb} \cJ_\rb^{(\bfeta_N^n)^\delta} \nabla_\rb^{(\bfeta_N^n)^\delta} p_N^{n+1} \cdot \nabla_\rb^{(\bfeta_N^n)^\delta} r \srd \bfx\\
        &- \int_{\Gamma} \bigl((\bfu_N^{n+1} - \dbfeta_N^{n+1}) \cdot \bfn^{\omega_N^n}\bigr) r \srd S+\int_\Gamma \Delta \omega_N^{n+\frac{1}{2}} \cdot \Delta \varphi \srd S = 0 \taswellas
    \end{aligned}
\end{equation}
\begin{equation}\label{eq:ids coupled probl}
    \int_\Gamma \Bigl(\frac{\omega_N^{n+\frac{1}{2}}-\omega_N^{n-\frac{1}{2}}}{\Delta t}\Bigr) \phi \srd S = \int_\Gamma \zeta_N^{n+\frac{1}{2}} \phi \srd S \tand \int_\Gamma \Bigl(\frac{\bfeta_N^{n+1}-\bfeta_N^n}{\Delta t}\Bigr) \cdot \bfphi \srd S = \int_\Gamma \zeta_N^{n+1} \bfe_3 \cdot \bfphi \srd S.
\end{equation}
\end{defn}

By recalling that $\bfu_N^{n+\frac{1}{2}} = \bfu_N^n$, $\bfeta_N^{n+\frac{1}{2}} = \bfeta_N^n$, $p_N^{n+\frac{1}{2}} = p_N^n$ and $\omega_N^{n+1} = \omega_N^{n+\frac{1}{2}}$, and introducing the following notation for discrete total energy $E_N^{n+\frac{i}{2}}$ and  discrete dissipation $D_N^{n+1}$ for $i=1,2$:
\begin{equation*}
    \begin{aligned}
        E_N^{n+\frac{i}{2}}
        &\coloneqq \frac{1}{2} \int_{\Omf} \Bigl(1+\frac{\omega_N^{n}}{R}\Bigr) \bigl|\bfu_N^{n+\frac{i}{2}}\bigr|^2 \srd \bfx + \frac{\rhob}{2} \int_{\Omb} \bigl|\dbfeta_N^{n+\frac{i}{2}} \bigr|^2 \srd \bfx + \frac{\rhop}{2} \int_\Gamma \bigl|\zeta_N^{n+\frac{i}{2}}\bigr|^2 \srd S\\
        &\quad + \frac{c_0}{2} \int_{\Omb} \bigl|p_N^{n+\frac{i}{2}}\bigr|^2 \srd \bfx + \mue \int_{\Omb} |\bfD(\bfeta_N^{n+\frac{i}{2}})|^2 \srd \bfx + \frac{\lambde}{2} \int_{\Omb} |\nabla \cdot \bfeta_N^{n+\frac{i}{2}}|^2 \srd \bfx + \frac{1}{2}\int_\Gamma |\Delta \omega_N^{n+\frac{i}{2}}|^2 \srd S,
    \end{aligned}
\end{equation*}
\begin{equation*}
    \begin{aligned}
        D_N^{n+1}
        &\coloneqq 2 \nu (\Delta t)\int_{\Omf} \Bigl(1 + \frac{\omega_N^n}{R}\Bigr)\bigl|\bfD_\rf^{\omega_N^n}(\bfu_N^{n+1})\bigr|^2 \srd \bfx + 2\muv (\Delta t) \int_{\Omb} |\bfD(\dbfeta_N^{n+1})|^2 \srd \bfx + \lambdv (\Delta t) \int_{\Omb} |\nabla \cdot \dbfeta_N^{n+1}|^2 \srd \bfx\\
        &\quad + \kappa (\Delta t) \int_{\Omb} \cJ_\rb^{(\bfeta_N^n)^\delta} \bigl|\nabla_\rb^{(\bfeta_N^n)^\delta} p_N^{n+1}\bigr|^2 \srd \bfx + \beta(\Delta t) \sum_{i=1}^2 \int_{\Gamma} \cJ_\Gamma^{\omega_N^n} \bigl|(\dbfeta_N^{n+1} - \bfu_N^{n+1}) \cdot \bftau_i^{\omega_N^n}\bigr|^2 \srd S,
    \end{aligned}
\end{equation*}
we can derive the following energy estimate:

\begin{lem}\label{lem:discr energy eqs coupled probl}
A weak solution to the semidiscretized coupled problem in \autoref{def:weak sol to coupled problem} satisfies
\begin{equation*}
    E_N^{n+\frac{1}{2}} + \frac{\rhop}{2} \int_\Gamma \bigl|\zeta_N^{n+\frac{1}{2}} - \zeta_N^n\bigr|^2 \srd S + \frac{1}{2} \int_\Gamma \bigl|\Delta (\omega_N^{n+\frac{1}{2}} - \omega_N^{n-\frac{1}{2}})\bigr|^2 \srd S = E_N^n \tand
\end{equation*}
% and
\begin{equation*}
    \begin{aligned}
        &E_N^{n+1} + D_N^{n+1} + \frac{1}{2} \int_{\Omf} \Bigl(1 + \frac{\omega_N^n}{R}\Bigr) \bigl|\bfu_N^{n+1} - \bfu_N^n\bigr|^2 \srd \bfx + \frac{\rhob}{2} \int_{\Omb} \bigl|\dbfeta_N^{n+1} - \dbfeta_N^{n} \bigr|^2 \srd \bfx + \frac{\rhop}{2} \int_\Gamma \bigl|\zeta_N^{n+1} - \zeta_N^{n}\bigr|^2 \srd S\\
        &\quad + \frac{c_0}{2} \int_{\Omb} \bigl|p_N^{n+1}-p_N^n\bigr|^2 \srd \bfx + \mue \int_{\Omb} |\bfD(\bfeta_N^{n+1} - \bfeta_N^n)|^2 \srd \bfx + \frac{\lambde}{2} \int_{\Omb} |\nabla \cdot (\bfeta_N^{n+1} - \bfeta_N^{n})|^2 \srd \bfx = E_N^{n+\frac{1}{2}}.
    \end{aligned}
\end{equation*}
\end{lem}
\begin{remark}
The terms in the energy equality above that are not included in the definitions of the total energy $E_N^{n+\frac{i}{2}}$ and dissipation $D_N^{n+1}$, are due to numerical dissipation. 
\end{remark}
This energy estimate serves as the cornerstone for the uniform bounds established in the next section, as it implies that both $E_N^{n+\frac{i}{2}}$ and $\sum_{n=1}^N D_N^n$ are uniformly bounded by a constant $C$ independent of $n$ and $N$. Recall that $E_N^{n+\frac{i}{2}}$ and $D_N^n$ represent the {\emph{time-discrete}} total energy and dissipation, respectively. However, our ultimate goal is to obtain uniform estimates for the approximate {\emph{solutions}} themselves. Approximate solutions are functions of both space and time, not only space, given at discrete time points. 
To that end, we introduce the corresponding approximate solutions in the next section, and derive uniform boundedness estimates and weak convergence results that will be useful when passing to the limit, as the time-discretization parameter tends to zero.

\section{Approximate solutions: uniform boundedness and weak convergence results}\label{sec:approx sols on the complete time int}

In this section, we extend the approximations defined at a discrete number of time point via the splitting scheme presented in \autoref{sec:splitting scheme},
 to approximate solutions defined on the entire time interval $(0,T)$. 

First, for the fluid velocity $\bfu_N$, the Biot displacement $\bfeta_N$, the pore pressure $p_N$, the plate displacement~$\omega_N$ as well as the plate velocity $\zeta_N$, we introduce piecewise constant approximations: for $t \in \bigl((n-1)\Delta t,n \Delta t\bigr]$, we set
\begin{equation}\label{eq:pw const approx}
    \bfu_N(t) = \bfu_N^n, \enspace \bfeta_N(t) = \bfeta_N^n, \enspace p_N(t) = p_N^n, \enspace \omega_N(t) = \omega_N^{n-\frac{1}{2}}, \enspace \zeta_N(t) = \zeta_N^{n-\frac{1}{2}} \tand \zeta_N^*(t) = \zeta_N^n.
\end{equation}
These piecewise constant approximations have a disadvantage that they possibly exhibit jump discontinuities.
Thus, in order to pave the way for higher regularity in time, and for the limit passage in \autoref{sec:limit passage}, we also employ linear interpolations for the fluid velocity $\bfu_N$, for the Biot displacement $\bfeta_N$, the pore pressure $p_N$
and the plate displacement~$\omega_N$, i.e., we introduce
\begin{equation*}
    \bbfu_N(n \Delta t) = \bfu_N^n, \enspace \bbfeta_N(n \Delta t) = \bfeta_N^n, \enspace \bp_N(n \Delta t) = p_N^n \tand \bomega_N(n \Delta t) = \omega_N^{n-\frac{1}{2}}, \tfor n=0,1,\dots,N.
\end{equation*}
In the above, we formally set $\omega_N^{-\frac{1}{2}} = \omega_0$.
By construction of the splitting scheme, see also \eqref{eq:ids coupled probl}, for $t \in (n\Delta t,(n+1) \Delta t]$, we observe that
\begin{equation}\label{eq:repr of bomega & bbfeta}
    \begin{aligned}
        \dt \bomega_N(t) 
        &= \frac{\omega_N^{n+\frac{1}{2}}-\omega_N^{n-\frac{1}{2}}}{\Delta t} = \zeta_N^{n+\frac{1}{2}} = \zeta_N(t), \enspace \dt \bbfeta_N(t) = \frac{\bfeta_N^{n+1}-\bfeta_N^n}{\Delta t} = \dbfeta_N^{n+1} \tand\\
        \left.\dt \bbfeta_N(t) \right|_\Gamma 
        &= \left.\frac{\bfeta_N^{n+1}-\bfeta_N^n}{\Delta t}\right|_\Gamma = \zeta_N^n \bfe_3 = \zeta_N^*(t) \bfe_3.
    \end{aligned}
\end{equation}
Similar relations are also valid for the fluid velocity $\bbfu_N$ as well as the fluid pore pressure $\bp_N$.

In order to derive uniform boundedness of approximate solutions, we make the following a priori assumptions, which will be later justified by showing that they indeed hold on a sufficiently small time interval $[0,T]$ for $T>0$, as shown in \autoref{lem:verification of assumptions} below:

\bigskip
\noindent{\bf{A priori assumptions on uniform boundedness of geometric quantities:}}
\begin{enumerate}
    \item{\bf{ Assumption~1.B (Uniform boundedness of displacements):}}
    There is a constant $R_{\max} > 0$ so that for all $N$, we have $|\omega_N^{n-\frac{1}{2}}| \le R_{\max} < R$, for all $n=0,1,\dots,N$, and $\left|\left.(\bfeta_N^n)^\delta \right|_\Gamma\right| \le R_{\max} < R$ for all $n=0,1,\dots,N$.
    \item{\bf{Assumption~2.B (Uniform invertibility of the Lagrangian map):}}
    There exists a constant $c_0 > 0$ such that for all $N$, we have $\det\bigl(\bfI + \nabla (\bfeta_N^n)^\delta\bigr) \ge c_0 > 0$ for all $n=0,1,\dots,N$.
    \item{\bf{Assumption~2.C (Uniform boundedness of the Lagrangian map):}}
    There are constants $c_1$, $c_2 > 0$ such that for all $N$, we obtain $|(\bfI + \nabla (\bfeta_N^n)^\delta)^{-1}| \le c_1$ and $|\bfI + \nabla (\bfeta_N^n)^\delta| \le c_2$, for all $n=0,1,\dots,N$.
\end{enumerate}

\begin{lem}\label{lem:unif bddness of approx sols}{\bf{(Uniform boundedness)}}
Suppose that Assumptions 1.B, 2.B, and 2.C hold. 
Then, for all $N>0$, the approximate functions satisfy the following uniform boundedness properties:
    \begin{itemize}
        \item the fluid velocity $\bfu_N$ is uniformly bounded in $\rL^\infty(0,T;\rL^2(\Omf)^3)$ and $\rL^2(0,T;\rW^{1,p}(\Omf)^3)$,  for all~$p \in (1,2)$,
        \item the Biot displacement $\bfeta_N$ is uniformly bounded in $\rL^\infty(0,T;\rH^1(\Omb)^3)$,
        \item the pore pressure $p_N$ is uniformly bounded in $\rL^\infty(0,T;\rL^2(\Omb)) \cap \rL^2(0,T;\rH^1(\Omb))$, and
        \item the plate displacement $\omega_N$ is uniformly bounded in $\rL^\infty(0,T;\rH_0^2(\Gamma))$.
    \end{itemize}
    Moreover, linear interpolations satisfy the following uniform boundedness properties:
    \begin{itemize}
        \item the Biot displacement $\bbfeta_N$ is uniformly bounded in $\rW^{1,\infty}(0,T;\rL^2(\Omb)^3)$, and 
        \item the plate displacement $\bomega_N$ is uniformly bounded in $\rW^{1,\infty}(0,T;\rL^2(\Gamma))$.
    \end{itemize}
\end{lem}

Before proving \autoref{lem:unif bddness of approx sols}, we first remark that because we are working in 3D, the fluid velocity is uniformly bounded only in $\rL^2(0,T,\rW^{1,p}(\Omf))$, for $p\in (1,2)$, and not in $\rL^2(0,T,\rH^1(\Omf))$, which was the case for the 2D problem.
This has important implications for the compactness arguments and the limit passage below.

Secondly, we remark that the assumptions in \autoref{lem:unif bddness of approx sols} imply the two assumptions in \autoref{lem:ex of weak sol to fluid & Biot subprobl}. This is addressed more precisely in the following. 

\begin{remark}\label{rem:inv of Lagrangian map}
 The assumptions in \autoref{lem:unif bddness of approx sols} imply the two assumptions in \autoref{lem:ex of weak sol to fluid & Biot subprobl}:
\begin{enumerate}
 \item Assumption~1.B in \autoref{lem:unif bddness of approx sols} implies Assumption~1.A in \autoref{lem:ex of weak sol to fluid & Biot subprobl}. 
\item Assumptions~1.B and 2.B in \autoref{lem:unif bddness of approx sols} imply Assumption~2.A in \autoref{lem:ex of weak sol to fluid & Biot subprobl}. 
\\
Indeed, the following result from \cite[Theorem~5.5-2]{Cia:88} will help us prove this statement by providing a criterion for injectivity of a given map in terms of the existence of another injective map that coincides with the given map on the boundary:
\begin{lem}\label{lem:inj crit cont map}
Consider a bounded Lipschitz domain $\Omega$, and assume that $\bfPhi \colon \oOmega \to \bR^n$ is continuous as a map from $\oOmega$ to $\bR^n$, $\rC^1$ as a map from $\Omega$ to $\bR^n$ and fulfills the pointwise condition $\det(\nabla \bfPhi) > 0$.
If there is an injective continuous map $\bfPhi_0 \colon \oOmega \to \bR^n$ such that $\left. \bfPhi \right|_{\dOmega} = \left. \bfPhi_0 \right|_{\dOmega}$, then $\bfPhi \colon \oOmega \to \bR^n$ is an injective homeomorphism with $\bfPhi(\oOmega) = \bfPhi_0(\oOmega)$, and $\bfPhi \colon \Omega \to \bR^n$ is an injective diffeomorphism with~$\bfPhi(\Omega) = \bfPhi_0(\Omega)$.
\end{lem}
Indeed, Assumption~2.B implies that $\det(\bfI + \nabla (\bfeta_N^n)^\delta) > 0$, which is one of the requirements in \autoref{lem:inj crit cont map}. Furthermore, $\bfI+ \nabla (\bfeta_N^n)^\delta$ possesses all the necessary regularity properties from \autoref{lem:inj crit cont map}. Therefore, by \autoref{lem:inj crit cont map}, it suffices to construct an injective mapping $\bfvphi_0 \colon \oOmega_\rb \to \bR^3$ such that $\Id + (\bfeta_N^n)^\delta = \bfvphi_0$ on $\del \Omb$.

For this, we can use a standard ALE mapping, e.g., $\bfvphi_0(x,y,z) = (x, y, z + (1 - \frac{z}{R}) (\bfeta_N^n)^\delta(x, y, 0))$, which is injective by Assumption~1.B. Therefore, Assumption~2.A in \autoref{lem:ex of weak sol to fluid & Biot subprobl} is satisfied.
As a result of \autoref{lem:inj crit cont map}, we also get that $(\Id + (\bfeta_N^n)^\delta)(\oOmega_\rb) = \bfvphi_0(\oOmega_\rb)$.
In other words, the deformed configuration is fully determined by the behavior on the boundary.

\item Assumption~3.A in \autoref{lem:ex of weak sol to fluid & Biot subprobl} corresponds precisely to Assumption~2.B in \autoref{lem:unif bddness of approx sols}.
\end{enumerate}
\end{remark}

Next, we prove \autoref{lem:unif bddness of approx sols}.

\begin{proof}[Proof of~\autoref{lem:unif bddness of approx sols}]
The main ingredients here are the uniform energy estimates obtained in \autoref{lem:discr energy eqs coupled probl}.
As a preparation, upon recalling the conventions from \autoref{sec:splitting scheme}, we define the initial energy $E_0$ by
\begin{equation*}
    \begin{aligned}
        E_0
        &\coloneqq \frac{1}{2} \int_{\Omf} \Bigl(1+\frac{\omega_0}{R}\Bigr) |\bfu_0|^2 \srd \bfx + \frac{\rhob}{2} \int_{\Omb} |\bfxi_0|^2 \srd \bfx + \frac{\rhop}{2} \int_\Gamma |\zeta_0|^2 \srd S + \frac{c_0}{2} \int_{\Omb} |p_0|^2 \srd \bfx\\
        &\quad + \mue \int_{\Omb} |\bfD(\bfeta_0)|^2 \srd \bfx + \frac{\lambde}{2} \int_{\Omb} |\nabla \cdot \bfeta_0|^2 \srd \bfx + \frac{1}{2}\int_\Gamma |\Delta \omega_0|^2 \srd S,
    \end{aligned}
\end{equation*}
indicating that $E_0$ is independent of the discretization parameter $N$.
Next, concatenating the first part of the assertion of \autoref{lem:discr energy eqs coupled probl} with the second one, for $E_N^n$, $E_N^{n+1}$ as well as $D_N^{n+1}$, we find that
\begin{equation*}
    \begin{aligned}
        &E_N^{n+1} + D_N^{n+1} + \frac{1}{2} \int_{\Omf} \Bigl(1 + \frac{\omega_N^n}{R}\Bigr) \bigl|\bfu_N^{n+1} - \bfu_N^n\bigr|^2 \srd \bfx + \frac{\rhob}{2} \int_{\Omb} \bigl|\dbfeta_N^{n+1} - \dbfeta_N^{n} \bigr|^2 \srd \bfx + \frac{\rhop}{2} \int_\Gamma \bigl|\zeta_N^{n+1} - \zeta_N^{n}\bigr|^2 \srd S\\
        &\quad + \frac{\rhop}{2} \int_\Gamma \bigl|\zeta_N^{n+\frac{1}{2}} - \zeta_N^{n}\bigr|^2 \srd S + \frac{c_0}{2} \int_{\Omb} \bigl|p_N^{n+1}-p_N^n\bigr|^2 \srd \bfx + \frac{1}{2}\int_\Gamma |\Delta (\omega_N^{n+\frac{1}{2}} - \omega_N^{n-\frac{1}{2}})|^2 \srd S\\
        &\quad + \mue \int_{\Omb} |\bfD(\bfeta_N^{n+1} - \bfeta_N^n)|^2 \srd \bfx + \frac{\lambde}{2} \int_{\Omb} |\nabla \cdot (\bfeta_N^{n+1} - \bfeta_N^{n})|^2 \srd \bfx = E_N^{n}.
    \end{aligned}
\end{equation*}
Summing from $0$ to $n$, and using the cancellations, we first find that
\begin{equation}\label{eq:unif energy ineq}
    E_N^{n+1} + \sum_{j=0}^{n} D_N^{j+1} \le E_0
\end{equation}
for all $n=0,1,\dots,N-1$.
Similarly, we obtain that
\begin{equation}\label{eq:num dissip ests}
    \begin{aligned}
        &\sum_{n=0}^{N-1} \biggl(\frac{1}{2} \int_{\Omf} \Bigl(1 + \frac{\omega_N^n}{R}\Bigr) \bigl|\bfu_N^{n+1} - \bfu_N^n\bigr|^2 \srd \bfx + \frac{\rhob}{2} \int_{\Omb} \bigl|\dbfeta_N^{n+1} - \dbfeta_N^{n} \bigr|^2 \srd \bfx + \frac{\rhop}{2} \int_\Gamma \bigl|\zeta_N^{n+1} - \zeta_N^{n}\bigr|^2 \srd S\\
        &\quad + \frac{\rhop}{2} \int_\Gamma \bigl|\zeta_N^{n+\frac{1}{2}} - \zeta_N^{n}\bigr|^2 \srd S + \frac{c_0}{2} \int_{\Omb} \bigl|p_N^{n+1}-p_N^n\bigr|^2 \srd \bfx + \frac{1}{2}\int_\Gamma |\Delta (\omega_N^{n+\frac{1}{2}} - \omega_N^{n-\frac{1}{2}})|^2 \srd S\\
        &\quad + \mue \int_{\Omb} |\bfD(\bfeta_N^{n+1} - \bfeta_N^n)|^2 \srd \bfx + \frac{\lambde}{2} \int_{\Omb} |\nabla \cdot (\bfeta_N^{n+1} - \bfeta_N^{n})|^2 \srd \bfx\biggr) \le E_0.
    \end{aligned}
\end{equation}
To prove uniform boundedness of the fluid velocity we recall Assumption~1.B, which implies the existence of a constant $\Tilde{C} > 0$ such that for all $n= 0,1,\dots,N-1$ and all~$N$, we have $1 + \frac{\omega_N^n}{R} \ge 1 - \frac{R_{\max}}{R} \ge \Tilde{C}$.
Let now $t \in (0,T)$ be arbitrary.
Then there is $n \in \{0,1,\dots,N-1\}$ with $t \in (n \Delta t,(n+1) \Delta t]$, so it follows that $\bfu_N(t) = \bfu_N^{n+1}$.
Making use of the preceding uniform lower bound of $1 + \frac{\omega_N^n}{R}$, recalling the definition of~$E_N^{n+1}$ and employing \eqref{eq:unif energy ineq}, we find that
\begin{equation}\label{eq:unif bddness u_N in Linfty L2}
    \| \bfu_N(t) \|_{\rL^2(\Omf)}^2 = \| \bfu_N^{n+1} \|_{\rL^2(\Omf)}^2 = \int_{\Omf} |\bfu_N^{n+1}|^2 \srd \bfx \le C \cdot \int_{\Omf} \Bigl(1+\frac{\omega_N^n}{R}\Bigr) \bigl|\bfu_N^{n+1}\bigr|^2 \srd \bfx \le C E_0,
\end{equation}
where $C > 0$ is independent of $N$.
This shows the uniform boundedness of $\bfu_N$ in $\rL^\infty(0,T;\rL^2(\Omf)^3)$.

The uniform boundedness with higher spatial regularity is not so straightforward:
The uniform energy estimates from \autoref{lem:discr energy eqs coupled probl} only provide information about the {\em transformed symmetric part of the gradient}.
To deduce properties of the gradient, we would like to employ Korn's inequality, which now needs to be applied to the moving domain case.
This is a problem since the Korn constant generally depends on the underlying domain.
As pointed out in \autoref{sec:def of a weak sol}, the family of domains in only in $\rC^{0,\alpha}$ for all $\alpha \in (0,1)$, so standard approaches fall short here. 
We note that this is another place where the 2D and 3D cases differ.
However, Lengeler showed in \cite[Proposition~2.9]{Len:14} a Korn-type inequality for non-Lipschitz domains, where the Korn constant only depends on the bound of the structure displacement.
The lower regularity of the domain merely yields the boundedness for lower integrability parameters though.
We shall make these considerations precise in the sequel.
There are some further preparations required for this.
First, similarly as in \eqref{eq:repr of gradient on moving fluid dom}, for $f$ denoting a function defined on the fixed domain, and recalling the ALE mapping from \autoref{ssec:maps from ref to phys dom}, we find that
\begin{equation}\label{eq:grad and transformed grad for same fct}
    \nabla f = (\nabla_\rf^{\omega_N^n} f) [\nabla \bfPhi_\rf^{\omega_N^n}].
\end{equation}
Moreover, note that
\begin{equation*}
    \nabla \bfPhi_\rf^{\omega_N^n} = \begin{pmatrix}
        1 & 0 & (1+\frac{z}{R}) \del_x \omega_N^n\\
        0 & 1 & (1+\frac{z}{R}) \del_y \omega_N^n\\
        0 & 0 & 1+\frac{\omega_N^n}{R}
    \end{pmatrix}.
\end{equation*}
From \eqref{eq:unif energy ineq}, it follows that for given $T > 0$, the family $\omega_N^n$ is uniformly bounded in $\rH_0^2(\Gamma)$.
Thus, thanks to the embedding $\rH^2(\Gamma) \hookrightarrow \rW^{1,s}(\Gamma)$ for all $s < \infty$, we find that this family is also uniformly bounded in~$\rW^{1,s}(\Gamma)$.
Together with Assumption~1.B, this results in
\begin{equation}\label{eq:unif bddness of grad of ALE}
    \| \nabla \bfPhi_\rf^{\omega_N^n} \|_{\rL^s(\Omega_\rf)} \le C \bigl(\| \omega_N^n \|_{\rW^{1,s}(\Gamma)} + \| \omega_N^n \|_{\rL^\infty(\Gamma)} + 1\bigr) \le C,
\end{equation}
where the (generic) constant $C > 0$ is independent of $n$ and $N$.
Let now $p \in (1,2)$ be arbitrary.
We then find $r \in (p,2)$ and $r' \in (2,\infty)$ such that $\nicefrac{1}{r} + \nicefrac{1}{r'} = \nicefrac{1}{p}$.
Making use of \eqref{eq:grad and transformed grad for same fct}, H\"older's inequality together with \eqref{eq:unif bddness of grad of ALE}, Assumption~1.B, which also yields that $\cJ_\rf^{\omega_N^n} = 1 + \frac{\omega_N^n}{R}$ is bounded from below, and \cite[Proposition~2.9]{Len:14}, and employing the notation $\Tilde{\bfu}_N^{n+1} = \bfu_N^{n+1} \circ (\bfPhi_\rf^{\omega_N^n})^{-1}$, we find that
\begin{equation}\label{eq:est of grad of u}
    \begin{aligned}
        \left(\int_{\Omf}|\nabla \bfu_N^{n+1}|^p \srd\bfx\right)^{\frac{2}{p}}
        &= \left(\int_{\Omf}\Bigl|\nabla_\rf^{\omega_N^n} u_N^{n+1} \bigl[\nabla \Phi_\rf^{\omega_N^n}\bigr]\Bigr|^p \srd\bfx\right)^{\frac{2}{p}}
        \le C \left(\int_{\Omf} \bigl|\nabla_\rf^{\omega_N^n} \bfu_N^{n+1}\bigr|^r \srd \bfx \right)^{\frac{2}{r}}\\
        &\le C \left(\int_{\Omf} \cJ_\rf^{\omega_N^n} \bigl|\nabla_\rf^{\omega_N^n} \bfu_N^{n+1}\bigr|^r \srd \bfx \right)^{\frac{2}{r}}
        = C \left(\int_{\Omf(t)} \bigl|\nabla\Tilde{\bfu}_N^{n+1}\bigr|^r \srd \bfx \right)^{\frac{2}{r}}\\
        &\le C\bigl(\| \bfD(\Tilde{\bfu}_N^{n+1}) \|_{\rL^2(\Omf(t))}^2 + \| \Tilde{\bfu}_N^{n+1} \|_{\rL^2(\Omf(t))}^2\bigr)\\
        &= C\left(\int_{\Omf} \cJ_\rf^{\omega_N^n} \bigl|\bfD_\rf^{\omega_N^n}(\bfu_N^{n+1})\bigr|^2 \srd \bfx + \int_{\Omf} \cJ_\rf^{\omega_N^n} |\bfu_N^{n+1}|^2 \srd \bfx\right).
    \end{aligned}
\end{equation}
From \eqref{eq:est of grad of u} and the observation that \cite[Proposition~2.9]{Len:14} yields that the Korn inequality only depends on bounds of $\omega_N^n$ in $\rH_0^2(\Gamma)$ and $\rL^\infty(\Gamma)$, which are in turn uniform in $n$ and $N$, it now follows that
\begin{equation*}
    \begin{aligned}
        \int_0^T \| \nabla \bfu_N(t) \|_{\rL^p(\Omf)}^2 \srd t
        &= (\Delta t) \sum_{n=0}^{N-1} \| \nabla \bfu_N^{n+1} \|_{\rL^p(\Omf)}^2
        = (\Delta t) \sum_{n=0}^{N-1} \left(\int_{\Omf}|\nabla \bfu_N^{n+1}|^p \srd\bfx\right)^{\frac{2}{p}}\\
        &\le C (\Delta t) \sum_{n=0}^{N-1} \left(\int_{\Omf} \cJ_\rf^{\omega_N^n} \bigl|\bfD_\rf^{\omega_N^n}(\bfu_N^{n+1})\bigr|^2 \srd \bfx + \int_{\Omf} \cJ_\rf^{\omega_N^n} |\bfu_N^{n+1}|^2 \srd \bfx\right)
    \end{aligned}
\end{equation*}
for $C > 0$ independent of $n$ and $N$.
As a result of \eqref{eq:unif energy ineq}, we obtain
\begin{equation*}
    \begin{aligned}
        (\Delta t) \sum_{n=0}^{N-1} \int_{\Omf} \cJ_\rf^{\omega_N^n} \bigl|\bfD_\rf^{\omega_N^n}(\bfu_N^{n+1})\bigr|^2 \srd \bfx \le E_0 \tand (\Delta t) \sum_{n=0}^{N-1} \int_{\Omf} \cJ_\rf^{\omega_N^n} |\bfu_N^{n+1}|^2 \srd \bfx \le (\Delta t) \cdot N \cdot E_0 = T E_0.
    \end{aligned}
\end{equation*}
In total, we find that 
\begin{equation*}
    \int_0^T \| \nabla \bfu_N(t) \|_{\rL^p(\Omf)}^2 \srd t \le C (1+T) E_0,
\end{equation*}
where $C > 0$ is independent of $n$ and $N$.
As we have already verified that the sequence $\bfu_N$ is uniformly bounded in $\rL^\infty(0,T;\rL^2(\Omf)^3) \hookrightarrow \rL^2(0,T;\rL^p(\Omf)^3)$ for all $p < 2$, we infer that $\bfu_N$ is indeed also uniformly bounded in $\rL^2(0,T;\rW^{1,p}(\Omf)^3)$ for all $p < 2$.

Next, we address the uniform boundedness of the Biot displacement $\bfeta_N$ in $\rL^\infty(0,T;\rH^1(\Omb)^3)$, for arbitrary~$t \in (0,T)$.
As above, we argue that there is $n \in \{0,1,\dots,N-1\}$ with $\bfeta_N(t) = \bfeta_N^{n+1}$.
Thus, using Poincar\'e's inequality, Korn's inequality for the Biot domain as asserted in \autoref{lem:Korn's ineq Biot} and~\eqref{eq:unif energy ineq} together with the form of $E_N^{n+1}$ as made precise in \autoref{ssec:weak form & energy eq for coupled semi-discr probl}, we infer that
\begin{equation*}
    \| \bfeta_N(t) \|_{\rH^1(\Omb)}^2 = \| \bfeta_N^{n+1} \|_{\rH^1(\Omb)}^2 \le C \cdot \| \nabla \bfeta_N^{n+1} \|_{\rL^2(\Omb)}^2 \le C \cdot \| \bfD(\bfeta_N^{n+1}) \|_{\rL^2(\Omb)}^2 \le C E_0,
\end{equation*}
where the constant $C > 0$ is again uniform in $N$.

Analogously as in \eqref{eq:unif bddness u_N in Linfty L2}, upon invoking the form of $E_N^{n+1}$, we deduce from \eqref{eq:unif energy ineq} the uniform boundedness of $p_N$ in $\rL^\infty(0,T;\rL^2(\Omb))$.
The same is valid for the uniform boundedness of $\omega_N$ in $\rL^\infty(0,T;\rH_0^2(\Gamma))$.
For the uniform boundedness of the pore pressure $p_N$ in $\rL^2(0,T;\rH^1(\Omb))$, inequality \eqref{eq:unif energy ineq} yields
\begin{equation*}
    \begin{aligned}
        \kappa (\Delta t) \sum_{n=0}^{N-1} \int_{\Omb} \cJ_\rb^{(\bfeta_N^n)^\delta} \bigl|\nabla_\rb^{(\bfeta_N^n)^\delta} p_N^{n+1}\bigr|^2 \srd \bfx \le E_0.
    \end{aligned}
\end{equation*}
To control the gradient of $p_N$ rather than the transformed gradient, we start by invoking the identity $\nabla p_N^{n+1} = \nabla_\rb^{(\bfeta_N^n)^\delta} p_N^{n+1} \cdot (\bfI + \nabla (\bfeta_N^n)^\delta)$.
In conjunction with Assumption~2.B, it then follows that
\begin{equation*}
    (\Delta t) \sum_{n=0}^{N-1} \int_{\Omb} \bigl|\nabla_\rb^{(\bfeta_N^n)^\delta} p_N^{n+1}\bigr|^2 \srd \bfx \le C
\end{equation*}
for some $C > 0$ independent of $N$.
With regard to the identity $\nabla p_N^{n+1} = \nabla_\rb^{(\bfeta_N^n)^\delta} p_N^{n+1} \cdot (\bfI + \nabla (\bfeta_N^n)^\delta)$ as well as Assumption~2.C, which in turn implies that $|\bfI + \nabla (\bfeta_N^n)^\delta| \le c_2$, we then find that
\begin{equation*}
    (\Delta t) \sum_{n=0}^{N-1} \int_{\Omb} \bigl|\nabla p_N^{n+1}\bigr|^2 \srd \bfx \le c_2^2 \cdot (\Delta t) \sum_{n=0}^{N-1} \int_{\Omb} \bigl|\nabla_\rb^{(\bfeta_N^n)^\delta} p_N^{n+1}\bigr|^2 \srd \bfx \le C
\end{equation*}
for $C > 0$ independent of $N$.
Similarly as above for the fluid velocity, for $C$ as just described, this eventually yields the uniform boundedness of the pore pressure $p_N$ in $\rL^2(0,T;\rH^1(\Omb))$ thanks to
\begin{equation*}
    \| p_N \|_{\rL^2(0,T;\rH^1(\Omb))}^2 \le C (\Delta t) \sum_{n=0}^{N-1} \int_{\Omb} \bigl|\nabla p_N^{n+1}\bigr|^2 \srd \bfx \le C.
\end{equation*}

Concerning the remaining uniform boundedness assertions on the uniform boundedness of $\bbfeta_N$ as well as $\bomega_N$, we start by recalling $\dt \bbfeta_N(t) = \dbfeta_N^{n+1}$ and $\dt \bomega_N(t) = \zeta_N^{n+\frac{1}{2}} = \zeta_N(t)$ for some $n \in \{0,1,\dots,N-1\}$.
Moreover, we observe that the uniform boundedness of $\bbfeta_N$ in $\rL^\infty(0,T;\rL^2(\Omb)^3)$ follows in a similar manner as the above uniform boundedness of $\bfeta_N$ in $\rL^\infty(0,T;\rH^1(\Omb)^3)$.
On the other hand, for $t \in (0,T)$, invoking the above relation, and exploiting \eqref{eq:unif energy ineq} along with the form of $E_N^{n+1}$, we find that
\begin{equation*}
    \| \dt \bbfeta_N(t) \|_{\rL^2(\Omb)}^2 = \| \dbfeta_N^{n+1} \|_{\rL^2(\Omb)}^2 \le C E_0,
\end{equation*}
where $C > 0$ is independent of $N$.
The uniform boundedness of $\bomega_N$ in $\rW^{1,\infty}(0,T;\rL^2(\Gamma))$ can be deduced similarly by virtue of \eqref{eq:repr of bomega & bbfeta} and the form of $E_N^{n+1}$.
\end{proof}

The corollary below is an immediate consequence of the uniform boundedness established in \autoref{lem:unif bddness of approx sols}.

\begin{cor}\label{cor:weak & weak* conv}
Suppose that the assumptions stated in \autoref{lem:unif bddness of approx sols} are satisfied.
Then there is a subsequence such that the following weak convergence results hold:
\begin{itemize}
    \item {\bf{Fluid velocity:}} 
    \begin{itemize}
    \item $\bfu_N \rightharpoonup \bfu$ weakly$^*$ in $\rL^\infty(0,T;\rL^2(\Omf)^3)$ and 
    \item $\nabla_\rf^{\tau_{\Delta t}\omega_N} \bfu_N \rightharpoonup \bfG$ weakly in $\rL^2(0,T;\rL^p(\Omf)^{3 \times 3})$ for all $p \in (1,2)$, for some $\bfG$, where $\tau_{\Delta t}$ denotes the time shift by $\Delta t$, i.e., $\tau_{\Delta t} f(t,\cdot) = f(t-\Delta t,\cdot)$ for a function $f$ defined on $(0,T)$.
    \end{itemize}
    \item {\bf{Biot displacement:}} 
    \begin{itemize}
    \item $\bfeta_N \rightharpoonup \bfeta$ weakly$^*$ in $\rL^\infty(0,T;\rH^1(\Omb)^3)$ and 
    \item $\overline{\bfeta}_N \rightharpoonup \overline{\bfeta}$ weakly$^*$ in $\rW^{1,\infty}(0,T;\rL^2(\Omb)^3)$.
    \end{itemize}
    \item {\bf{Biot pressure:}} 
    \begin{itemize}
    \item $p_N \rightharpoonup p$ weakly$^*$ in $\rL^\infty(0,T;\rL^2(\Omb))$ and 
    \item $p_N \rightharpoonup p$ weakly in $\rL^2(0,T;\rH^1(\Omb))$.
    \end{itemize}
    \item {\bf{Plate displacement:}} 
    \begin{itemize}
    \item $\omega_N$ satisfies $\omega_N \rightharpoonup \omega$ weakly$^*$ in $\rL^\infty(0,T;\rH_0^2(\Gamma))$ and
    \item $\overline{\omega}_N \rightharpoonup \overline{\omega}$ weakly$^*$ in~$\rW^{1,\infty}(0,T;\rL^2(\Gamma))$.
    \end{itemize}
\end{itemize}
Additionally, we have $\bfeta = \overline{\bfeta}$ and $\omega = \overline{\omega}$.
\end{cor}

In \autoref{lem:lim of the gradient of the pressure}, we will prove that $\bfG = \nabla_\rf^\omega \bfu$.

\bigskip
Next, we show that the assumptions from \autoref{lem:unif bddness of approx sols} are indeed fulfilled for sufficiently small $T > 0$, and under suitable conditions on $\bfeta_0$.

\begin{lem}\label{lem:verification of assumptions}
Assume that the initial data satisfy $|\omega_0| \le R_0 < R$, and that $\bfeta_0$ is such that $\Id + (\bfeta_0)^\delta$ is invertible with $\det(\bfI + \nabla (\bfeta_0)^\delta) \ge c_0 > 0$ on $\Omb$.
Then there exists $T > 0$ sufficiently small such that for all $N$, Assumptions~1.B, 2.B and~2.C from \autoref{lem:unif bddness of approx sols} are satisfied.
\end{lem}

\begin{proof} The proof of this result is different from the proof of the corresponding result in 2D due to different embedding properties.  

First observe that it is clear that $\omega_0$ is bounded, while the boundedness of $(\bfeta_0)^\delta$ on $\Gamma$ is a consequence of the invertibility of $\Id + (\bfeta_0)^\delta$ and $\det(\bfI + \nabla (\bfeta_0)^\delta) \ge c_0 > 0$.
In particular, thanks the continuity of $|\bfI + \nabla (\bfeta_0)^\delta|$ and $|(\bfI + \nabla (\bfeta_0)^\delta)^{-1}|$ on the compact set $\oOmega_\rb$, we find constants $\alpha_1$, $\alpha_2 > 0$ such that $|\bfI + \nabla (\bfeta_0)^\delta| \le \alpha_1$ and $|(\bfI + \nabla (\bfeta_0)^\delta)^{-1}| \le \alpha_2$.

In the remainder of the proof, we show that there is a sufficiently small time $T > 0$ so that Assumptions~1.B, 2.B and~2.C from \autoref{lem:unif bddness of approx sols} hold uniformly for all $N$ and $n \Delta t$ up to time $T$.
For this purpose, we first invoke the energy equality from \autoref{lem:discr energy eqs coupled probl}.
Similarly as in \eqref{eq:unif energy ineq}, we find that 
\begin{equation*}
    E_N^{k+\frac{1}{2}} \le E_0 \tand E_N^{k+1} \le E_0 \tforall k=0,1,\dots,N-1.
\end{equation*}
Thus, after solving the two subproblems as exposed in \autoref{ssec:plate subprobl} and \autoref{ssec:fluid & Biot subprobl} on the time step $[k \Delta t,(k+1)\Delta t]$, for a constant $C > 0$ merely depending on the initial energy $E_0$, we get
\begin{equation}\label{eq:bddness eta, omega & zeta}
    \begin{aligned}
        \| \dbfeta_N^n \|_{\rL^2(\Omb)}
        &\le C, &&\tfor n=0,1,\dots,k+1,\\
        \| \omega_N^{n+\frac{1}{2}} \|_{\rH_0^2(\Gamma)}
        &\le C, &&\tfor n=0,1,\dots,k, \tand\\
        \| \zeta_N^{n+\frac{1}{2}} \|_{\rL^2(\Gamma)}
        &\le C, &&\tfor 0 \le n + \frac{i}{2} \le k+1 \tand i=0,1.
    \end{aligned}
\end{equation}

The first step is to establish the uniform bound of the plate displacements $\omega_N^{n-\frac{1}{2}}$.
To this end, we recall that $\dt \overline{\omega}_N = \zeta_N$, so the piecewise linear and piecewise constant character of $\overline{\omega}_N$ and $\zeta_N$, respectively, joint with \eqref{eq:bddness eta, omega & zeta}$_2$ and \eqref{eq:bddness eta, omega & zeta}$_3$ yields that there is $C = C(E_0) > 0$ independent of $N$ such that
\begin{equation}\label{eq:bounds for bar omega}
    \begin{aligned}
        \| \overline{\omega}_N \|_{\rW^{1,\infty}(0,(k+1)\Delta t;\rL^2(\Gamma))}
        &\le C \tand
        \| \overline{\omega}_N \|_{\rL^\infty(0,(k+1)\Delta t;\rH_0^2(\Gamma))} \le C.
    \end{aligned}
\end{equation}
Note that \eqref{eq:bounds for bar omega} means that $\overline{\omega}_N$ is Lipschitz-continuous in time with values in $\rL^2(\Gamma)$.
Next, we invoke that for $s \in (1,2)$, by complex interpolation, see e.g., \cite[Section~2.1]{Lun:18}, we have $\rH^s(\Gamma) = [\rL^2(\Gamma),\rH^2(\Gamma)]_{\frac{s}{2}}$.
Consequently, for all~$t$ and $t+\tau \in [0,(k+1) \Delta t]$ with $\tau > 0$, we obtain
\begin{equation*}
    \| \overline{\omega}_N(t+\tau) - \overline{\omega}_N(t) \|_{\rH^s(\Gamma)} \le C \cdot \| \overline{\omega}_N(t+\tau) - \overline{\omega}_N(t) \|_{\rL^2(\Gamma)}^{1-\frac{s}{2}} \cdot \| \overline{\omega}_N(t+\tau) - \overline{\omega}_N(t) \|_{\rH^2(\Gamma)}^{\frac{s}{2}}.
\end{equation*}
Combining the latter inequality with \eqref{eq:bounds for bar omega} and the above remark on the Lipschitz continuity in time, we infer that
\begin{equation}\label{eq:Holder est omega_N}
    \| \overline{\omega}_N(t+\tau) - \overline{\omega}_N(t) \|_{\rH^s(\Gamma)} \le C \cdot \tau^{1-\frac{s}{2}},
\end{equation}
where $C = C(E_0) > 0$ is again independent of $k$ and $N$.
The choice $t = 0$ and $\tau = (k+1) \Delta t$ joint with the embedding $\rH^s(\Gamma) \hookrightarrow \rC(\Gamma)$ thanks to $s > 1$ then leads to
\begin{equation*}
    \| \omega_N^{k+1} - \omega_0 \|_{\rC(\Gamma)} \le C \cdot [(k+1) \Delta t]^{1-\frac{s}{2}} \le C \cdot T^{1-\frac{s}{2}}.
\end{equation*}
As the above constant $C > 0$ only depends on $E_0$, with regard to $|\omega_0| < R$, we may choose $T > 0$ sufficiently small such that $C \cdot T^{1-\frac{s}{2}} < R - \| \omega_0 \|_{\rC(\Gamma)}$, showing that the first part of Assumption~1.B can be guaranteed for $T > 0$ sufficiently small.

What is left to show is the boundedness of the trace of $\bfeta_N^n$.
For this purpose, we employ the triangle inequality, the definition of $\dbfeta_N^n = \frac{\bfeta_N^n - \bfeta_N^{n-1}}{\Delta t}$ as well as \eqref{eq:bddness eta, omega & zeta}$_1$ to get
\begin{equation*}
    \| \bfeta_N^{k+1} - \bfeta_0 \|_{\rL^2(\Omb)} \le (\Delta t) \sum_{n=1}^{k+1} \| \dbfeta_N^n \|_{\rL^2(\Omb)} \le C(k+1) (\Delta t) \le C T,
\end{equation*}
with $C = C(E_0) > 0$ independent of $N$.
On the other hand, the triangle inequality, the construction of the plate subproblem in \autoref{ssec:plate subprobl} and \eqref{eq:bddness eta, omega & zeta}$_3$ imply that
\begin{equation*}
    \| \omega_N^{k+1} - \omega_0 \|_{\rL^2(\Gamma)} \le (\Delta t) \sum_{n=1}^{k+1} \| \zeta_N^{n-\frac{1}{2}} \|_{\rL^2(\Gamma)} \le C (\Delta t) (k+1) \le C T.
\end{equation*}
Concatenating the previous two estimates, and invoking the odd extension from \eqref{eq:odd ext}, we conclude that
\begin{equation*}
    \| \bfeta_N^{k+1} - \bfeta_0 \|_{\rL^2(\tOmb)} \le C\bigl(\| \bfeta_N^{k+1} - \bfeta_0 \|_{\rL^2(\Omb)} + \| \omega_N^{k+1} - \omega_0 \|_{\rL^2(\Gamma)}\bigr) \le C T.
\end{equation*}
Again, let us emphasize that $C = C(E_0) > 0$ is independent of $N$.
Thanks to the regularization, for~$C = C(\delta,E_0) > 0$, we find that
\begin{equation*}
    \| (\bfeta_N^{k+1})^\delta - (\bfeta_0)^\delta \|_{\rH^3(\Omb)} \le C \cdot T.
\end{equation*}
Since the trace is continuous from $\rH^3(\Omb)$ to $\rH^2(\Gamma)$, and thanks to $\rH^2(\Gamma) \hookrightarrow \rC(\Gamma)$ and $\rH^2(\Omb) \hookrightarrow \rC(\Omb)$, for $C = C(E_0,\delta) > 0$ independent of $N$, we deduce from the last estimate that
\begin{equation}\label{eq:bounds for eta}
    \left\| \left.(\bfeta_N^{k+1})^\delta \right|_\Gamma - \left.(\bfeta_0)^\delta \right|_\Gamma \right\|_{\rC(\Gamma)} \le C \cdot T \tand \left\| \nabla(\bfeta_N^{k+1})^\delta - (\bfeta_0)^\delta \right\|_{\rC(\Gamma)} \le C \cdot T.
\end{equation}
At this stage, let us observe that $\det(\bfI + \bfA)$ depends continuously on $\bfA$, and the matrix norms $|\bfI + \bfA|$ and~$|(\bfI + \bfA)^{-1}|$ are continuous functions of $\bfA$.
Moreover, the constant $C$ in \eqref{eq:bounds for eta} merely depends on~$E_0$ and $\delta$ and, in particular, it is independent of $k$ and $N$ (although it may depend on $\delta$).
By virtue of \eqref{eq:bounds for eta}, we may choose $T > 0$ sufficiently small such that the second part of Assumption~1.B and Assumptions~2.B and~2.C are satisfied, completing the proof.
\end{proof}

\section{Approximate solutions: Strong convergence via compactness arguments}\label{sec:comp args}

The weak and weak$^*$ convergences established in \autoref{cor:weak & weak* conv} are insufficient to pass to the limit in the nonlinear terms in the semidiscrete formulation \eqref{eq:weak form coupled semidiscr probl} and \eqref{eq:ids coupled probl}.
Instead, we need to upgrade to strong convergence results.
To that end, we apply the classical Aubin--Lions compactness theorem on fixed domains, together with the generalized Aubin--Lions framework of \cite{MC:19}. We first establish compactness for the plate displacement and for the displacement in the Biot domain, and then address the compactness of the fluid velocity on the moving fluid domain.

Throughout this section, we consider $T > 0$ sufficiently small, determined in \autoref{lem:verification of assumptions}.

\subsection{Compactness for the Biot displacement}
\ 

In this subsection, we verify the strong convergence of the Biot structure displacement $\overline{\bfeta}_N$ by employing the classical Aubin--Lions compactness result.

\begin{lem}\label{lem:strong conv Biot displacement}
We have that $\rW^{1,\infty}(0,T;\rL^2(\Omb)^3) \cap \rL^\infty(0,T;\rH^1(\Omb)^3)\Subset \rC(0,T;\rL^2(\Omb)^3)$.
In particular, we deduce the existence of a subsequence, still denoted by $(\overline{\bfeta}_N)$, such that $\overline{\bfeta}_N \to \bfeta$ strongly in~$\rC(0,T;\rL^2(\Omb)^3)$.
\end{lem}

\begin{proof}
For the compactness of the embedding, we first note that $\rH^1(\Omb) \Subset \rL^2(\Omb)$ by the Rellich--Kondrachov theorem.
Thus, the compactness of the embedding is a result of the classical Aubin--Lions compactness theorem.
For the existence of a convergent subsequence $(\overline{\bfeta}_N)$, we still require the uniform boundedness of $\overline{\bfeta}_N$ in $\rW^{1,\infty}(0,T;\rL^2(\Omb)^3) \cap \rL^\infty(0,T;\rH^1(\Omb)^3)$.
The first part directly follows from \autoref{lem:unif bddness of approx sols}.
In the same lemma, we have also shown that the Biot displacement $\bfeta_N$ is uniformly bounded in $\rL^\infty(0,T;\rH^1(\Omb)^3)$.
The uniform boundedness of $\overline{\bfeta}_N$ can be concluded in the same way.
\end{proof}

\subsection{Compactness for the plate displacement}
\ 

Next, we address the compactness for the plate displacement.
Here we observe that we can only deduce the strong convergence in $\rC(0,T;\rH^s(\Gamma))$ for $s < 2$, owing to the use of the Arzela--Ascoli theorem and the required compactness of the Sobolev embedding, even though the plate displacements are uniformly boundedness in $\rL^\infty(0,T;\rH_0^2(\Gamma))$ by \autoref{lem:unif bddness of approx sols}.
The associated result is given below.

\begin{lem}\label{lem:comp for plate displacement}
Let $s \in (0,2)$.
Then there exist subsequences $\overline{\omega}_N$ and $\omega_N$ such that $\overline{\omega}_N \to \omega$ strongly in~$\rC(0,T;\rH^s(\Gamma))$ and $\omega_N \to \omega$ strongly in $\rL^\infty(0,T;\rH^s(\Gamma))$.
\end{lem}

\begin{proof}
Let us observe that for $s \in (0,2)$, there exists $s' \in (0,2)$ such that $s' = s + \eps$, where $\eps > 0$ is sufficiently small.
From \eqref{eq:Holder est omega_N}, for $\tau > 0$ as well as $t$, $t + \tau \in [0,T]$, we recall that
\begin{equation*}
    \| \overline{\omega}_N(t+\tau) - \overline{\omega}_N(t) \|_{\rH^{s'}(\Gamma)} \le C \tau^{1-\frac{s'}{2}},
\end{equation*}
where we emphasize that the constant $C$ is independent of $N$.
This estimate establishes the equicontinuity of $(\overline{\omega}_N)$.
At the same time, note that $\rH^{s'}(\Gamma) \Subset \rH^s(\Gamma)$ by the choice of $s' = s + \eps$.
On the other hand, the uniform boundedness of $\overline{\omega}_N$ in $\rL^\infty(0,T;\rH^s(\Gamma))$ for $s \in (0,2)$ is especially implied by \autoref{lem:unif bddness of approx sols} upon noting that the proof is analogous to the one of the uniform boundedness of $\omega_N$ in $\rL^\infty(0,T;\rH_0^2(\Gamma))$.
Therefore, the Arzela--Ascoli theorem yields the existence of a subsequence, still denoted by $\overline{\omega}_N$ such that~$\overline{\omega}_N \to \omega$ strongly in~$\rC(0,T;\rH^s(\Gamma))$.

For the second part of the assertion, the aim is to show that $\| \omega_N(t) - \overline{\omega}_N(t) \|_{\rL^\infty(0,T;\rH^s(\Gamma))} \to 0$.
To this end, observe that $\overline{\omega}_N(n \Delta t) = \omega_N(t)$ for $n \Delta t \le t < (n+1) \Delta t$.
Thus, additionally exploiting~\eqref{eq:Holder est omega_N}, we find the existence of $\tau \le \Delta t$ such that
\begin{equation*}
    \| \omega_N(t) - \overline{\omega}_N(t) \|_{\rL^\infty(0,T;\rH^s(\Gamma))} = \| \overline{\omega}_N(t+\tau) - \overline{\omega}_N(t) \|_{\rL^\infty(0,T;\rH^s(\Gamma))} \le C(\Delta t)^{1-\frac{s}{2}} \to 0 \tas N \to \infty,
\end{equation*}
where we invoke again the independence of the constant $C$ of $N$.
This shows that $\omega_N$ and $\overline{\omega}_N$ have the same limit in $\rL^\infty(0,T;\rH^s(\Gamma))$ for $s \in (0,2)$ and thus completes the proof.
\end{proof}

Before proceeding with the compactness for the Biot velocity, the plate velocity, the pore pressure and finally also the fluid velocity, we remark that in the test space $\cQ_N^{n+1}$ from \eqref{eq:sol & test space fluid & Biot subprobl}, the pore pressure and the fluid velocity are decoupled from the Biot and plate velocity, so we may address the compactness arguments for these quantities separately.
In fact, we will decouple the test space into smaller test spaces for the Biot and place displacement and velocity, the pore pressure and the fluid velocity.
Due to the kinematic coupling condition at the plate interface, we must handle the Biot and plate velocity together, which is presented next.

\subsection{Compactness for the Biot and plate velocities}\label{ssec:comp Biot vel & plate vel}
The aim of this subsection is to establish compactness results for the Biot and plate velocities $(\bfxi_N,\zeta_N)$.
More precisely, we intend to show the existence of convergent subsequences in $\rL^2(0,T;\rH^{s}(\Omb)^3 \times \rH^{s}(\Gamma))$ for $s \in (-\frac{1}{2},0)$.
The reasons for considering negative Sobolev spaces are twofold.
First, the main existence result of this work includes the case in which the Biot medium is purely elastic, namely, the viscoelasticity coefficients can be zero, which implies that the Biot velocities can only be $\rL^2(\Omb)^3$.
{{Secondly, in contrast with classical FSI problems with elastic structures and no-slip boundary conditions, where the interface velocity is the same as the trace of the fluid velocity, in the present construction of the splitting scheme for the regularized FPSI problem, the plate velocity $\zeta_N$ generally {\em does not coincide with the traces of the fluid velocity $\bfu_N \in \rH^1(\Omf)^3$}.
As a consequence, we {\em cannot expect higher regularity of the plate velocities} compared to the regularity resulting from the finite-energy space.}}
In other words, the plate velocities can only be expected to lie in $\rL^2(\Gamma)$.

Below, we assert the compactness result for the Biot and plate velocities.

\begin{prop}\label{prop:comp of Biot & plate vel}
Let $s \in (-\frac{1}{2},0)$.
Then there is a subsequence, still denoted by $(\bfxi_N,\zeta_N)$, such that~$(\bfxi_N,\zeta_N) \to (\bfxi,\zeta)$ strongly in $\rL^2(0,T;\rH^{s}(\Omb)^3 \times \rH^{s}(\Gamma))$.
\end{prop}

\begin{proof} The proof of this result is different from the 2D case because of the lower regularity of the plate displacement, and the weaker result on the uniform boundedness of the fluid velocity in $\rL^2(0,T,\rW^{1,p})$ instead of $\rL^2(0,T,\rH^1)$ for the 2D case.

The main idea  is to use a compactness criterion for piecewise constant functions due to Dreher and J\"ungel \cite{DJ:12}.
We introduce the Biot and plate velocity test space~$\cQ_v$, which encodes slightly higher regularity in the Biot component compared to \autoref{def:sol & test space}, and which is given by
\begin{equation*}
    \cQ_v \coloneqq \{(\bfpsi,\varphi) \in (\Vd \cap \rH^2(\Omb)^3) \times \rH_0^2(\Gamma) : \bfpsi = \varphi \bfe_3 \ton \Gamma\}.
\end{equation*}
For $s \in (-\frac{1}{2},0)$, it then follows that $\rL^2(\Omb)^3 \times \rL^2(\Gamma) \Subset \rH^s(\Omb) \times \rH^s(\Gamma) \subset \cQ_v'$.
Let us stress that the first embedding is compact, and this chain of embeddings will be required for the compactness criterion \cite{DJ:12}.

Moreover, we recall the time shift $\tau_{\Delta t}$ given by $\tau_{\Delta t} f(t,\cdot) = f(t-\Delta t,\cdot)$ for a function $f$ defined on~$(0,T)$.
The Dreher--J\"ungel compactness criterion still requires us to show that for $C > 0$ uniform in $N$, we have
\begin{equation}\label{eq:est for DJ crit}
    \left\| \frac{\tau_{\Delta t}(\bfxi_N,\zeta_N) - (\bfxi_N,\zeta_N)}{\Delta t}\right\|_{\rL^1(\Delta t,T;\cQ_v')} + \| (\bfxi_N,\zeta_N) \|_{\rL^\infty(0,T;\rL^2(\Omb) \times \rL^2(\Gamma))} \le C.
\end{equation}

From \eqref{eq:repr of bomega & bbfeta}, it follows that $\dt \bbfeta_N(t) = \frac{\bfeta_N^{n+1}-\bfeta_N^n}{\Delta t} = \dbfeta_N^{n+1} = \bfxi_N(t)$ and $\dt \bomega_N(t) = \zeta_N(t)$.
Therefore, the uniform boundedness of $\bfxi_N$ in $\rL^\infty(0,T;\rL^2(\Omb)^3)$ is a consequence of the uniform boundedness of $\bbfeta_N$ in~$\rW^{1,\infty}(0,T;\rL^2(\Omb)^3)$ as shown in \autoref{lem:unif bddness of approx sols}, while the uniform boundedness of $\zeta_N$ in~$\rL^\infty(0,T;\rL^2(\Gamma))$ follows in the same way from the uniform boundedness result of $\bomega_N$ in \autoref{lem:unif bddness of approx sols}.
This implies that the second term in \eqref{eq:est for DJ crit} is uniformly bounded.

For the first term, we invoke the coupled semidiscrete formulation from \eqref{eq:weak form coupled semidiscr probl} and \eqref{eq:ids coupled probl}, where we choose the test functions $\bfv$ for the fluid velocity and $r$ for the Biot pore pressure equal to zero.
Thus, upon replacing $\dbfeta_N$ by $\bfxi_N$, and using that $\bfxi_N = \zeta_N \bfe_3$ on $\Gamma$, for $(\bfpsi,\varphi) \in \cQ_v$, we obtain
\begin{equation}\label{eq:result of coupled semidiscr form for Biot & plate vel}
    \begin{aligned}
        &\quad \rhob\int_{\Omb} \Bigl(\frac{\bfxi_N^{n+1}-\bfxi_N^n}{\Delta t}\Bigr) \cdot \bfpsi \srd \bfx + \rhop\int_{\Gamma} \Bigl(\frac{\zeta_N^{n+1}-\zeta_N^{n}}{\Delta t}\Bigr) \varphi \srd S\\
        &= -\int_{\Gamma} \Bigl(\frac{1}{2}\bfu_N^{n+1} \cdot \bfu_N^n - p_N^{n+1}\Bigr)(\bfpsi \cdot \bfn^{\omega_N^n}) \srd S - \beta\sum_{i=1}^2 \int_{\Gamma} \cJ_\Gamma^{\omega_N^n} (\zeta_N^{n+1} \bfe_3 - \bfu_N^{n+1}) \cdot \bftau_i^{\omega_N^n}(\bfpsi \cdot \bftau_i^{\omega_N^n}) \srd S\\
        &\quad - 2 \mue \int_{\Omb} \bfD(\bfeta_N^{n+1}) : \bfD(\bfpsi) \srd \bfx - \lambde \int_{\Omb} (\nabla \cdot \bfeta_N^{n+1})(\nabla \cdot \bfpsi) \srd \bfx - 2 \muv \int_{\Omb} \bfD(\bfxi_N^{n+1}) : \bfD(\bfpsi) \srd \bfx\\
        &\quad -\lambdv \int_{\Omb} (\nabla \cdot \bfxi_N^{n+1})(\nabla \cdot \bfpsi) \srd \bfx + \alpha \int_{\Omb} \cJ_\rb^{(\bfeta_N^n)^\delta} p_N^{n+1} \nabla_\rb^{(\bfeta_N^n)^\delta} \cdot \bfpsi \srd \bfx - \int_\Gamma \Delta \omega_N^{n+\frac{1}{2}} \cdot \Delta \varphi \srd S.
    \end{aligned}
\end{equation}
By construction, the estimate of the first term in \eqref{eq:est for DJ crit} is implied provided we manage to estimate the right-hand side in \eqref{eq:result of coupled semidiscr form for Biot & plate vel} by the $\cQ_v'$-norm.
Consequently, let $(\bfpsi,\varphi) \in \cQ_v$ such that $\| (\bfpsi,\varphi) \|_{\cQ_v} \le 1$.
In particular, it follows that $\| \bfpsi \|_{\rH^2(\Omb)} \le 1$ and $\| \varphi \|_{\rH_0^2(\Gamma)} \le 1$.
Now, let us observe that most of the terms on the right-hand side in \eqref{eq:result of coupled semidiscr form for Biot & plate vel} can be controlled easily when taking into account the uniform boundedness results established in \autoref{lem:unif bddness of approx sols} and the underlying uniform energy estimates.
In particular, let us observe that similar uniform boundedness results for $\bfxi_N$ as for $\bfeta_N$ can be derived in the viscoelastic case $\muv$, $\lambdv > 0$ upon recalling the shape of the discrete dissipation $D_N^{n+1}$ from \autoref{ssec:weak form & energy eq for coupled semi-discr probl}.
We will only make precise the treatment of a few selected terms to show the general ideas.
Due to the uniform boundedness of $\bfu_N$ in $\rL^2(0,T;\rW^{1,p}(\Omf)^3)$, choosing $p > \frac{3}{2}$ and making use of classical properties of the trace and Sobolev embeddings, we find that $\bfu_N^{n}$ admits a uniform bound in $\rW^{1-\frac{1}{p},p}(\Gamma)^3 \hookrightarrow \rL^q(\Gamma)^3$ for some $q > 2$.
Likewise, employing the uniform boundedness of $p_N$ in $\rL^2(0,T;\rH^1(\Omb))$, we derive a uniform bound of $p_N^n$ in $\rL^2(\Gamma)$.
On the other hand, for $\bfpsi \in \rH^2(\Omb)^3$, we get $\bfpsi \in \rH^{\frac{3}{2}}(\Gamma)^3 \hookrightarrow \rL^\infty(\Gamma)^3 \cap \rL^2(\Gamma)^3$.
Note that $\bfn^{\omega_N^n}$ can be bounded in $\rL^r(\Gamma)$ by definition of $\bfn^{\omega_N^n}$ as made precise in \eqref{eq:normal & tangential vectors}, the uniform boundedness of $\omega_N$ in $\rL^\infty(0,T;\rH_0^2(\Gamma))$ and the Sobolev embedding $\rH^2(\Gamma) \hookrightarrow \rW^{1,r}(\Gamma)$ for all~$r \in (1,\infty)$.
Putting together these arguments, and using H\"older's inequality, for all $n \in \{0,1,\dots,N-1\}$, we find
\begin{equation*}
    \begin{aligned}
        &\quad \left| \int_{\Gamma} \Bigl(\frac{1}{2}\bfu_N^{n+1} \cdot \bfu_N^n - p_N^{n+1}\Bigr)(\bfpsi \cdot \bfn^{\omega_N^n}) \srd S \right|\\
        &\le C 
        \bigl(\| \bfu_N^{n+1} \|_{\rL^2(\Gamma)} \cdot \| \bfu_N^{n} \|_{\rL^2(\Gamma)} \cdot \| \bfpsi \|_{\rL^\infty(\Gamma)} + \| p_N^{n+1} \|_{\rL^2(\Gamma)} \cdot \| \bfpsi \|_{\rL^2(\Gamma)}\bigr) \le C,
    \end{aligned}
\end{equation*}
where $C > 0$ is independent of $N$ by the above arguments.
The other term on the right-hand side in~\eqref{eq:result of coupled semidiscr form for Biot & plate vel} that we discuss in detail is the term with the factor $\alpha$.
For this, we recall $\cJ_\rb^{(\bfeta_N^n)^\delta} = \det(\bfI + \nabla (\bfeta_N^n)^\delta)$ and $\nabla^{(\bfeta_N^n)^\delta} \cdot \bfpsi = \tr\bigl[\nabla \bfpsi \cdot (\bfI + \nabla (\bfeta_N^n)^\delta)^{-1}\bigr]$.
From Assumption~2.C and $\| \bfpsi \|_{\rH^1(\Omb)} \le 1$, we deduce the uniform boundedness of $\| \nabla^{(\bfeta_N^n)^\delta} \cdot \bfpsi \|_{\rL^2(\Omb)}$.
On the other hand, the boundedness of $\bfeta_N^n$ in $\rH^1(\Omb)^3$ as discussed above yields the uniform boundedness of the Jacobian involving the regularized Biot displacement, i.e., $|\cJ_\rb^{(\bfeta_N^n)^\delta}| \le C$ for $C > 0$ independent of~$N$.
Consequently, additionally invoking the uniform boundedness of $p_N$ in~$\rL^\infty(0,T;\rL^2(\Omb))$, we get $\left|\alpha \int_{\Omb} \cJ_\rb^{(\bfeta_N^n)^\delta} p_N^{n+1} \nabla_\rb^{(\bfeta_N^n)^\delta} \cdot \bfpsi \srd \bfx\right| \le C$.
Putting the various estimates together, we infer the existence of a constant $C > 0$ independent of $n$ and~$N$ such that
\begin{equation*}
    \left\| \frac{(\bfxi_N^{n+1},\zeta_N^{n+1}) - (\bfxi_N^{n},\zeta_N^{n})}{\Delta t}\right\|_{\cQ_v'} \le C. 
\end{equation*}
It then follows that
\begin{equation*}
    \left\| \frac{\tau_{\Delta t}(\bfxi_N,\zeta_N) - (\bfxi_N,\zeta_N)}{\Delta t}\right\|_{\rL^1(\Delta t,T;\cQ_v')} \le (\Delta t) \sum_{n=1}^{N-1} \left\| \frac{(\bfxi_N^{n+1},\zeta_N^{n+1}) - (\bfxi_N^{n},\zeta_N^{n})}{\Delta t}\right\|_{\cQ_v'} \le C T.
\end{equation*}
We have now shown that all the assumptions of 
the compactness criterion of Dreher and J\"ungel \cite{DJ:12} are satisfied, which implies the compactness result of \autoref{prop:comp of Biot & plate vel}.
\end{proof}

\subsection{Compactness for the Biot pore pressure}
In this subsection, we show the compactness result for the pore pressure $p_N$.

\begin{prop}\label{prop:comp for Biot pore pressure}
There is a subsequence $p_N$ so that $p_N \to p$ strongly in $\rL^2(0,T;\rL^2(\Omb))$.
\end{prop}

The proof of this result relies on the compactness criterion of Dreher and J\"ungel~\cite{DJ:12} for piecewise constant functions and follows the same reasoning as the proof of the corresponding 2D result, Theorem~8.2 in \cite{KCM:24}. 

\subsection{Compactness for the fluid velocity}\label{ssec:comp fluid vel}
To prove compactness of the fluid velocity we employ a generalized Aubin--Lions compactness theorem for functions on moving domains introduced in \cite{MC:19}.
To use the generalized Aubin--Lions compactness result from \cite{MC:19} we first transform the fluid problem back onto the physical domain $\Omega_{\rf,N}^n = \{(x,y,z) \in \bR^3 : 0 \le x, y \le L, \, -R \le z \le \omega_N^n(x,y)\}$ with the associated moving boundary
$\Gamma_N^n = \{(x,y,z) \in \bR^3 : 0 \le x, y \le L, \, z = \omega_N^n(x,y)\}.$
We slightly redefine the solution and test spaces for the fluid velocity as follows:
\begin{equation}\label{SpacesCompactVel}
    \begin{aligned}
        V_N^{n+1} 
        &= \bigl\{ \bfu \in \rH^1(\Omega_{\rf,N}^n)^3 : \nabla \cdot \bfu = 0 \ton \Omega_{\rf,N}^n, \, \bfu = 0 \ton \del \Omega_{\rf,N}^n \setminus \Gamma_N^n\bigr\},\\
        Q_N^n 
        &= V_N^{n+1} \cap \rH^3(\Omega_{\rf,N}^n)^3 \cap \rW^{1,4}(\Omega_{\rf,N}^n)^3.
    \end{aligned}
\end{equation}
Note that $\rH^3(\Omega_{\rf,N}^n) \hookrightarrow \rW^{1,4}(\Omega_{\rf,N}^n)$, so the intersection with $\rW^{1,4}(\Omega_{\rf,N}^n)^3$ does not affect the space.
In particular, when equipped with the norm $\| \cdot \|_{Q_N^n} \coloneqq \| \cdot \|_{\rH^3(\Omega_{\rf,N}^n)} + \| \cdot \|_{\rW^{1,4}(\Omega_{\rf,N}^n)}$, we observe that the space $Q_N^n$ is isomorphic to the Hilbert space $\rH^3(\Omega_{\rf,N}^n)^3$.
The reason why we intersect the space~$V_N^{n+1}$ with $\rH^3(\Omega_{\rf,N}^n)^3 \cap \rW^{1,4}(\Omega_{\rf,N}^n)^3$ is to simplify the upcoming estimates.
It is important to note that this more restrictive space is dense in the 
natural test space.

For the semidiscrete formulation of the approximate fluid velocity $\bfu_N^{n+1} \in V_N^{n+1}$, we recall the coupled semidiscrete formulation \eqref{eq:weak form coupled semidiscr probl} and set the test functions $\varphi$, $\bfpsi$ and~$r$ equal to zero.
Moreover, we transform the integrals onto the physical domain, invoke the relation of the $z$-component on the fixed and the moving domain given by $\frac{\hz + R}{R} = \frac{R + z}{R + \omega}$, where we denote by $\hz$ the $z$-component on the fixed domain, and introduce $\Tilde{\bfu}_N^n = \bfu_N^n \circ \bfPhi_\rf^{\omega_N^{n-1}} \circ (\bfPhi_\rf^{\omega_N^n})^{-1}$  for the ALE map $\bfPhi_\rf^{\omega_N^n} \colon \Omf \to \Omega_{\rf,N}^n$ as made precise in \autoref{ssec:maps from ref to phys dom}, to address the mismatch of domains for the plate displacement to get that for all $\bfv \in Q_N^n$
\begin{equation}\label{eq:semidiscr form fluid vel phys dom}
    \begin{aligned}
        &\quad \int_{\Omega_{\rf,N}^n} \frac{\bfu_N^{n+1}-\Tilde{\bfu}_N^n}{\Delta t} \cdot \bfv \srd \bfx + 2 \nu \int_{\Omega_{\rf,N}^n} \bfD(\bfu_N^{n+1}) : \bfD(\bfv) \srd \bfx\\
        &\quad + \frac{1}{2}\int_{\Omega_{\rf,N}^n} \Bigl[\Bigl(\Bigl(\Bigl(\Tilde{\bfu}_N^n - \zeta_N^{n+\frac{1}{2}}\frac{R+z}{R+\omega_N^n}\bfe_3\Bigr) \cdot \nabla\Bigr)\bfu_N^{n+1}\Bigr) \cdot \bfv\\
        &\qquad -\Bigl(\Bigl(\Bigl(\Tilde{\bfu}_N^n - \zeta_N^{n+\frac{1}{2}}\frac{R+z}{R+\omega_N^n}\bfe_3\Bigr) \cdot \nabla\Bigr)\bfv\Bigr) \cdot \bfu_N^{n+1}\Bigr] \srd \bfx\\
        &\quad + \frac{1}{2R} \int_{\Omega_{\rf,N}^n} \frac{R}{R+\omega_N^n} \zeta_N^{n+\frac{1}{2}} \bfu_N^{n+1} \cdot \bfv \srd \bfx + \frac{1}{2} \int_{\Gamma_N^n} (\bfu_N^{n+1} - \dbfeta_N^{n+1}) \cdot \bfn (\Tilde{\bfu}_N^n \cdot \bfv) \srd S\\
        &\quad - \int_{\Gamma_N^n} \Bigl(\frac{1}{2}\bfu_N^{n+1} \cdot \Tilde{\bfu}_N^n - p_N^{n+1}\Bigr) (\bfv \cdot \bfn) \srd S
        - \beta \sum_{i=1}^2 \int_{\Gamma_N^n} (\dbfeta_N^{n+1} - \bfu_N^{n+1}) \cdot \bftau_i (\bfv \cdot \bftau_i) \srd S = 0. 
    \end{aligned}
\end{equation}

To be able to compare the functions on different physical domains, 
we introduce a maximal domain~$\Omf^M$ containing all the physical domains: 
\begin{equation*}
    \Omf^M = \{(x,y,z) \in \bR^3: 0 \le x,y \le L, \, - R \le z \le M(x,y)\},
\end{equation*}
where $M(x,y)$ is specified below in the following lemma.

\begin{lem}\label{lem:bddness of plate displacements}
For $\alpha \in (0,1)$, there exist smooth functions $m(x,y)$ and $M(x,y)$ with $m = M = 0$ on $\del \Gamma$ such that $ m(x,y) \le \omega_N^n(x,y) \le M(x,y)$ for all $x,y \in [0,L]$, $N \in \bN$ and $n=0,1,\dots,N$.
Moreover, for all $N \in \bN$, $n = 0,1,\dots,N-1$ and $l=0,1,\dots,N-n$ as well as $\alpha \in (0,1)$, there are $\alpha$-H\"older-continuous functions $m_N^{n,l}(x,y)$ and $M_N^{n,l}(x,y)$ whose H\"older constants can be bounded above by some constant $L > 0$ that is independent of $n$, $l$ and $N$ with
\begin{enumerate}[(i)]
    \item $m_N^{n,l}(x,y) \le \omega_N^{n+i}(x,y) \le M_N^{n,l}(x,y)$ for all $x$, $y \in [0,L]$ and $i = 0,1,\dots,l$,
    \item $M_N^{n,l}(x,y) - m_N^{n,l}(x,y) \le C (l \Delta t)^{\frac{1}{4}}$, for all $x$, $y \in [0,L]$,
    \item $\| M_N^{n,l}(x,y) - m_N^{n,l}(x,y) \|_{\rL^2(\Gamma)} \le C(l \Delta t)$,
\end{enumerate}
where the constant $C > 0$ is independent of $n$, $l$ and $N$.
\end{lem}

\begin{proof}
The first part of the statement on the existence of suitable functions $m$ and $M$ is a consequence of \autoref{lem:verification of assumptions}.

For the second part of the statement
we fix $N \in \bN$ and consider finitely many functions $\omega_N^{n+i}(x,y)$ for $i=0,\dots,l$ defined on the time interval from $t= n \Delta t$ to~$t + h = n \Delta t + l \Delta t$.
We then define $m_N^{n,l}(x,y)$ and $M_N^{n,l}(x,y)$, $x$, $y \in [0,L]$, to be the functions obtained by considering the minimum and maximum of the finitely functions $\omega_N^{n+i}(x,y)$ for $i=0,\dots,l$.
If necessary, we mollify the functions to get sufficient regularity.
This shows the validity of~(i).

To show (ii), we proceed as in the proof of \autoref{lem:comp for plate displacement}, and  find that the upper bound on the norm $\| \omega_N^{n+l} - \omega_N^n \|_{\rH^{\frac{3}{2}}(\Gamma)}$ only depends on the length of the interval, namely for $C > 0$ independent of $N$, $n$ and~$l$, to get $\| \omega_N^{n+l} - \omega_N^n \|_{\rH^{\frac{3}{2}}(\Gamma)} \le C h^{\frac{1}{4}}$.
Thanks to the embedding $\rH^{\frac{3}{2}}(\Gamma) \hookrightarrow \rC(\Gamma)$, we obtain~(ii).
Finally, (iii) is a consequence of the uniform energy bound for the plate velocity in $\rL^\infty(0,T;\rL^2(\Gamma))$ as revealed in \autoref{lem:unif bddness of approx sols}.
\end{proof}

Once the maximal domain is defined, we extend the fluid velocities 
$\bfu_N^n$ from $\Omega_{\rf,N}^n$ to the common maximal domain $\Omf^M$ by zero, so that we can compare the velocities defined on the maximal domain. 
However, as the plate displacements $\omega_N^n(x,y)$ are merely $\alpha$-H\"older-continuous for $\alpha \in (0,1)$, it is a priori not clear that such an extension by zero is possible.
At this stage, we invoke \cite[Theorem~8.2]{KT:24}, see also \cite[Corollary~B.2]{TW:20} for a similar result:

\begin{lem}\label{lem:ext by zero for Holder}
Suppose that $\omega \colon \Gamma \subset \bR^2 \to \bR$ is $\alpha$-H\"older continuous with $\sup_\Gamma |\omega| \le M$.
Let $\Omega_\omega$ be the subgraph defined by $\Omega_\omega \coloneqq \{(x,y,z) \in \bR^3 : (x,y) \in \Gamma \tand 0 < z < \omega(x,y)\}$.
Moreover, for $\bfu \in \rH^s(\Omega_\omega)$, define the extension by zero by $\Tilde{\bfu} = \bfu \tin \Omega_\omega$ and $\Tilde{\bfu} = 0 \tin \bR^3 \setminus \Omega_\omega$.
Then for all $s < \frac{\alpha}{2}$, we find that
\begin{equation*}
    \| \Tilde{\bfu} \|_{\rH^{s \alpha}(\bR^3)} \le C \cdot \| \bfu \|_{\rH^s(\Omega_\omega)},
\end{equation*}
where the constant $C > 0$ only depends on $s$, $M$ and $\alpha$.
\end{lem}

Thanks to \autoref{lem:ext by zero for Holder}, we may extend the fluid velocities by zero and deduce the uniform boundedness.

\begin{lem}\label{lem:unif bddness of u_N in L2Hs}
The approximate fluid velocities that are obtained by extension by zero, and which we still denote by $\bfu_N$, are uniformly bounded in $\rL^2(0,T;\rH^s(\Omf^M)^3)$ for $s \in (0,\frac{1}{2})$.
\end{lem}

\begin{proof}
Let $s \in (0,\frac{1}{2})$ be arbitrary.
There exist $s' \in (0,\frac{1}{2})$ and $\alpha \in (0,1)$ such that $s' \alpha = s$.
Using Korn-type inequality from \cite[Proposition~2.9]{Len:14},
similarly as in \autoref{lem:unif bddness of approx sols}, we find that~$\bfu_N$ is uniformly bounded in $\rL^2(0,T;\rW^{1,p}(\Omf)^3)$ for all $p \in (1,2)$.
By Sobolev embeddings, there exists $p \in (\frac{6}{5},2)$ such that $\rW^{1,p}(\Omf) \hookrightarrow \rH^{s'}(\Omf)$.
From \autoref{lem:ext by zero for Holder} together with the uniform boundedness of $\omega_N^n$ implied by \autoref{lem:verification of assumptions}, there exists a constant independent of $N$ and $n$, such that the extension by zero $\Tilde{\bfu}_N$ satisfies
\begin{equation*}
    \| \Tilde{\bfu}_N(t) \|_{\rH^s(\Omf^M)} \le C \cdot \| \bfu_N^{n+1} \|_{\rH^{s'}(\Omega_{\rf,N}^n)}
\end{equation*}
for $t \in (n \Delta t,(n+1) \Delta t]$.
By the above uniform boundedness result, we conclude the assertion.
\end{proof}

We are now in the position to discuss the relative compactness of the sequence~$\bfu_N$.

\begin{prop}\label{prop:rel comp fluid vel}
The sequence $\bfu_N$ obtained from an extension by zero to the maximal domain $\Omf^M$ is relatively compact in $\rL^2(0,T;\rL^2(\Omf^M)^3)$.
\end{prop}

The proof is based on the generalized Aubin--Lions compactness lemma \cite{MC:19}.
For completeness, we state the theorem in the variant required for the present setting in \autoref{sec:appendix}.

\begin{proof}[Proof of \autoref{prop:rel comp fluid vel}]
We start by showing that all the assumptions of Theorem 3.1 in \cite{MC:19} as stated in \autoref{lem:gen Aubin--Lions compactness crit} hold. 
\medskip
\noindent
First, we introduce the overarching Hilbert spaces:
\begin{equation*}
    H = \rL^2(\Omf^M)^3 \tand V = \rH^s(\Omf^M)^3 \tfor 0 < s < \frac{1}{2}.
\end{equation*}
From \autoref{lem:bddness of plate displacements}, the maximal domain $\Omf^M$ can  be chosen to be Lipschitz continuous, so by the compactness of Sobolev embeddings, we obtain that $V \Subset H$.
Next, we define the spaces $V_N^n$ and $Q_N^n$ in the statement of Theorem 3.1 in \cite{MC:19} to be the spaces \eqref{SpacesCompactVel} above. 
As required by Theorem 3.1 in \cite{MC:19},  $(V_N^n,Q_N^n) \hookrightarrow V \times V$ where the continuous embedding can be achieved using the extension by zero to the maximal domain, uniformly in $n$ and $N$.
Furthermore, the chain of embeddings~$V_N^n \Subset \overline{Q_N^n}^H \hookrightarrow (Q_N^n)'$ readily follows by construction.
It now remains to check the validity of (A1), (A2), (B), (C1), (C2) and~(C3) from Theorem 3.1 in \cite{MC:19}.

\medskip
\noindent
\textbf{Verification of Property(A):} The uniform boundedness properties in (A) follow from \autoref{lem:unif bddness of u_N in L2Hs}, uniform energy estimates \eqref{eq:unif energy ineq}, and \autoref{lem:unif bddness of approx sols}.

\medskip
\noindent
\textbf{Verification of Property(B):} 
We start by considering the semidiscrete formulation for the fluid velocity on the physical domain from \eqref{eq:semidiscr form fluid vel phys dom}
with $\bfv \in Q_N^n$ such that~$\| \bfv \|_{Q_N^n} \le 1$, and derive that
\begin{equation}\label{eq:cons of semidiscr form moving dom fluid}
    \begin{aligned}
        &\quad \left|\int_{\Omega_{\rf,N}^n} \frac{\bfu_N^{n+1}-\Tilde{\bfu}_N^n}{\Delta t} \cdot \bfv \srd \bfx\right|\\
        &\le 2 \nu \left|\int_{\Omega_{\rf,N}^n} \bfD(\bfu_N^{n+1}) : \bfD(\bfv) \srd \bfx\right| + \frac{1}{2}\left|\int_{\Omega_{\rf,N}^n} \Bigl[\Bigl(\Bigl(\Bigl(\Tilde{\bfu}_N^n - \zeta_N^{n+\frac{1}{2}}\frac{R+z}{R+\omega_N^n}\bfe_3\Bigr) \cdot \nabla\Bigr)\bfu_N^{n+1}\Bigr) \cdot \bfv\right.\\
        &\qquad \left.-\Bigl(\Bigl(\Bigl(\Tilde{\bfu}_N^n - \zeta_N^{n+\frac{1}{2}}\frac{R+z}{R+\omega_N^n}\bfe_3\Bigr) \cdot \nabla\Bigr)\bfv\Bigr) \cdot \bfu_N^{n+1}\Bigr] \srd \bfx\right| + \frac{1}{2R} \left|\int_{\Omega_{\rf,N}^n} \frac{R}{R+\omega_N^n} \zeta_N^{n+\frac{1}{2}} \bfu_N^{n+1} \cdot \bfv \srd \bfx\right|\\
        &\quad + \frac{1}{2} \left|\int_{\Gamma_N^n} (\bfu_N^{n+1} - \dbfeta_N^{n+1}) \cdot \bfn (\Tilde{\bfu}_N^n \cdot \bfv) \srd S\right| + \left|\int_{\Gamma_N^n} \Bigl(\frac{1}{2}\bfu_N^{n+1} \cdot \Tilde{\bfu}_N^n - p_N^{n+1}\Bigr) (\bfv \cdot \bfn) \srd S\right|\\
        &\quad + \beta \sum_{i=1}^2 \left|\int_{\Gamma_N^n} (\dbfeta_N^{n+1} - \bfu_N^{n+1}) \cdot \bftau_i (\bfv \cdot \bftau_i) \srd S\right|.
    \end{aligned}
\end{equation}
Proving Property(B) entails showing that the right-hand side is bounded uniformly in $N$, $n$ and $\| \bfv \|_{Q_N^n}$.
First, it readily follows that
\begin{equation*}
    2 \nu \left|\int_{\Omega_{\rf,N}^n} \bfD(\bfu_N^{n+1}) : \bfD(\bfv) \srd \bfx\right| \le C \cdot \| \bfu_{N}^{n+1} \|_{\rH^1(\Omega_{\rf,N}^n)} \cdot \| \bfv \|_{\rH^1(\Omega_{\rf,N}^n)} \le C \cdot \| \bfu_{N}^{n+1} \|_{V_N^{n+1}}
\end{equation*}
by definition of the space $Q_N^n$ as well as $\| \bfv \|_{Q_N^n} \le 1$.

The estimates of the following terms are more involved, which is due to the lack of Lipschitz continuity of the plate displacement.
For the next term, we first use H\"older's inequality and make use of the boundedness of the factor $\frac{R + z}{R + \omega_N^n}$ to get
\begin{equation*}
    \begin{aligned}
        &\quad \left|\int_{\Omega_{\rf,N}^n} \Bigl(\Bigl(\Bigl(\Tilde{\bfu}_N^n - \zeta_N^{n+\frac{1}{2}}\frac{R+z}{R+\omega_N^n}\bfe_3\Bigr) \cdot \nabla\Bigr)\bfu_N^{n+1}\Bigr) \cdot \bfv \srd \bfx \right|\\
        &\le C \cdot \bigl(\| \Tilde{\bfu}_N^n \|_{\rL^2(\Omega_{\rf,N}^n)} + \| \zeta_N^{n+\frac{1}{2}} \|_{\rL^2(\Gamma)}\bigr) \| \bfu_N^{n+1} \|_{\rH^1(\Omega_{\rf,N}^n)}  \cdot \| \bfv \|_{\rL^\infty(\Omega_{\rf,N}^n)}.
    \end{aligned}
\end{equation*}
Now, we discuss the terms on the right-hand side of the above inequality.
To this end, we recall that the Jacobian of the ALE map $\bfPhi_\rf^{\omega_N^n}$ given by $1 + \frac{\omega_N^n}{R}$ is uniformly bounded thanks to the uniform boundedness of the family of plate displacements.
In conjunction with the uniform energy estimate of the fluid velocity~$\bfu_N^n$ on the fixed domain, this yields that $\| \Tilde{\bfu}_N^n \|_{\rL^2(\Omega_{\rf,N}^n)}$ can be bounded by a constant that is independent of $N$ and $n$.
On the other hand, for $\| \zeta_N^{n+\frac{1}{2}} \|_{\rL^2(\Gamma)}$, we also invoke the uniform energy estimates as presented in the proof of \autoref{lem:unif bddness of approx sols}.
Thus, it remains to estimate $\| \bfv \|_{\rL^\infty(\Omega_{\rf,N}^n)}$.
For this purpose, we make use of the uniform boundedness of the Jacobian of the ALE map, use the Sobolev embedding $\rW^{1,p}(\Omf) \hookrightarrow \rL^\infty(\Omf)$ for $p \in (3,4)$ and employ \cite[Lemma~2.6]{LR:14} to find an estimate of the~$\| \cdot \|_{\rW^{1,p}(\Omf)}$-norm by the $\| \cdot \|_{\rW^{1,4}(\Omega_{\rf,N}^n)}$-norm, where the constant merely depends on the uniform bound of the plate displacements instead of $N$ and $n$.
Note that we will not distinguish variables on moving and fixed domains in terms of notation, but the distinction will become clear from the context.
Thus, still denoting the transformed version of $\bfv$ on $\Omf$ by $\bfv$, we get
\begin{equation*}
    \| \bfv \|_{\rL^\infty(\Omega_{\rf,N}^n)} \le C \cdot \| \bfv \|_{\rL^\infty(\Omf)} \le C \cdot \| \bfv \|_{\rW^{1,p}(\Omf)} \le C \cdot \| \bfv \|_{\rW^{1,4}(\Omega_{\rf,N}^n)} \le C \cdot \| \bfv \|_{Q_N^n} \le C
\end{equation*}
by definition of the space $Q_N^n$.
Putting together these arguments, we find that
\begin{equation*}
    \left|\int_{\Omega_{\rf,N}^n} \Bigl(\Bigl(\Bigl(\Tilde{\bfu}_N^n - \zeta_N^{n+\frac{1}{2}}\frac{R+z}{R+\omega_N^n}\bfe_3\Bigr) \cdot \nabla\Bigr)\bfu_N^{n+1}\Bigr) \cdot \bfv \srd \bfx \right| \le C \cdot \| \bfu_N^{n+1} \|_{V_N^{n+1}},
\end{equation*}
where the constant $C > 0$ is uniform in $N$ and $n$.
For the other term that is associated with the convective term, we use H\"older's inequality, handle the terms related to $\Tilde{\bfu}_N^n$ and $\zeta_N^{n+\frac{1}{2}}$ as above, invoke the definition of the space $Q_N^n$, and employ the uniform boundedness of the Jacobian of the ALE map together with the embedding $\rW^{1,r}(\Omf) \hookrightarrow \rL^4(\Omf)$ for $r \in (\frac{12}{7},2)$ and \cite[Lemma~2.6]{LR:14} for the estimate of $\| \cdot \|_{\rW^{1,r}(\Omf)}$ by $\| \cdot \|_{\rH^1(\Omega_{\rf,N}^n)}$ with a constant that is uniform in $n$ and $N$.
This strategy leads to
\begin{equation*}
    \begin{aligned}
        &\quad\left|\int_{\Omega_{\rf,N}^n} \Bigl(\Bigl(\Bigl(\Tilde{\bfu}_N^n - \zeta_N^{n+\frac{1}{2}}\frac{R+z}{R+\omega_N^n}\bfe_3\Bigr) \cdot \nabla\Bigr)\bfv\Bigr) \cdot \bfu_N^{n+1} \srd \bfx\right|\\
        &\le C \cdot \bigl(\| \Tilde{\bfu}_N^n \|_{\rL^2(\Omega_{\rf,N}^n)} + \| \zeta_N^{n+\frac{1}{2}} \|_{\rL^2(\Gamma)}\bigr) \| \bfu_N^{n+1} \|_{\rL^4(\Omega_{\rf,N}^n)}  \cdot \| \nabla \bfv \|_{\rL^4(\Omega_{\rf,N}^n)}
        \le C \cdot \| \bfv \|_{\rW^{1,4}(\Omega_{\rf,N}^n)} \cdot \| \bfu_N^{n+1} \|_{\rL^4(\Omf)}\\
        &\le C \cdot \| \bfv \|_{Q_N^n} \cdot \| \bfu_N^{n+1} \|_{\rW^{1,r}(\Omf)}
        \le C \cdot \| \bfu_N^{n+1} \|_{\rH^1(\Omega_{\rf,N}^n)}
        \le C \cdot \| \bfu_N^{n+1} \|_{V_N^{n+1}}.
    \end{aligned}
\end{equation*}
For the next term, we may proceed similarly as above:
We estimate $\zeta_N^{n+\frac{1}{2}}$ in $\rL^2(\Gamma)$, estimate~$\| \bfv \|_{\rL^4(\Omega_{\rf,N}^n)}$ by $\| \bfv \|_{Q_N^n}$ and handle $\bfu_N^{n+1}$ in $\rL^4(\Omega_{\rf,N}^n)$ as we have seen in the preceding estimate to get
\begin{equation*}
    \left|\int_{\Omega_{\rf,N}^n} \frac{R}{R+\omega_N^n} \zeta_N^{n+\frac{1}{2}} \bfu_N^{n+1} \cdot \bfv \srd \bfx\right| \le C \cdot \| \zeta_N^{n+\frac{1}{2}} \|_{\rL^2(\Gamma)} \cdot \| \bfu_N^{n+1} \|_{\rL^4(\Omega_{\rf,N}^n)} \cdot \| \bfv \|_{\rL^4(\Omega_{\rf,N}^n)} \le C \cdot \| \bfu_N^{n+1} \|_{V_N^{n+1}}
\end{equation*}
for a constant $C > 0$ that is uniform in $N$ and $n$.

The estimates of the boundary integrals require some further preparation.
In particular, we stress that due to the lack of Lipschitz regularity of the plate displacements, the considerations are significantly different from the 2D case.
First, let us recall the Jacobian of the transformation from the interface $\Gamma_N^n$ to the fixed interface $\Gamma$.
It is given by
\begin{equation*}
    \cJ_\Gamma^{\omega_N^n} = \sqrt{1 + |\del_x \omega_N^n|^2 + |\del_y \omega_N^n|^2}.
\end{equation*}
At this stage, we invoke the uniform boundedness of the family $\omega_N^n$ in $\rH_0^2(\Gamma)$, and as in the proof of \autoref{lem:unif bddness of approx sols}, we note that $\rH^2(\Gamma) \hookrightarrow \rW^{1,s}(\Gamma)$ for all $s < \infty$.
Hence, for every $s < \infty$, we can estimate $\cJ_\Gamma^{\omega_N^n}$ in $\rL^s(\Gamma)$ by a constant that is uniform in $N$ and $n$.
Next, we discuss how to estimate $\| \bfv \|_{\rL^\infty(\Gamma)}$.
In this regard, we observe that $\rW^{\frac{5}{7},\frac{7}{2}}(\Gamma) \hookrightarrow \rL^\infty(\Gamma)$ by Sobolev embeddings.
In addition, using continuity of the trace from $\rW^{1,\frac{7}{2}}(\Omf)$ to $\rW^{\frac{5}{7},\frac{7}{2}}(\Gamma)$, \cite[Lemma~2.6]{LR:14} in order to estimate $\| \cdot \|_{\rW^{1,\frac{7}{2}}(\Omf)}$ by $\| \cdot \|_{\rW^{1,4}(\Omega_{\rf,N}^n)}$, with a constant that is independent of $N$ and $n$, as well as the definition of the space $Q_N^n$, we find that
\begin{equation*}
    \| \bfv \|_{\rL^\infty(\Gamma)} \le C \cdot \| \bfv \|_{\rW^{\frac{5}{7},\frac{7}{2}}(\Gamma)} \le C \cdot \| \bfv \|_{\rW^{1,\frac{7}{2}}(\Omf)} \le C \cdot \| \bfv \|_{\rW^{1,4}(\Omega_{\rf,N}^n)} \le C \cdot \| \bfv \|_{Q_N^n} \le C.
\end{equation*}
Let us emphasize that the (generic) constant $C > 0$ is independent of $N$ and $n$.
By the relations shown in~\eqref{eq:repr of bomega & bbfeta} as well as the uniform energy estimates established in the context of \autoref{lem:unif bddness of approx sols}, we infer that~$\dbfeta_N^{n+1}$ can be uniformly bounded in $\rL^2(\Gamma)$.
It remains to handle the norms involving $\bfu_N^{n+1}$ and $\Tilde{\bfu}_N^n$.
When choosing $s < \infty$ in $\rL^s(\Gamma)$ sufficiently large for the above handling of the Jacobian $\cJ_\Gamma^{\omega_N^n}$, we may consider $\rL^{\frac{5}{2}}(\Gamma)$.
For $\eps > 0$ sufficiently small, it then follows that~$\rW^{\frac{4}{9}-2\eps,\frac{9}{5}}(\Gamma) \hookrightarrow \rL^{\frac{5}{2}}(\Gamma)$.
Additionally making use of classical properties of the trace and the embedding $\rH^{1-\eps,\frac{9}{5}}(\Omf)$, complex interpolation, the uniform boundedness of $\bfu_N^n$ in $\rL^2(\Omf)$ as revealed in \autoref{lem:unif bddness of approx sols}, and \cite[Lemma~2.6]{LR:14} in order to control the transformation back to the moving domain, we get
\begin{equation*}
    \begin{aligned}
        \| \bfu_N^{n+1} \|_{\rL^{\frac{5}{2}}(\Gamma)} 
        &\le C \cdot \| \bfu_N^{n+1} \|_{\rW^{\frac{4}{9}-2\eps,\frac{9}{5}}(\Gamma)} \le C \cdot \| \bfu_N^{n+1} \|_{\rW^{1-2\eps,\frac{9}{5}}(\Omf)} \le C \cdot \| \bfu_N^{n+1} \|_{\rH^{1-\eps,\frac{9}{5}}(\Omf)}\\
        &\le C \cdot \| \bfu_N^{n+1} \|_{\rL^{\frac{9}{5}}(\Omf)}^\eps \cdot \| \bfu_N^{n+1} \|_{\rW^{1,\frac{9}{5}}(\Omf)}^{1-\eps} \le C \cdot \| \bfu_N^{n+1} \|_{\rH^1(\Omega_{\rf,N}^n)}^{1-\eps} \le C \cdot \| \bfu_N^{n+1} \|_{V_N^{n+1}}^{1-\eps}.
    \end{aligned}
\end{equation*}
Upon noting that $\Tilde{\bfu}_N^n$ and $\bfu_N^n$ have the same trace along the interface $\Gamma$, we argue that the term involving~$\Tilde{\bfu}_N^n$ can be handled in a similar way.
Concatenating the previous estimates, and exploiting the elementary estimate $1 + (a b)^\delta  \le (1 + a + b)^{2 \delta}, \tfor a,b \ge 0 \tand \delta \in (\nicefrac{1}{2},1)$, we conclude that
\begin{equation*}
    \begin{aligned}
        \left|\int_{\Gamma_N^n} (\bfu_N^{n+1} - \dbfeta_N^{n+1}) \cdot \bfn (\Tilde{\bfu}_N^n \cdot \bfv) \srd S\right|
        &= \left|\int_{\Gamma} \cJ_\Gamma^{\omega_N^n} (\bfu_N^{n+1} - \dbfeta_N^{n+1}) \cdot \bfn (\bfu_N^n \cdot \bfv) \srd S\right|\\
        &\le C \cdot \bigl(\| \bfu_N^{n+1} \|_{\rL^{\frac{5}{2}}(\Gamma)} + \| \dbfeta_N^{n+1} \|_{\rL^2(\Gamma)}\bigr) \cdot \| \bfu_N^n \|_{\rL^{\frac{5}{2}}(\Gamma)} \cdot \| \bfv \|_{\rL^\infty(\Gamma)}\\
        &\le C \bigl(1 + \| \bfu_N^{n+1} \|_{V_N^{n+1}}^{1-\eps}\bigr) \cdot \| \bfu_N^n \|_{V_N^n}^{1-\eps}\\
        &\le C \cdot \bigl(1+ \| \bfu_N^{n+1} \|_{V_N^{n+1}} + \| \bfu_N^n \|_{V_N^n}\bigr)^{2(1-\eps)}.
    \end{aligned}
\end{equation*}
We stress that the (generic) constant $C > 0$ above is independent of $N$ and $n$, and $2(1-\eps) < 2$ is valid.
Similarly, we find that 
\begin{equation*}
    \begin{aligned}
        \left|\int_{\Gamma_N^n} \Bigl(\frac{1}{2}\bfu_N^{n+1} \cdot \Tilde{\bfu}_N^n - p_N^{n+1}\Bigr) (\bfv \cdot \bfn) \srd S\right|
        &\le C \cdot \bigl(1 + \| p_N^{n+1} \|_{\rH^1(\Omb)} + \| \bfu_N^n \|_{V_N^n} + \| \bfu_N^{n+1} \|_{V_N^{n+1}}\bigr)^{2(1-\eps)}, \tand\\
        \beta \sum_{i=1}^2 \left|\int_{\Gamma_N^n} (\dbfeta_N^{n+1} - \bfu_N^{n+1}) \cdot \bftau_i (\bfv \cdot \bftau_i) \srd S\right|
        &\le C \cdot \bigl(1 + \| \bfu_N^{n+1} \|_{V_N^{n+1}}\bigr),
    \end{aligned}
\end{equation*}
where $C > 0$ is again uniform in $N$ and $n$.
Combining these estimates, we find that
\begin{equation*}
    \sup_{\| \bfv \|_{Q_N^n} \le 1} \left|\int_{\Omega_{\rf,N}^n} \frac{\bfu_N^{n+1}-\Tilde{\bfu}_N^n}{\Delta t} \cdot \bfv \srd \bfx\right| \le C \cdot \bigl(a_N^n + \| \bfu_N^n \|_{V_N^n} + \| \bfu_N^{n+1} \|_{V_N^{n+1}}\bigr)^{2(1-\eps)},
\end{equation*}
where $C > 0$ is uniform in $N$ and $n$, and $a_N^n \coloneqq 1 + \| p_N^{n+1} \|_{\rH^1(\Omb)}$ is such that
\begin{equation*}
    (\Delta t) \sum_{n=0}^{N-1} |a_N^n|^2 \le 2 \bigl((\Delta t) N + \| p_N \|_{\rL^2(0,T;\rH^1(\Omb))}^2\bigr) \le C.
\end{equation*}
The boundedness of the pore pressure follows from \autoref{lem:unif bddness of approx sols}.
By \eqref{eq:cons of semidiscr form moving dom fluid}, the above estimates lead to 
\begin{equation}\label{eq:est first part prop B}
    \left|\int_{\Omega_{\rf,N}^n} \frac{\bfu_N^{n+1}-\Tilde{\bfu}_N^n}{\Delta t} \cdot \bfv \srd \bfx\right| \le C \cdot \bigl(a_N^n + \| \bfu_N^n \|_{V_N^n} + \| \bfu_N^{n+1} \|_{V_N^{n+1}}\bigr)^{2(1-\eps)}
\end{equation}
for $C > 0$ uniform in $N$ and $n$ and $\eps > 0$ small so that $2(1-\eps) \in (1,2)$.

To finish proving that Property(B) from Theorem 3.1 in \cite{MC:19} holds, it remains to estimate the term $\bigl|\int_{\Omega_{\rf,N}^n} \frac{\Tilde{\bfu}_N^n - \bfu_N^n}{\Delta t} \cdot \bfv \srd \bfx\bigr|$.
Indeed, following the approach in \cite[Theorem~4.2]{MC:19}, see also \cite[Theorem~10.6]{CGM:21} for the 3D case, we get the existence of a constant $C > 0$ that is uniform in $N$ and $n$ such that
\begin{equation}\label{eq:est second part prop B}
    \left|\int_{\Omega_{\rf,N}^n} \frac{\Tilde{\bfu}_N^n - \bfu_N^n}{\Delta t} \cdot \bfv \srd \bfx\right| \le C \cdot \| \bfu_N^n \|_{V_N^n}.
\end{equation}

Combining \eqref{eq:est first part prop B} and \eqref{eq:est second part prop B}, we find that there exists a constant $C > 0$ that is independent of $N$ and~$n$, $a_N^n \coloneqq 1 + \| p_N^{n+1} \|_{\rH^1(\Omb)}$ with $(\Delta t) \sum_{n=0}^{N-1} |a_N^n|^2 \le C$ and $p \in [1,2)$ such that
\begin{equation*}
    \left\| P_N^n \frac{\bfu_N^{n+1}-\bfu_N^n}{\Delta t} \right\|_{(Q_N^n)'} \le C\bigl(a_N^n + \| \bfu_N^n \|_{V_N^n} + \| \bfu_N^{n+1} \|_{V_N^{n+1}}\bigr)^p
\end{equation*}
for all $n = 0,1,\dots,N-1$.
This completes the verification of Property(B). 

\medskip
\noindent
\textbf{Verification of Property(C):} This property is crucial for compactness on moving domains as it specifies the conditions under which one can take time derivatives of functions defined on different domains. 
Conditions~(C1) and (C2) introduce ``common'' test and solutions spaces that depend smoothly on the time shifts, locally in time. 
Condition (C3) is a ``global'' estimate in which one has to show that Conditions~(C1) and (C2) hold uniformly for all $n$, $N$, and time-shifts by $l\Delta t$. 

The proof of {\bf{Property(C1)}} can be obtained analogously to the respective proof in \cite[Theorem~4.2]{MC:19} upon defining the common test space $Q_N^{n,l}$ similarly as $Q_N^n$ in \eqref{SpacesCompactVel} and estimating the test function $\bfv$ carefully to guarantee uniformity of the constants in $N$, $l$ and $n$.

{\bf{Property(C2)}} follows from a {\emph{squeezing procedure}} in the wake of \cite[Definition~4.2]{MC:19}.
Concerning the last estimate required for (C2), we get
\begin{equation*}
    \| I_{N,l,n}^i\bfv - \bfv \|_{\rL^2\bigl(\Omf^{M_N^{n,l}}\bigr)} \le C \bigl((l \Delta t)^{\frac{1}{4}} + (l \Delta t)^{\frac{1}{8}}\bigr) \cdot \| \bfv \|_{V_N^{n+i}},
\end{equation*}
completing the proof for the function $g(h) \coloneqq C \bigl(h^{\frac{1}{4}} + h^{\frac{1}{8}}\bigr)$.

The proof of the uniform Ehrling property, \textbf{Property(C3)}, is similar to the one in \cite[Section~4.2]{MC:19}, so we omit the details here.

Now that all the assumptions of Theorem 3.1 in \cite{MC:19} as formulated in \autoref{lem:gen Aubin--Lions compactness crit} have been verified, the compactness results for fluid velocity stated in \autoref{prop:rel comp fluid vel} follows by applying this theorem.
\end{proof}

\section{Passing to the limit in the regularized weak formulation}\label{sec:limit passage}

In this section, we pass to the limit in the regularized weak formulation in order to show existence of a finite-energy weak solution as stated in \autoref{thm:ex of a weak sol}.
For this purpose, we summarize the strong convergences that we have established so far.
\\
{\bf{A summary of strong convergence results:}}
\begin{enumerate}[(i)]
    \item $\overline{\bfeta}_N \to \bfeta$  in~$\rC(0,T;\rL^2(\Omb)^3)$, see \autoref{lem:strong conv Biot displacement}.
    \item $\overline{\omega}_N \to \omega$ in~$\rC(0,T;\rH^s(\Gamma))$ and $\omega_N \to \omega$ in~$\rL^\infty(0,T;\rH^s(\Gamma))$ for $s \in (0,2)$, see \autoref{lem:comp for plate displacement}.
    \item $\bfxi_N \to \bfxi$ in $\rL^2(0,T;\rH^{s}(\Omb)^3)$ and $\zeta_N \to \zeta$  in $\rL^2(0,T;\rH^{s}(\Gamma))$ for $s \in (-\frac{1}{2},0)$, see  \autoref{prop:comp of Biot & plate vel}.
    \item $p_N \to p$ in $\rL^2(0,T;\rL^2(\Omb))$, see \autoref{prop:comp for Biot pore pressure}.
    \item  $\bfu_N \to \bfu$ in $\rL^2(0,T;\rL^2(\Omf^M)^3)$, see \autoref{prop:rel comp fluid vel}.
    \item $\zeta_N^* \to \zeta$ in $\rL^2(0,T;\rH^{s}(\Omb)^3)$ for $s \in (-\frac{1}{2},0)$.
\end{enumerate}
Concerning the last statement, it might not be clear that $(\zeta_N)$ and $(\zeta_N^*)$ converge to the same limit in~$\rL^2(0,T;\rH^{s}(\Gamma))$,
however, \eqref{eq:num dissip ests} implies
\begin{equation*}
    \sum_{n=0}^{N-1} \| \zeta_N^{n+\frac{1}{2}} - \zeta_N^n \|_{\rL^2(\Gamma)}^2 \le C, \tso \| \zeta_N - \zeta_N^* \|_{\rL^2(0,T;\rL^2(\Gamma))} \to 0.
\end{equation*}

We will use these convergence results to pass to the limit in the semidiscrete coupled formulation~\eqref{eq:weak form coupled semidiscr probl}.
Before passing to the limit, we need some further preparation.
First, we identify the weak limit of the family of the gradients of the fluid velocity.
Second, we establish the strong convergence for the traces of the fluid velocity on the boundary of the fluid domain.
Third, we show a convergence result for the test functions defined on approximate domains.
These steps will be carried out in the following subsections.

\subsection{Convergence of the gradient of the fluid velocity}\label{ssec:conv of grad of fluid vel}
\ 

From \autoref{cor:weak & weak* conv}, we recall that for all $p \in (1,2)$, the sequence $(\nabla_\rf^{\tau_{\Delta t} \omega_N} \bfu_N)$ converges weakly in~$\rL^2(0,T;\rL^p(\Omf)^3)$ to some $\bfG$.
Here, we identify this weak limit as the gradient of the limit of the sequence of fluid velocities $\bfu_N$.
More precisely, we have the following.

\begin{lem}\label{lem:lim of the gradient of the pressure}
Let $\bfu$, $\bfG$ and $\omega$ be the weak* limit of the sequence $\bfu_N$ in $\rL^\infty(0,T;\rL^2(\Omf)^3)$, the weak limit of the sequence $(\nabla_\rf^{\tau_{\Delta t}\omega_N} \bfu_N)$ in $\rL^2(0,T;\rL^p(\Omf)^3)$ and the weak* limit of the sequence $\omega_N$ in $\rL^\infty(0,T;\rH_0^2(\Gamma))$, respectively.
Then it holds that $\bfG = \nabla_\rf^\omega \bfu$.
\end{lem}

The proof of this result is conceptually similar to the proof of Proposition~7.6 in \cite{MC:14}, so we omit the details here for brevity.

\subsection{Strong convergence of the fluid velocity traces at the interface $\Gamma$}\label{ssec:strong conv of fluid vel traces}
\ 

Recall the definition of the time shift $\tau_{\Delta t}$  introduced in \autoref{cor:weak & weak* conv}. The main result in this subsection is the following.

\begin{prop}\label{prop:conv of fluid vel traces}
Let $\hbfu_N = \bfu_N \circ \bfPhi_\rf^{\tau_{\Delta t} \omega_N}$ and $\hbfu = \bfu \circ \bfPhi_\rf^\omega$.
Then the traces $\left.\hbfu_N\right|_{\Gamma}$ of the approximate fluid velocities on $\Gamma$ converge to the trace of the limiting fluid velocity on $\Gamma$ as $N \to \infty$, i.e., for $s \in (\frac{1}{2},1)$, we have $\left.\hbfu_N\right|_{\Gamma} \to \left.\hbfu\right|_{\Gamma}$ in $\rL^2(0,T;\rH^{s-\frac{1}{2}}(\Gamma)^3)$.
\end{prop}

The proof of \autoref{prop:conv of fluid vel traces} will be carried out in several steps.
One important ingredient is the lemma below on strong convergences.
Let us emphasize that even though the {\emph{assumptions are weaker}} than in the 2D case, see \cite[Lemma~9.1]{KCM:24}, we do get convergence in the same spaces as in 2D.

\begin{lem}\label{lem:conv in L2Hs}
Assume that $(\bff_n)$ and $\bff$ are uniformly bounded in $\rL^2(0,T;\rW^{1,p}(\Omf)^3)$ for all $p \in (1,2)$ as well as~$\bff_n \to \bff$ in $\rL^2(0,T;\rL^2(\Omf)^3)$.
Then for $s \in (0,1)$, it follows that $\bff_n \to \bff$ in $\rL^2(0,T;\rH^s(\Omf)^3)$.
Moreover, we have $\left. \bff_n \right|_\Gamma \to \left. \bff \right|_\Gamma$  in $\rL^2(0,T;\rH^{s-\frac{1}{2}}(\Gamma)^3)$ for $s \in (\frac{1}{2},1)$.
\end{lem}

\begin{proof}
Let $s \in (0,1)$.
For some $p \in (1,2)$, we find $s' \in (s,1)$ with $s' - \frac{3}{p} > s - \frac{3}{2}$, so $\rH^{s',p}(\Omf) \hookrightarrow \rH^s(\Omf)$ by Sobolev embeddings.
Furthermore, by complex interpolation, we have $\rH^{s',p}(\Omf) = [\rL^p(\Omf),\rW^{1,p}(\Omf)]_{s'}$.
Thus, additionally using that $\rL^2(\Omf) \hookrightarrow \rL^p(\Omf)$ for $p \in (1,2)$, we find that
\begin{equation*}
    \begin{aligned}
        \| \bff_n - \bff \|_{\rL^2(0,T;\rH^s(\Omf))}^2
        &= \int_0^T \| (\bff_n - \bff)(t) \|_{\rH^s(\Omf)}^2 \srd t
        \le C \cdot \int_0^T \| (\bff_n - \bff)(t) \|_{\rH^{s',p}(\Omf)}^2 \srd t\\
        &\le C \cdot \int_0^T \| (\bff_n - \bff)(t) \|_{\rL^p(\Omf)}^{2(1-s')} \cdot \| (\bff_n - \bff)(t) \|_{\rW^{1,p}(\Omf)}^{2 s'} \srd t\\
        &\le C \cdot \| \bff_n - \bff \|_{\rL^2(0,T;\rL^2(\Omf))}^{2(1-s')} \cdot \| \bff_n - \bff \|_{\rL^2(0,T;\rW^{1,p}(\Omf))}^{2s'}.
    \end{aligned}
\end{equation*}
Now, the factor $\| \bff_n - \bff \|_{\rL^2(0,T;\rW^{1,p}(\Omf))}$ is uniformly bounded by assumption, while for the first factor, we observe that it tends to zero as $N \to \infty$ thanks to the assumption on the strong convergence of $\bff_n \to \bff$ in $\rL^2(0,T;\rL^2(\Omf)^3)$.
On the other hand, by classical properties of the trace, for $s \in (\frac{1}{2},1)$, we get
\begin{equation*}
    \| \left.\bff_n\right|_\Gamma - \left.\bff\right|_\Gamma\|_{\rL^2(0,T;\rH^{s-\frac{1}{2}}(\Gamma))}^2 \le C \cdot \| \bff_n - \bff \|_{\rL^2(0,T;\rH^s(\Omf))}^2. \qedhere
\end{equation*}
\end{proof}

We are now in the position to prove \autoref{prop:conv of fluid vel traces}.

\begin{proof}[Proof of~\autoref{prop:conv of fluid vel traces}]
The main idea is to link the strong convergence $\bfu_N \to \bfu$ in~$\rL^2(0,T;\rL^2(\Omf^M)^3)$ from \autoref{prop:rel comp fluid vel} with the uniform boundedness of $\bfu_N$ and its limit $\bfu$ in~$\rL^2(0,T;\rW^{1,p}(\Omf)^3)$ as discussed in \autoref{lem:unif bddness of approx sols}, and to employ \autoref{lem:conv in L2Hs}.

In the first part of the proof, we shall verify that $\hbfu_N \to \hbfu$ in $\rL^2(0,T;\rL^p(\Omf)^3)$ for all $p \in (1,2)$.
To this end, we express the relevant difference $\| \hbfu_N - \hbfu \|_{\rL^2(0,T;\rL^p(\Omf))}^2$ in terms of $\bfu_N$ and $\bfu$ defined on the maximal fluid domain $\Omf^M$.
In fact, we get
\begin{equation*}
    \begin{aligned}
        &\quad\| \hbfu_N - \hbfu \|_{\rL^2(0,T;\rL^p(\Omf))}^2\\
        &= \int_0^T \int_{\Omf} \left|\bfu_N\Bigl(t,x,y,z+\Bigl(1+\frac{z}{R}\Bigr) \tau_{\Delta t} \omega_N\Bigr) - \bfu\Bigl(t,x,y,z+\Bigl(1+\frac{z}{R}\Bigr) \omega\Bigr)\right|^p \srd \bfx \srd t\\
        &\le \int_0^T \int_{\Omf} \left|\bfu_N\Bigl(t,x,y,z+\Bigl(1+\frac{z}{R}\Bigr) \tau_{\Delta t} \omega_N\Bigr) - \bfu\Bigl(t,x,y,z+\Bigl(1+\frac{z}{R}\Bigr) \tau_{\Delta t} \omega_N\Bigr)\right|^p \srd \bfx \srd t\\
        &\quad + \int_0^T \int_{\Omf} \left|\bfu\Bigl(t,x,y,z+\Bigl(1+\frac{z}{R}\Bigr) \tau_{\Delta t} \omega_N\Bigr) - \bfu\Bigl(t,x,y,z+\Bigl(1+\frac{z}{R}\Bigr) \omega\Bigr)\right|^p \srd \bfx \srd t
        \eqqcolon I_1 + I_2.
    \end{aligned}
\end{equation*}
In order to handle the integral $I_1$, we invoke the uniform boundedness of the term $1+\frac{\omega_N^n}{R}$ to find that
\begin{equation*}
    \begin{aligned}
        I_1 
        &= \sum_{n=0}^{N-1} \int_{n \Delta t}^{(n+1)\Delta t} \int_{\Omega_{\rf,N}^n} \Bigl(1+\frac{\omega_N^n}{R}\Bigr) |\bfu_N^{n+1} - \bfu|^p \srd \bfx\\
        &\le C \sum_{n=0}^{N-1} \int_{n \Delta t}^{(n+1) \Delta t} \int_{\Omega_{\rf,N}^n} |\bfu_N^{n+1} - \bfu|^p \srd \bfx
        \le C \cdot \| \bfu_N - \bfu \|_{\rL^2(0,T;\rL^2(\Omf^M))}^2.
    \end{aligned}
\end{equation*}
Thus, with regard to the strong convergence $\bfu_N \to \bfu$ in $\rL^2(0,T;\rL^2(\Omf^M)^3)$ by \autoref{prop:rel comp fluid vel}, the integral $I_1$ indeed tends to zero as $N \to \infty$.

Concerning the second integral, $I_2$, we introduce 
\begin{equation*}
    z^*(t,x,y) = \frac{R(\omega(t,x,y) - \tau_{\Delta t}\omega_N(t,x,y))}{R + \tau_{\Delta t}\omega_N(t,x,y)}.
\end{equation*}
Thus, for $z^*$, it holds that $z^* + (1+\frac{z^*}{R}) \tau_{\Delta t} \omega_N = \omega$.
With the above $z^*$, we further split $I_2$, namely, 
\begin{equation*}
    \begin{aligned}
        I_2 
        &= I_{2,1} + I_{2,2}, \twhere\\
        I_{2,1}
        &\coloneqq \int_0^T \int_\Gamma \int_{-R}^{\min(0,z^*(t,x,y))} \left|\bfu\Bigl(t,x,y,z+\Bigl(1+\frac{z}{R}\Bigr) \tau_{\Delta t} \omega_N\Bigr) - \bfu\Bigl(t,x,y,z+\Bigl(1+\frac{z}{R}\Bigr) \omega\Bigr)\right|^p \srd z \srd(x,y) \srd t,\\
        I_{2,2}
        &\coloneqq \int_0^T \int_\Gamma \int_{\min(0,z^*(t,x,y))}^{0} \left|\bfu\Bigl(t,x,y,z+\Bigl(1+\frac{z}{R}\Bigr) \tau_{\Delta t} \omega_N\Bigr)\right|^p \srd z \srd(x,y) \srd t,
    \end{aligned}
\end{equation*}
where for $I_{2,2}$, we used that $\bfu$ is extended by zero to the maximal fluid domain $\Omf^M$, so it vanishes for values of $z$ that exceed $\omega$.
For the estimate of $I_{2,1}$, the fundamental theorem of calculus implies that
\begin{equation*}
    \int_{z + (1+\frac{z}{R})\tau_{\Delta t}\omega_N}^{z+(1+\frac{z}{R})\omega} \del_{z'} \bfu(t,x,y,z') \srd z' = \bfu\Bigl(t,x,y,z+\Bigl(1+\frac{z}{R}\Bigr)\omega\Bigr) - \bfu\Bigl(t,x,y,z+\Bigl(1+\frac{z}{R}\Bigr)\tau_{\Delta t}\omega_N\Bigr).
\end{equation*}
Consequently, using H\"older's inequality, we obtain
\begin{equation*}
    \begin{aligned}
        I_{2,1}
        &\le \int_0^T \int_\Gamma \int_{-R}^{\min(0,z^*(t,x,y))} \left(\int_{z + (1+\frac{z}{R})\tau_{\Delta t}\omega_N(t,x,y)}^{z+(1+\frac{z}{R})\omega(t,x,y)} |\del_{z'} \bfu(t,x,y,z')| \srd z'\right)^p \srd z \srd(x,y) \srd t\\
        &\le \int_0^T \int_\Gamma \int_{-R}^{\min(0,z^*(t,x,y))} \left|\Bigl(1+\frac{z}{R}\Bigr) \cdot \bigl(\omega(t,x,y) - \tau_{\Delta t} \omega_N(t,x,y)\bigr)\right|^{\frac{p}{p'}}\\
        &\quad \cdot\int_{z + (1+\frac{z}{R})\tau_{\Delta t}\omega_N(t,x,y)}^{z+(1+\frac{z}{R})\omega(t,x,y)} |\del_{z'} \bfu(t,x,y,z')|^p \srd z' \srd z \srd(x,y) \srd t.
    \end{aligned}
\end{equation*}
In a similar way as in the proof of \autoref{lem:unif bddness of approx sols}, one can show that the term $\| \nabla \bfu \|_{\rL^2(0,T;\rL^p(\Omf^\omega(t)))}$ is uniformly bounded.
On the other hand, as shown in \autoref{lem:comp for plate displacement} and also recalled at the beginning of this section, for $s \in (0,2)$, we have $\overline{\omega}_N \to \omega$ strongly in $\rC(0,T;\rH^s(\Gamma))$.
By virtue of \eqref{eq:Holder est omega_N}, this yields the pointwise uniform convergence of $\tau_{\Delta t} \omega_N \to \omega$ in $[0,T] \times \Gamma$ as $N \to \infty$.
Putting together these observations, we conclude that $I_{2,1} \to 0$ as $N \to \infty$.

Next, we tackle $I_{2,2}$.
For this, we use the embedding $\rW^{1,p}(-R,\omega) \hookrightarrow \rL^\infty(-R,\omega)$ for all $p > 1$ to obtain
\begin{equation*}
    \begin{aligned}
        I_{2,2}
        &\le C \cdot \int_0^T \int_\Gamma |\min(0,z^*(t,x,y))| \cdot \max_{z \in [-R,\omega(t,x,y)]} |\bfu(t,x,y,z)|^p \srd(x,y) \srd z \srd t\\
        &\le C \cdot \int_0^T \int_\Gamma |\min(0,z^*(t,x,y))| \cdot \int_{-R}^{\omega(t,x,y)} |\dz \bfu(t,x,y,z)|^p + |\bfu(t,x,y,z)|^p \srd z \srd(x,y) \srd t.
    \end{aligned}
\end{equation*}
The convergence $I_{2,2} \to 0$ as $N \to \infty$ then follows from $|\min(0,z^*(t,x,y))| \to 0$ uniformly on $[0,T] \times \Gamma$, which is based on the aforementioned convergence of $\tau_{\Delta t} \omega_N$ to $\omega$, together with the uniform boundedness of $\bfu$ in $\rL^2(0,T;\rL^2(\Omf^\omega)^3)$ as well as $\nabla \bfu$ in $\rL^2(0,T;\rL^p(\Omf^\omega)^{3 \times 3})$ for all $p \in (1,2)$.
Therefore, we have shown that $\| \hbfu_N - \hbfu \|_{\rL^2(0,T;\rL^p(\Omf)^3)} \to 0$.

Next, we show that $\hbfu_N$ and $\hbfu$ are uniformly bounded in $\rL^2(0,T;\rW^{1,p}(\Omf)^3)$ so that we can use \autoref{lem:conv in L2Hs} to obtain the final convergence result.
To this end, we recall from \autoref{lem:unif bddness of approx sols} that $\hbfu_N$ is uniformly bounded in $\rL^2(0,T;\rW^{1,p}(\Omf)^3)$ for all $p \in (1,2)$.
At the same time, $\hbfu$ is the strong limit of $\hbfu_N$ in $\rL^2(0,T;\rL^2(\Omf)^3)$ and thus also in $\rL^2(0,T;\rL^p(\Omf)^3)$ for $p \in (1,2)$.
Therefore, as $\hbfu_N$ converges weakly in $\rL^2(0,T;\rW^{1,p}(\Omf)^3)$ along a subsequence to a weak limit that has to coincide with $\hbfu$, we argue that $\hbfu$ is also in $\rL^2(0,T;\rW^{1,p}(\Omf)^3)$.

We can now combine the previous steps and use \autoref{lem:conv in L2Hs} to obtain the strong convergence result of fluid velocity traces as stated in \autoref{prop:conv of fluid vel traces}.
\end{proof}

\subsection{Convergence of the test functions on approximate fluid domains}\label{ssec:conv of test fcts on approx fluid doms}
One obstacle in passing to the limit is that the test functions for the fluid velocity in $\cVtestom$, see \autoref{def:sol & test space}, are defined on the fixed reference domain $\Omf$ and satisfy $\nabla_\rf^\omega \cdot \bfv = 0$ on $\Omf$, while the test functions in the semidiscrete formulation lie in the test space $Q_N^n$ and satisfy $\nabla_\rf^{\omega_N^n} \cdot \bfv = 0$ in $\Omf$.
Thus, we need to find a way to compare the test functions in $Q_N^n$ to those in the test space $\cVtestom$.
The idea is to use test functions that are defined on the maximal domain $\Omf^M$, and which are constructed such that the restrictions of these test functions to the domains associated with $\omega$ and composed with the ALE mapping $\bfPhi_\rf^\omega$ lead to a space of test functions $\cX_\rf^\omega$ that is dense in the fluid velocity test space $\cV_\rf^\omega$.
We then denote the space of all such test functions defined on $\Omf^M$ by $\cX$.
More precisely, due to the reduced regularity of the ALE mapping, we need to define the test space in a slightly more careful way.
First, we introduce 
\begin{equation*}
    \prescript{}{0}{V_\rf^M} \coloneqq \bigl\{\bfv \in \rC^1(\overline{\Omf^M})^3: \nabla \cdot \bfv = 0 \tand \bfv = 0 \tfor x,y \in \{0,L\} \tand z = -R\bigr\}.
\end{equation*}
We then define $V_\rf^M \coloneqq \overline{\prescript{}{0}{V_\rf^M}}^{\rH^1(\Omf^M)^3}$.
Similarly as mentioned in the context of \autoref{def:sol & test space}, one can verify that $V_\rf^M$ admits the characterization
\begin{equation*}
    V_\rf^M = \bigl\{\bfv \in \rH^1(\Omf^M)^3 : \nabla \cdot \bfv = 0 \tand \bfv = 0 \tfor x,y \in \{0,L\} \tand z = -R\bigr\}.
\end{equation*}
The space $\cX$ consists of functions $\bfv \in \rC_\rc^1([0,T);\rH^1(\Omf^M)^3)$ such that for all $t \in [0,T)$, the following hold:
\begin{enumerate}[(a)]
    \item $\bfv(t)$ is a smooth vector-valued function in $\Omf^M$,
    \item $\nabla \cdot \bfv = 0$ in $\Omf^M$,
    \item $\bfv(t) = 0$ on $\del \Omf^M \setminus \Gamma_M$, where $\Gamma_M = \{(x,y,M(x,y)) : 0 \le x,y \le L\}$ is the top boundary of the maximal fluid domain $\Omf^M$.
\end{enumerate}
For $\bfv \in \cX$, we set
\begin{equation}\label{eq:mod test fcts}
    \tbfv \coloneqq \left. \bfv \right|_{\Omf^\omega} \circ \bfPhi_\rf^\omega \tand \tbfv_N \coloneqq \left. \bfv \right|_{\Omf^{\omega_N}} \circ \bfPhi_\rf^{\omega_N}.
\end{equation}

Let us observe that the test functions $\tbfv$ are dense in the fluid velocity test space $\cVfom$ related to the fixed domain formulation, while the test functions $\tbfv_N$ restricted to the time intervals $[n \Delta t,(n+1) \Delta t)$ are dense in $V_\rf^{\omega_N^n}$, with $V_\rf^{\omega_N^n}$ as made precise in \eqref{eq:spat fluid sol space}.
Thus, for every fixed $N$, we can take into account the semidiscrete formulation with the test function $\tbfv_N$.
Note that the $\tbfv_N$ are discontinuous in time due to the jumps in $\omega_N$ at each time step $n \Delta t$.

In the lemma below, we state the convergence result for the test functions $\tbfv_N$ and their gradients.
We emphasize that the lack of Lipschitz continuity implies that we cannot obtain uniform convergence of the gradients as in the 2D case. Instead, the gradients converge only in~$\rL^\infty(0,T;\rL^p(\Omf^M)^{3 \times 3}))$ for all $p < \infty$.

\begin{lem}\label{lem:conv of test fcts}
Let $\bfv \in \cX$ and $\tbfv$ as well as $\tbfv_N$ be as defined in \eqref{eq:mod test fcts}.
Then $\tbfv_N \to \tbfv$ pointwise uniformly in $[0,T] \times \Omf$ and $\nabla \tbfv_N \to \nabla \tbfv$ in $\rL^\infty(0,T;\rL^p(\Omf)^{3 \times 3}))$ for all $p < \infty$.
\end{lem}

\begin{proof}
The proof is inspired by the one of \cite[Lemma~4]{MC:13a}.
In fact, let us first recall that
\begin{equation*}
    \tbfv(t,x,y,z) = \bfv\Bigl(t,x,y,z+\Bigl(1+\frac{z}{R}\Bigr)\omega(t,x,y)\Bigr) \tand \tbfv_N(t,x,y,z) = \bfv\Bigl(t,x,y,z+\Bigl(1+\frac{z}{R}\Bigr)\omega_N(t,x,y)\Bigr). 
\end{equation*}
Thus, making use of the mean value theorem, we find that 
\begin{equation*}
    \begin{aligned}
        |\tbfv_N(t,x,y,z) - \tbfv(t,x,y,z)| 
        &= \left|\bfv\Bigl(t,x,y,z+\Bigl(1+\frac{z}{R}\Bigr)\omega_N(t,x,y)\Bigr) - \bfv\Bigl(t,x,y,z+\Bigl(1+\frac{z}{R}\Bigr)\omega(t,x,y)\Bigr)\right|\\
        &= |\dz \bfv(t,x,y,\zeta)| \cdot \left|\Bigl(1 + \frac{z}{R}\Bigr)(\omega_N(t,x,y) - \omega(t,x,y))\right|.
    \end{aligned}
\end{equation*}
From here, as $(1 + \frac{z}{R})$ is clearly bounded, and invoking the boundedness of $|\dz \bfv(t,x,y,\zeta)|$ thanks to the smoothness assumption on $\bfv$ coming from the test space as well as the uniform convergence of $\omega_N$ to $\omega$ by \autoref{lem:comp for plate displacement} and the Sobolev embedding $\rH^s(\Gamma) \hookrightarrow \rL^\infty(\Gamma)$ for $s > 1$, we deduce the first convergence statement of the lemma.

For the second part, for $i \in \{x,y\}$, we compute
by the chain rule that
\begin{equation*}
    \begin{aligned}
        \del_i \tbfv_N(t,x,y,z) 
        &= \del_i \Bigl[\bfv\Bigl(t,x,y,z+\Bigl(1+\frac{z}{R}\Bigr) \omega_N(t,x,y)\Bigr)\Bigr]
        = \del_i \bfv\Bigl(t,x,y,z+\Bigl(1+\frac{z}{R}\Bigr)\omega_N(t,x,y)\Bigr)\\
        &\quad + \Bigl(1+\frac{z}{R}\Bigr) \del_i \omega_N(t,x,y) \cdot \del_z \bfv\Bigl(t,x,y,z+\Bigl(1+\frac{z}{R}\Bigr)\omega_N(t,x,y)\Bigr)\\
        \del_i \tbfv(t,x,y,z) 
        &= \del_i \bfv\Bigl(t,x,y,z+\Bigl(1+\frac{z}{R}\Bigr)\omega(t,x,y)\Bigr)\\
       &\quad+ \Bigl(1+\frac{z}{R}\Bigr) \del_i \omega(t,x,y) \cdot \del_z \bfv\Bigl(t,x,y,z+\Bigl(1+\frac{z}{R}\Bigr)\omega(t,x,y)\Bigr).
    \end{aligned}
\end{equation*}
The desired convergence of $\del_i \tbfv_N(t,x,y,z)$ to $\del_i \bfv$ then follows in a similar way as the uniform convergence in the first part of the proof upon additionally invoking the uniform boundedness of $\del_i \omega_N$ in $\rL^p$ for all~$p < \infty$ and the convergence of $\del_i \omega_N$ to $\del_i \omega$ in $\rL^p$ for all $p < \infty$ thanks to \autoref{lem:comp for plate displacement}.
The procedure for controlling the difference $\dz \tbfv_N - \dz \tbfv$ is similar, finishing the proof of this lemma.
\end{proof}

The construction of the above test space $\cX$ is considerably simpler than the corresponding construction in the situation of FSI problems with purely elastic structures, where the fluid and and structure test spaces are related via a kinematic coupling condition, included in the velocity test space. 
This is not the case for fluid-poroelastic structure interaction problems and the fluid test space in our case can therefore be decoupled from the rest.

\subsection{Passage to the limit in the regularized weak formulation}\label{ssec:limit passage}
Here, we pass to the limit in  the semidiscretized formulation of the regularized problem in order to prove the main result, \autoref{thm:ex of a weak sol}. 
To this end, we take a test function $(\tbfv_N,\varphi,\bfpsi,r)$ for given $\bfv \in \cX$,  where $\tbfv_N$ was introduced in \eqref{eq:mod test fcts}.
For every $n = 0,1,\dots,N-1$, we then test the semidiscrete formulation \eqref{eq:weak form coupled semidiscr probl} with $(\tbfv_N(t),\varphi(t),\bfpsi(t),r(t))$ for every $t \in [n \Delta t,(n+1) \Delta t)$ and integrate in time from $n \Delta t$ to $(n+1) \Delta t$ and then sum from $n=0$ to $N-1$ to obtain an integral over $[0,T]$.
By definition of the approximate solutions as made precise in \autoref{sec:approx sols on the complete time int}, for each $(\tbfv_N,\varphi,\bfpsi,r)$ in the test space with $\bfv \in \cX$, we get
\begin{equation*}
    \begin{aligned}
        &\quad \int_0^T\int_{\Omf} \Bigl(1 + \frac{\tau_{\Delta t}\omega_N}{R}\Bigr) \dt \bbfu_N \cdot \tbfv_N \srd \bfx \srd t +  2 \nu \int_0^T \int_{\Omf} \Bigl(1 + \frac{\tau_{\Delta t}\omega_N}{R}\Bigr)\bfD_\rf^{\tau_{\Delta t}\omega_N}(\bfu_N) : \bfD_\rf^{\tau_{\Delta t}\omega_N}(\tbfv_N) \srd \bfx \srd t\\
        &\quad +\int_0^T \int_{\Gamma} \Bigl(\frac{1}{2}\bfu_N \cdot \tau_{\Delta t}\bfu_N - p_N\Bigr)(\bfpsi - \tbfv_N) \cdot \bfn^{\tau_{\Delta t} \omega_N} \srd S \srd t\\
        &\quad + \frac{1}{2}\int_0^T \int_{\Omf} \Bigl(1 + \frac{\tau_{\Delta t}\omega_N}{R}\Bigr) \Bigl[\Bigl(\Bigl(\Bigl(\tau_{\Delta t}\bfu_N - \zeta_N\frac{R+z}{R}\bfe_3\Bigr) \cdot \nabla_\rf^{\tau_{\Delta t}\omega_N}\Bigr)\bfu_N\Bigr) \cdot \tbfv_N\\
        &\qquad -\Bigl(\Bigl(\Bigl(\tau_{\Delta t}\bfu_N - \zeta_N\frac{R+z}{R}\bfe_3\Bigr) \cdot \nabla_\rf^{\tau_{\Delta t}\omega_N}\Bigr)\tbfv_N\Bigr) \cdot \bfu_N\Bigr] \srd \bfx \srd t + \frac{1}{2R} \int_0^T \int_{\Omf} \zeta_N \bfu_N \cdot \tbfv_N \srd \bfx \srd t\\
        &\quad + \frac{1}{2}\int_0^T \int_{\Gamma} (\bfu_N - \zeta_N^* \bfe_3) \cdot \bfn^{\tau_{\Delta t}\omega_N} (\tau_{\Delta t}\bfu_N \cdot \tbfv_N) \srd S \srd t\\
        &\quad + \beta\sum_{i=1}^2 \int_0^T \int_{\Gamma} \cJ_\Gamma^{\tau_{\Delta t}\omega_N}(\zeta_N^* \bfe_3 - \bfu_N) \cdot \bftau_i^{\tau_{\Delta t}\omega_N}(\bfpsi - \tbfv_N) \cdot \bftau_i^{\tau_{\Delta t}\omega_N} \srd S \srd t\\
        &\quad + \rhob\int_0^T \int_{\Omb} \Bigl(\frac{\bfxi_N-\tau_{\Delta t}\bfxi_N}{\Delta t}\Bigr) \cdot \bfpsi \srd \bfx \srd t + \rhop \int_0^T \int_{\Gamma} \dt \overline{\zeta}_N \cdot \varphi \srd S \srd t + 2 \mue \int_0^T \int_{\Omb} \bfD(\bfeta_N) : \bfD(\bfpsi) \srd \bfx \srd t\\
        &\quad + \lambde \int_0^T \int_{\Omb} (\nabla \cdot \bfeta_N)(\nabla \cdot \bfpsi) \srd \bfx \srd t + 2 \muv \int_0^T \int_{\Omb} \bfD(\bfxi_N) : \bfD(\bfpsi) \srd \bfx \srd t\\
        &\quad + \lambdv \int_0^T \int_{\Omb} (\nabla \cdot \bfxi_N)(\nabla \cdot \bfpsi) \srd \bfx \srd t - \alpha \int_0^T \int_{\Omb} \cJ_\rb^{(\tau_{\Delta t}\bfeta_N)^\delta} p_N \nabla_\rb^{(\tau_{\Delta t}\bfeta_N)^\delta} \cdot \bfpsi \srd \bfx \srd t\\
        &\quad + c_0 \int_0^T \int_{\Omb} \del_t \overline{p}_N \cdot r \srd \bfx \srd t - \alpha \int_0^T \int_{\Omb} \cJ_\rb^{(\tau_{\Delta t}\bfeta_N)^\delta} \bfxi_N \cdot \nabla_\rb^{(\tau_{\Delta t}\bfeta_N)^\delta} r \srd \bfx \srd t\\
        &\quad - \alpha \int_0^T \int_{\Gamma} (\zeta_N^* \bfe_3 \cdot \bfn^{(\tau_{\Delta t}\bfeta_N)^\delta}) r \srd S \srd t + \kappa \int_0^T \int_{\Omb} \cJ_\rb^{(\tau_{\Delta t}\bfeta_N)^\delta} \nabla_\rb^{(\tau_{\Delta t}\bfeta_N)^\delta} p_N \cdot \nabla_\rb^{(\tau_{\Delta t}\bfeta_N)^\delta} r \srd \bfx \srd t\\
        &\quad - \int_0^T \int_{\Gamma} ((\bfu_N - \zeta_N^* \bfe_3) \cdot \bfn^{\tau_{\Delta t}\omega_N}) r \srd S \srd t + \int_0^T \int_\Gamma \Delta \omega_N \cdot \Delta \varphi \srd S \srd t = 0.
    \end{aligned}
\end{equation*}

Since many terms are analogous to those in classical fluid-structure interaction problems (see, e.g., \cite[Section 7.2]{MC:13a}), we will focus only on the terms that are specific to the poroelastic structure interaction problem or that involve the regularized displacement.

Let us start with a term involving the Jacobian of the regularized deformation gradient. Note that regularized terms will converge in arbitrary Sobolev spaces in the spatial variable, and thus are easy to handle. 
More precisely, since the regularization does not affect the time variable, using \autoref{lem:unif bddness of approx sols}, we conclude that $(\overline{\bfeta}_N)^\delta$ is uniformly bounded in~$\rW^{1,\infty}(0,T;\rL^2(\Omb)^3)$.
On the other hand, by definition of the regularization, for all $k \in \bN$, we have~$\| \nabla^k (\overline{\bfeta}_N)^\delta \|_{\rL^\infty(0,T;\rL^2(\Omb))} \le C$, for a constant $C = C(\delta)$ that is uniform in $N$.
Hence, in a similar way as in \autoref{lem:strong conv Biot displacement}, we deduce from the Aubin--Lions compactness lemma that there exists a strongly convergent subsequence, still denoted by $(\overline{\bfeta}_N)^\delta$, that converges to $\bfeta$ strongly in $\rC(0,T;\rC^k(\Omb)^3)$ for all $k \in \bN$.
Similarly to the plate displacement, by invoking the Lipschitz continuity in time and proceeding as in the proof of \autoref{lem:verification of assumptions}, we further argue that this convergence carries over to $(\tau_{\Delta t} \bfeta_N)^\delta$.
In particular, we conclude that
\begin{equation*}
    \cJ_\rb^{(\tau_{\Delta t}\bfeta_N)^\delta} = \det\bigl(\bfI + \nabla(\tau_{\Delta t} \bfeta_N)^\delta\bigr) \to \det\bigl(\bfI + \nabla \bfeta^\delta\bigr) = \cJ_\rb^{\bfeta^\delta}
\end{equation*}
strongly in $\rC(0,T;\rC(\Omb)^3)$.
The aforementioned convergence also implies that
\begin{equation*}
    \nabla_\rb^{(\tau_{\Delta t}\bfeta_N)^\delta} \cdot \bfpsi = \nabla \cdot \bfpsi \cdot \bigl(\bfI + \nabla(\tau_{\Delta t} \bfeta_N)^\delta\bigr)^{-1} \to \nabla \cdot \bfpsi \cdot \bigl(\bfI + \nabla \bfeta^\delta\bigr)^{-1} = \nabla_\rb^{\bfeta^\delta} \cdot \bfpsi
\end{equation*}
strongly in $\rC(0,T;\rC(\Omb)^3)$, and similarly for $\nabla_\rb^{(\tau_{\Delta t}\bfeta_N)^\delta} r \to \nabla_\rb^{\bfeta^\delta} r$.
Thus, from the strong convergence of $p_N$ in $\rL^2(0,T;\rL^2(\Omb))$ by \autoref{prop:comp for Biot pore pressure}, the weak* convergence of $\bfxi_N$ in $\rL^\infty(0,T;\rL^2(\Omb)^3)$, which follows from \autoref{lem:unif bddness of approx sols} and the relation between $\bfxi$ and $\bfeta$, as well as the weak convergence of $p_N$ in~$\rL^2(0,T;\rH^1(\Omb))$, we conclude that
\begin{equation*}
    \begin{aligned}
        - \alpha \int_0^T \int_{\Omb} \cJ_\rb^{(\tau_{\Delta t}\bfeta_N)^\delta} p_N \nabla_\rb^{(\tau_{\Delta t}\bfeta_N)^\delta} \cdot \bfpsi \srd \bfx \srd t
        &\longrightarrow - \alpha \int_0^T \int_{\Omb} \cJ_\rb^{\bfeta^\delta} p \nabla_\rb^{\bfeta^\delta} \cdot \bfpsi \srd \bfx \srd t.
    \end{aligned}
\end{equation*}
The desired convergence of the other three integrals involving the regularized Biot displacement follows likewise upon invoking the weak* convergence of $\bfxi_N$ in $\rL^\infty(0,T;\rL^2(\Omb)^3)$, which follows from \autoref{lem:unif bddness of approx sols} and the relation of $\bfxi$ and $\bfeta$, the weak convergence of $p_N$ in~$\rL^2(0,T;\rH^1(\Omb))$ as well as the weak* convergence of $\zeta_N^*$ in $\rL^\infty(0,T;\rL^2(\Gamma))$ following from \autoref{cor:weak & weak* conv}.

Next, we proceed with the term $\frac{1}{2R} \int_0^T \int_{\Omf} \zeta_N \bfu_N \cdot \tbfv_N \srd \bfx \srd t$.
To this end, let us note that
\begin{equation*}
    \begin{aligned}
        &\quad \left|\int_0^T \int_{\Omf} \zeta_N \bfu_N \cdot \tbfv_N \srd \bfx \srd t - \int_0^T \int_{\Omf} \dt \omega \bfu \cdot \tbfv \srd \bfx \srd t\right|
        \le \left|\int_0^T \int_{\Omf} (\zeta_N - \dt \omega) \bfu_N \cdot \tbfv_N \srd \bfx \srd t\right|\\
        &+ \left|\int_0^T \int_{\Omf} \dt \omega (\bfu_N - \bfu) \cdot \tbfv_N \srd \bfx \srd t\right|
       + \left|\int_0^T \int_{\Omf} \dt \omega \bfu \cdot (\tbfv_N - \tbfv) \srd \bfx \srd t\right|.
    \end{aligned}
\end{equation*}
For the first term, we recall from \eqref{eq:repr of bomega & bbfeta} and \autoref{cor:weak & weak* conv} that $\zeta_N \rightharpoonup^* \dt \omega$ in $\rL^\infty(0,T;\rL^2(\Gamma))$.
Moreover, we have $\bfu_N \cdot \tbfv_N \to \bfu \cdot \tbfv$ strongly in $\rL^2(0,T;\rL^2(\Omf)) \hookrightarrow \rL^1(0,T;\rL^2(\Omf))$.
Indeed, we recall $\bfu_N \to \bfu$ strongly in $\rL^2(0,T;\rL^2(\Omf)^3)$.
In \autoref{lem:conv of test fcts}, we have verified the uniform convergence~$\tbfv_N \to \tbfv$.
Concerning the second term, by virtue of the strong convergence of $\bfu_N \to \bfu$ in $\rL^2(0,T;\rL^2(\Omf)^3)$, it remains to argue that $\dt \omega \tbfv_N$ is uniformly bounded in $\rL^2(0,T;\rL^2(\Omf)^3)$.
In fact, from the above convergences, we infer that $\dt \omega \in \rL^\infty(0,T;\rL^2(\Omf))$, while $\tbfv_N$ is uniformly bounded in $\rL^\infty(0,T;\rL^\infty(\Omf)^3)$, yielding that $\dt \omega \tbfv_N \in \rL^\infty(0,T;\rL^2(\Omf)^3) \hookrightarrow \rL^2(0,T;\rL^2(\Omf)^3)$.
As a result the second term also tends to zero as $N \to \infty$.
For the last term, we invoke the uniform convergence $\tbfv_N \to \tbfv$ as well as the fact that $\dt \omega \in \rL^\infty(0,T;\rL^2(\Omf))$ and $\bfu \in \rL^2(0,T;\rL^2(\Omf)^3)$.
Finally, we conclude that
\begin{equation*}
    \frac{1}{2R} \int_0^T \int_{\Omf} \zeta_N \bfu_N \cdot \tbfv_N \srd \bfx \srd t \longrightarrow \frac{1}{2R} \int_0^T \int_{\Omf} \dt \omega \bfu \cdot \tbfv \srd \bfx \srd t.
\end{equation*}

Most remaining terms include the normal or tangential vectors $\bfn^{\tau_{\Delta t}\omega_N} = (-\del_x \tau_{\Delta t} \omega_N,-\del_y \tau_{\Delta t} \omega_N,1)$, $\bftau_1^{\tau_{\Delta t} \omega_N} = (1,0,\del_x \tau_{\Delta t} \omega_N)$ and $\bftau_2^{\tau_{\Delta t} \omega_N} = (0,1,\del_y \tau_{\Delta t} \omega_N)$, so we elaborate on their handling in the following.
In fact, the strong convergence of $\omega_N \to \omega$ in $\rL^\infty(0,T;\rH^s(\Gamma))$ for $s \in (0,2)$, which also carries over to~$\tau_{\Delta t} \omega_N$, together with the embedding $\rH^s(\Gamma) \hookrightarrow \rW^{1,p}(\Gamma)$ for all $p < \infty$ for $s \in (0,2)$ {sufficiently close to $2$} implies that $\bfn^{\tau_{\Delta t}\omega_N} \to \bfn^\omega$ and $\bftau_i^{\tau_{\Delta t}\omega_N} \to \bftau_i^\omega$ strongly in $\rL^\infty(0,T;\rL^p(\Gamma))$ for all~$p < \infty$.
From \autoref{lem:conv of test fcts} and the embedding $\rW^{1,p}(\Omf) \hookrightarrow \rC(\overline{\Omega}_\rf)$ for $p > 3$, we deduce that $\tbfv_N \to \tbfv$ uniformly on $(0,T) \times \overline{\Omega}_\rf$, so in particular, the uniform convergence holds on the interface $\Gamma$. To handle the quadratic term in the fluid velocity, we use \autoref{prop:conv of fluid vel traces} to conclude that $\left. \bfu_N \right|_{\Gamma} \to \left. \bfu \right|_{\Gamma}$ strongly in $\rL^2(0,T;\rH^s(\Gamma)^3)$ for~$s \in (0,\frac{1}{2})$, and therefore, by Sobolev embedding for $s$ sufficiently close to $\frac{1}{2}$, we have $\left. \bfu_N \right|_{\Gamma} \to \left. \bfu \right|_{\Gamma}$ strongly in~$\rL^3(0,T;\rL^3(\Gamma)^3)$. Moreover, the same argument applies to $\tau_{\Delta t} \bfu_N$. An analogous result holds for the fluid pore pressure $p_N$, and therefore we conclude:
\begin{equation*}
    \int_0^T \int_{\Gamma} \Bigl(\frac{1}{2}\bfu_N \cdot \tau_{\Delta t}\bfu_N - p_N\Bigr)(\bfpsi - \tbfv_N) \cdot \bfn^{\tau_{\Delta t} \omega_N} \srd S \srd t \longrightarrow \int_0^T \int_{\Gamma} \Bigl(\frac{1}{2}|\bfu|^2 - p\Bigr)(\bfpsi - \tbfv) \cdot \bfn^{ \omega} \srd S \srd t.
\end{equation*}
With regard to the term $\frac{1}{2}\int_0^T \int_{\Gamma} (\bfu_N - \zeta_N^* \bfe_3) \cdot \bfn^{\tau_{\Delta t}\omega_N} (\tau_{\Delta t}\bfu_N \cdot \tbfv_N) \srd S \srd t$, we can proceed similarly for the factors $\bfu_N$, $\bfn^{\tau_{\Delta t}\omega_N}$, $\tau_{\Delta t}\bfu_N$ and $\tbfv_N$, while we additionally make use of the weak* convergence of $\zeta_N^*$ to $\zeta$ in $\rL^\infty(0,T;\rL^2(\Gamma))$ to argue that
\begin{equation*}
    \frac{1}{2}\int_0^T \int_{\Gamma} (\bfu_N - \zeta_N^* \bfe_3) \cdot \bfn^{\tau_{\Delta t}\omega_N} (\tau_{\Delta t}\bfu_N \cdot \tbfv_N) \srd S \srd t \longrightarrow \frac{1}{2}\int_0^T \int_{\Gamma} (\bfu \cdot \bfn^{\omega} - \zeta \bfe_3 \cdot \bfn^{\omega}) \bfu \cdot \tbfv \srd S \srd t.
\end{equation*}
In the same way, we find that
\begin{equation*}
    - \int_0^T \int_{\Gamma} \bigl((\bfu_N - \zeta_N^* \bfe_3) \cdot \bfn^{\tau_{\Delta t}\omega_N}\bigr) r \srd S \srd t \longrightarrow - \int_0^T \int_{\Gamma} \bigl((\bfu - \zeta \bfe_3) \cdot \bfn^{\omega}\bigr) r \srd S \srd t.
\end{equation*}

The treatment of the remaining terms is more delicate.
Let us start with the term associated to the Beavers--Joseph--Saffman condition, namely $\sum_{i=1}^2 \int_0^T \int_{\Gamma} \cJ_\Gamma^{\tau_{\Delta t}\omega_N}(\zeta_N^* \bfe_3 - \bfu_N) \cdot \bftau_i^{\tau_{\Delta t}\omega_N}(\bfpsi - \tbfv_N) \cdot \bftau_i^{\tau_{\Delta t}\omega_N} \srd S \srd t$.
The difficulty in the limit passage for this term arises from $\cJ_\Gamma^{\tau_{\Delta t}\omega_N}$ and the factors in the tangential vectors.
For the first term involving $\bftau_1^{\tau_{\Delta t}\omega_N}$, we need to handle 
\begin{equation*}
    \frac{\sqrt{1 + |\del_x \omega_N|^2 + |\del_y \omega_N|^2}}{1 + |\del_x \omega_N|^2}.
\end{equation*}
As the consideration of the second term is analogous, we focus on the first one.
Therefore, we set $f_N \coloneqq \sqrt{1+|\del_x \omega_N|^2+|\del_y \omega_N|^2}$, $f \coloneqq \sqrt{1+|\del_x \omega|^2+|\del_y \omega|^2}$, $g_N \coloneqq 1 + |\del_x \omega_N|^2$ and $g \coloneqq 1 + |\del_x \omega|^2$.
As $g$ and $g_N$ are bounded from below by $1$, it follows that
\begin{equation}\label{eq:est of diff of fractions}
    \left|\frac{f_N}{g_N} - \frac{f}{g}\right| = \left|\frac{(f_N-f)g}{g_N g}\right| + |f| \cdot \left|\frac{g-g_N}{g_N g}\right| \le |f_N-f| + |f| \cdot |g-g_N|.
\end{equation}
Next, we employ the elementary identity $\sqrt{a} - \sqrt{b} = \frac{a-b}{\sqrt{a} + \sqrt{b}}$ for $a$, $b > 0$ to get
\begin{equation*}
    \begin{aligned}
        f_N - f
        &= \sqrt{1+|\del_x \omega_N|^2+|\del_y \omega_N|^2} - \sqrt{1+|\del_x \omega|^2+|\del_y \omega|^2}\\
        &= \frac{|\del_x \omega_N|^2 - |\del_x \omega|^2 + |\del_y \omega_N|^2 - |\del_y \omega|^2}{\sqrt{1+|\del_x \omega_N|^2+|\del_y \omega_N|^2} + \sqrt{1+|\del_x \omega|^2+|\del_y \omega|^2}}.
    \end{aligned}
\end{equation*}
Now, exploiting that the denominator of the preceding fraction is bounded below by two, we argue that
\begin{equation*}
    |f_N - f| \le \frac{1}{2} \bigl(|\del_x \omega_N|^2 - |\del_x \omega|^2 + |\del_y \omega_N|^2 - |\del_y \omega|^2\bigr).
\end{equation*}
The strong convergence of $\omega_N$ now implies that $f_N \to f$ in $\rL^\infty(0,T;\rL^p(\Gamma))$ for all $p < \infty$.
Concerning the second term in \eqref{eq:est of diff of fractions}, we proceed analogously as above, namely we expand it as
\begin{equation*}
    \begin{aligned}
        g - g_N
        &= \frac{|\del_x \omega|^2 - |\del_x \omega_N|^2 + |\del_y \omega|^2 - |\del_y \omega_N|^2 + |\del_x \omega|^2 \cdot |\del_y \omega|^2 - |\del_x \omega_N|^2 \cdot |\del_y \omega_N|^2}{\sqrt{1 + |\del_x \omega|^2 + |\del_y \omega|^2 + |\del_x \omega|^2 \cdot |\del_y \omega|^2} + \sqrt{1 + |\del_x \omega_N|^2 + |\del_y \omega_N|^2 + |\del_x \omega_N|^2 \cdot |\del_y \omega_N|^2}}.
    \end{aligned}
\end{equation*}
Again, we note that the denominator can be bounded from below by two, while for the numerator, we use that
\begin{equation*}
    |\del_x \omega|^2 \cdot |\del_y \omega|^2 - |\del_x \omega_N|^2 \cdot |\del_y \omega_N|^2 = \bigl(|\del_x \omega|^2 - |\del_x \omega_N|^2\bigr) \cdot |\del_y \omega|^2 + |\del_x \omega_N|^2 \cdot \bigl(|\del_y \omega|^2 - |\del_y \omega_N|^2\bigr),
\end{equation*}
and the terms $|\del_y \omega_N|^2$ and and the terms $|\del_y \omega_N|^2$ and $|\del_x \omega_N|^2$ can be controlled again in $\rL^r(\Gamma)$ for large~$r$.
Similarly, we find that $|f|$ can be bounded in $\rL^r(\Gamma)$ for large~$r$.
In total, with regard to \eqref{eq:est of diff of fractions}, we conclude that the difference under consideration converges strongly in~$\rL^\infty(0,T;\rL^r(\Gamma))$ for large $r \in (1,\infty)$.
Putting together these arguments, we derive the convergence
\begin{equation*}
    \begin{aligned}
        &\quad \beta\sum_{i=1}^2 \int_0^T \int_{\Gamma} cJ_\Gamma^{\tau_{\Delta t}\omega_N}(\zeta_N^* \bfe_3 - \bfu_N) \cdot \bftau_i^{\tau_{\Delta t}\omega_N}(\bfpsi - \tbfv_N) \cdot \bftau_i^{\tau_{\Delta t}\omega_N} \srd S \srd t\\
        &\longrightarrow \beta\sum_{i=1}^2 \int_0^T \int_{\Gamma} \cJ_\Gamma^{\omega}(\zeta \bfe_3 - \bfu) \cdot \bftau_i^{\omega}(\bfpsi - \tbfv) \cdot \bftau_i^{\omega} \srd S \srd t.
    \end{aligned}
\end{equation*}

Finally, we discuss the treatment of the convective terms, and we treat the four terms separately. 
Notice that the main difference from the classical FSI problems is that now the plate velocity strongly converges only in negative Sobolev spaces, since there is no direct dissipation acting on the plate.
First observe that $(1+\frac{\tau_{\Delta t}\omega_N}{R}) \to (1+\frac{\omega}{R})$ uniformly on $(0,T) \times \Omf$.
Let us start with $(\tau_{\Delta t} \bfu_N \cdot \nabla_\rf^{\tau_{\Delta t} \omega_N})\bfu_N \cdot \tbfv_N$.
From \autoref{lem:conv of test fcts}, we recall the uniform convergence of $\tbfv_N \to \tbfv$ in~$(0,T) \times \Omf$.
On the other hand, applying \autoref{lem:conv in L2Hs} to the sequence $\bfu_N$ as in the proof of \autoref{prop:conv of fluid vel traces}, we argue that $\bfu_N \to \bfu$ strongly in~$\rL^2(0,T;\rH^s(\Omf)^3)$ for $s \in (0,1)$.
Thus, by Sobolev embeddings, it follows in particular that~$\bfu_N \to \bfu$ strongly in $\rL^2(0,T;\rL^4(\Omf)^3)$.
Note that this convergence carries over to $\tau_{\Delta t} \bfu_N$.
Moreover, we invoke from \autoref{lem:lim of the gradient of the pressure} the weak convergence of $\nabla_\rf^{\tau_{\Delta t} \omega_N} \bfu_N$ to $\nabla_\rf^\omega \bfu$ in $\rL^2(0,T;\rL^p(\Omf)^3)$ for all $p \in (1,2)$.
A concatenation of these arguments leads to 
\begin{equation*}
    \begin{aligned}
        &\quad \int_0^T \int_{\Omf} \Bigl(1 + \frac{\tau_{\Delta t}\omega_N}{R}\Bigr) \Bigl(\Bigl(\tau_{\Delta t}\bfu_N \cdot \nabla_\rf^{\tau_{\Delta t}\omega_N}\Bigr)\bfu_N\Bigr) \cdot \tbfv_N \srd \bfx \srd t
        \longrightarrow \int_0^T \int_{\Omf} \Bigl(1 + \frac{\omega}{R}\Bigr) \bigl((\bfu \cdot \nabla_\rf^{\omega})\bfu\bigr) \cdot \tbfv \srd \bfx \srd t.
    \end{aligned}
\end{equation*}
In a similar way, this time using that $\nabla_\rf^{\tau_{\Delta t}\omega_N}\tbfv_N \to \nabla_\rf^\omega \tbfv$ strongly in $\rL^\infty(0,T;\rL^{p'}(\Omf)^{3 \times 3})$ for all $p' < \infty$ by \autoref{lem:conv of test fcts}, we obtain
\begin{equation*}
    \begin{aligned}
        \int_0^T \int_{\Omf} \Bigl(1 + \frac{\tau_{\Delta t}\omega_N}{R}\Bigr) \Bigl(\Bigl(\tau_{\Delta t}\bfu_N \cdot \nabla_\rf^{\tau_{\Delta t}\omega_N}\Bigr)\tbfv_N\Bigr) \cdot \bfu_N \srd \bfx \srd t
        \longrightarrow \int_0^T \int_{\Omf} \Bigl(1 + \frac{\omega}{R}\Bigr) \bigl((\bfu \cdot \nabla_\rf^{\omega})\tbfv\bigr) \cdot \bfu \srd \bfx \srd t.
    \end{aligned}
\end{equation*}
Next, we take into account the term $(\zeta_N\frac{R+z}{R}\bfe_3 \cdot \nabla_\rf^{\tau_{\Delta t}\omega_N}) \bfu_N \cdot \tbfv_N$.
Let us recall that $\nabla_\rf^{\tau_{\Delta t}\omega_N} \bfu_N \rightharpoonup \nabla_\rf^\omega \bfu$ weakly in $\rL^2(0,T;\rL^p(\Omf)^{3 \times 3})$ for all $p \in (1,2)$, while $\tbfv_N \to \tbfv$ strongly in $\rL^\infty(0,T;\rW^{1,{p'}}(\Omf)^3)$ for all~$p' < \infty$.
On the other hand, we recall from \autoref{prop:comp of Biot & plate vel} that $\zeta_N \to \zeta$ strongly in $\rL^2(0,T;\rH^{s}(\Gamma))$ for $s \in (-\frac{1}{2},0)$.
For $s' \in (-\frac{1}{2},0)$ and $q > 2$, we then derive the strong convergence of $\zeta_N \to \zeta$ in $\rL^2(0,T;\rW^{s',q}(\Gamma))$ for some $q > 2$.
Thus, recalling that $\bfw = \frac{R + z}{R} \dt \omega \bfe_3$ we deduce the convergence
\begin{equation*}
    \begin{aligned}
        \int_0^T \int_{\Omf} \Bigl(1 + \frac{\tau_{\Delta t}\omega_N}{R}\Bigr) \Bigl(\Bigl(\zeta_N\frac{R+z}{R}\bfe_3 \cdot \nabla_\rf^{\tau_{\Delta t}\omega_N}\Bigr)\bfu_N\Bigr) \cdot \tbfv_N \srd \bfx \srd t
        \longrightarrow \int_0^T \int_{\Omf} \Bigl(1 + \frac{\omega}{R}\Bigr) \bigl((\bfw \cdot \nabla_\rf^{\omega})\bfu\bigr) \cdot \tbfv \srd \bfx \srd t.
    \end{aligned}
\end{equation*}
The final term includes $(\zeta_N\frac{R+z}{R}\bfe_3 \cdot \nabla_\rf^{\tau_{\Delta t}\omega_N}) \tbfv_N \cdot \bfu_N$.
Here we note that $\nabla_\rf^{\tau_{\Delta t}\omega_N} \tbfv_N \to \nabla_\rf^\omega \tbfv$ strongly in~$\rL^\infty(0,T;\rL^p(\Omf)^{3 \times 3})$ for all $p < \infty$ by \autoref{lem:conv of test fcts}.
Moreover, we recall that $\zeta_N$ converges strongly in~$\rL^2(0,T;\rH^s(\Gamma))$ for $s \in (-\frac{1}{2},0)$, while as in the proof of \autoref{prop:conv of fluid vel traces}, we argue that $\bfu_N \to \bfu$ strongly in $\rL^2(0,T;\rH^{s'}(\Omf)^3)$ for all $s' < 1$.
From there, we deduce the existence of $s'' \in (0,s')$ and $q > 2$ such that $\bfu_N \to \bfu$ strongly in $\rL^2(0,T;\rW^{s'',q}(\Omf)^3)$.
In total, we conclude that
\begin{equation*}
    \begin{aligned}
        \int_0^T \int_{\Omf} \Bigl(1 + \frac{\tau_{\Delta t}\omega_N}{R}\Bigr) \Bigl(\Bigl(\zeta_N\frac{R+z}{R}\bfe_3 \cdot \nabla_\rf^{\tau_{\Delta t}\omega_N}\Bigr)\tbfv_N\Bigr) \cdot \bfu_N \srd \bfx \srd t
        \longrightarrow \int_0^T \int_{\Omf} \Bigl(1 + \frac{\omega}{R}\Bigr) \bigl((\bfw \cdot \nabla_\rf^{\omega})\tbfv\bigr) \cdot \bfu \srd \bfx \srd t.
    \end{aligned}
\end{equation*}
By combining the last four convergence results, we obtain
\begin{equation*}
    \begin{aligned}
        &\quad \frac{1}{2}\int_0^T \int_{\Omf} \Bigl(1 + \frac{\tau_{\Delta t}\omega_N}{R}\Bigr) \Bigl[\Bigl(\Bigl(\Bigl(\tau_{\Delta t}\bfu_N - \zeta_N\frac{R+z}{R}\bfe_3\Bigr) \cdot \nabla_\rf^{\tau_{\Delta t}\omega_N}\Bigr)\bfu_N\Bigr) \cdot \tbfv_N\\
        &\qquad -\Bigl(\Bigl(\Bigl(\tau_{\Delta t}\bfu_N - \zeta_N\frac{R+z}{R}\bfe_3\Bigr) \cdot \nabla_\rf^{\tau_{\Delta t}\omega_N}\Bigr)\tbfv_N\Bigr) \cdot \bfu_N\Bigr] \srd \bfx \srd t\\
        &\longrightarrow \frac{1}{2}\int_0^T \int_{\Omf} \Bigl(1 + \frac{\omega}{R}\Bigr) \Bigl[\bigl(\bigl((\bfu - \bfw) \cdot \nabla_\rf^{\omega}\bigr)\bfu\bigr) \cdot \tbfv -\bigl(\bigl((\bfu - \bfw) \cdot \nabla_\rf^{\omega}\bigr)\tbfv\bigr) \cdot \bfu\Bigr] \srd \bfx \srd t.
    \end{aligned}
\end{equation*}

This proves the first statement of \autoref{thm:ex of a weak sol}.

In the last part of the proof we derive an energy estimate for the weak solution $(\bfu,\bfeta,p,\omega)$ resulting from the limit passage in the splitting scheme as described above.
For this purpose, let us recall from~\eqref{eq:unif energy ineq} the estimate
\begin{equation*}
    E_N^{n+1} + \sum_{j=1}^n D_N^{j+1} \le E_0
\end{equation*}
for all $n=0,\dots,N-1$, where
\begin{equation*}
    \begin{aligned}
        E_N^{n+1}
        &\coloneqq \frac{1}{2} \int_{\Omf} \Bigl(1+\frac{\omega_N^n}{R}\Bigr) \bigl|\bfu_N^{n+1}\bigr|^2 \srd \bfx + \frac{\rhob}{2} \int_{\Omb} \bigl|\dbfeta_N^{n+1} \bigr|^2 \srd \bfx + \frac{\rhop}{2} \int_\Gamma \bigl|\zeta_N^{n+1}\bigr|^2 \srd S\\
        &\quad + \frac{c_0}{2} \int_{\Omb} \bigl|p_N^{n+1}\bigr|^2 \srd \bfx + \mue \int_{\Omb} |\bfD(\bfeta_N^{n+1})|^2 \srd \bfx + \frac{\lambde}{2} \int_{\Omb} |\nabla \cdot \bfeta_N^{n+1}|^2 \srd \bfx + \frac{1}{2}\int_\Gamma |\Delta \omega_N^{n+1}|^2 \srd S
    \end{aligned}
\end{equation*}
and 
\begin{equation*}
    \begin{aligned}
        D_N^{n+1}
        &\coloneqq 2 \nu (\Delta t)\int_{\Omf} \Bigl(1 + \frac{\omega_N^n}{R}\Bigr)\bigl|\bfD_\rf^{\omega_N^n}(\bfu_N^{n+1})\bigr|^2 \srd \bfx + 2\muv (\Delta t) \int_{\Omb} |\bfD(\dbfeta_N^{n+1})|^2 \srd \bfx\\
        &\quad + \lambdv (\Delta t) \int_{\Omb} |\nabla \cdot \dbfeta_N^{n+1}|^2 \srd \bfx + \kappa (\Delta t) \int_{\Omb} \cJ_\rb^{(\bfeta_N^n)^\delta} \bigl|\nabla_\rb^{(\bfeta_N^n)^\delta} p_N^{n+1}\bigr|^2 \srd \bfx\\
        &\quad + \beta(\Delta t) \sum_{i=1}^2 \int_{\Gamma} \cJ_\Gamma^{\omega_N^n} \bigl|(\dbfeta_N^{n+1} - \bfu_N^{n+1}) \cdot \bftau_i^{\omega_N^n}\bigr|^2 \srd S.
    \end{aligned}
\end{equation*}
Hence, employing the definition of the approximate solutions from \autoref{sec:approx sols on the complete time int}, invoking the relations of~$\dbfeta_N$ and $\bfxi_N$ as well as $\dbfeta_N$ and $\zeta_N^*$ on the interface $\Gamma$, and transforming to the respective moving domains, for almost all $t \in [0,T]$, we deduce the energy inequality
\begin{equation*}
    \begin{aligned}
        &\frac{1}{2} \int_{\Omega_{\rf,N}(t)} \bigl|\bfu_N\bigr|^2 \srd \bfx + \frac{\rhob}{2} \int_{\Omb} \bigl|\bfxi_N \bigr|^2 \srd \bfx + \frac{c_0}{2} \int_{\Omb} \bigl|p_N\bigr|^2 \srd \bfx + \mue \int_{\Omb} |\bfD(\bfeta_N)|^2 \srd \bfx + \frac{\lambde}{2} \int_{\Omb} |\nabla \cdot \bfeta_N|^2 \srd \bfx\\
        &+\frac{\rhop}{2} \int_\Gamma \bigl|\zeta_N\bigr|^2 \srd S + \frac{1}{2}\int_\Gamma |\Delta \omega_N|^2 \srd S + 2 \nu \int_0^t \int_{\Omega_{\rf,N}(s)} \bigl|\bfD(\bfu_N)\bigr|^2 \srd \bfx \srd s + 2\muv \int_0^t \int_{\Omb} |\bfD(\bfxi_N)|^2 \srd \bfx \srd s\\
        &+ \lambdv \int_0^t \int_{\Omb} |\nabla \cdot \bfxi_N|^2 \srd \bfx \srd s + \kappa \int_0^t \int_{\Omega_{\rb,N}^\delta} \bigl|\nabla p_N\bigr|^2 \srd \bfx \srd s + \beta \sum_{i=1}^2 \int_0^t \int_{\Gamma(s)} \bigl|(\zeta_N^* \bfe_3 - \bfu_N) \cdot \bftau_i\bigr|^2 \srd S \srd t \le E_0.
    \end{aligned}
\end{equation*}
Thus, exploiting the weak and weak* convergences collected in \autoref{cor:weak & weak* conv}, noting that the weak convergence of the symmetrized gradient $\bfD(\bfu_N)$ is obtained without loss of integrability on the moving domain due to the uniform energy estimates derived in \autoref{sec:approx sols on the complete time int}, and using the weak lower semicontinuity of the norms, we obtain the energy estimate stated in the second part of \autoref{thm:ex of a weak sol}. This completes the proof of the main result of this work, \autoref{thm:ex of a weak sol}.

\section{Conclusions}\label{sec:conclusions}

In this paper, we proved the existence of a finite-energy weak solution to a regularized 3D FSI problem between an incompressible, viscous fluid and a multilayered poro(visco)elastic structure. The structure is described by the Biot equations of poro(visco)elasticity, coupled with a thin reticular plate at the interface. The coupling between the fluid and the multilayered structure is nonlinear, which necessitates considering a moving interface. The resulting geometric nonlinearities led to considerable challenges in the existence proof. The proof is constructive, based on a splitting scheme that decouples the problem into a plate subproblem and a fluid-Biot subproblem. These subproblems are solved separately to construct approximate solutions. We then pass to the limit by establishing uniform bounds, deducing the corresponding weak and weak* convergences, and employing compactness arguments to obtain the strong convergence needed for the nonlinear terms. While the existence of a weak solution is shown for the poroelastic case, the poroviscoelastic case is also considered as it may be advantageous for future work on weak-classical consistency and for practical applications where dissipation is typically present in the poroelastic region.

The present work is the first that addresses a nonlinearly coupled FSI problem with a multilayered poro(visco)\-elastic structure in three dimensions. Compared to the study of the two-dimensional problem in \cite{KCM:24}, we faced several additional difficulties. First, as discussed in \autoref{ssec:maps from ref to phys dom}, the finite-energy space for the plate displacement only yields H\"older continuity in space, so the fluid domain could only be regarded as the subgraph of a H\"older-continuous function. Consequently, we had to invoke the Lagrangian trace, and the solution and test spaces needed to be defined carefully in \autoref{def:sol & test space}. The lack of Lipschitz regularity of the plate displacement required us to be cautious when switching between the fixed and moving domain configurations due to a potential loss of integrability or regularity. This also became apparent in the uniform boundedness of the approximate fluid velocities in \autoref{lem:unif bddness of approx sols}, which only covers $\rL^2(0,T;\rW^{1,p}(\Omf)^3)$ for $p \in (1,2)$ and excludes the case $p = 2$. This, in turn, had implications for the weak and weak* convergences in \autoref{cor:weak & weak* conv} and required us to identify the weak limit of the sequence of gradients of the fluid velocities in \autoref{ssec:conv of grad of fluid vel}. Another consequence of the lower regularity of the plate displacements is visible in the compactness proof for the family of fluid velocities, see \autoref{prop:rel comp fluid vel}. Here, an adjustment of the generalized Aubin--Lions compactness criterion, as initially presented in \cite{MC:19} and further extended in \cite{KCM:24}, was necessary to accommodate the loss of regularity. Finally, the limit passage in \autoref{ssec:limit passage} called for significantly more care with regard to the lack of Lipschitz continuity of the plate displacements.

In \cite{KCM:24}, a weak-classical consistency result was shown, i.e., it was proved in 2D that as the regularization parameter tends to zero, the weak solution to the regularized problem converges to a classical solution of the original problem, provided a classical solution exists. It is a natural question whether such a result carries over to the present 3D case. As already observed in \cite{KCM:24}, the current regularization procedure based on an odd extension does not seem appropriate due to the critical convergence rate of the convolution kernel in three dimensions. However, in a forthcoming work, we plan to show that a weak-classical consistency result can still be obtained by using a different extension of the structure displacement.

\medskip

{\bf Acknowledgments.}
{Felix Brandt would like to thank the German National Academy of Sciences Leopoldina for support through the Leopoldina Fellowship Program with grant number LPDS~2024-07. 
The work of Sun\v{c}ica \v{C}ani\'c has been supported in part by the US National Science Foundation under grants DMS-2408928, DMS-2247000, DMS-2011319.
Boris Muha was supported by the Croatian Science Foundation under project number IP-2022-10-2962 and the Croatia-USA bilateral grant
"The mathematical framework for the diffuse interface method applied to coupled problems in fluid dynamics".}

\appendix 

\section{The generalized Aubin--Lions compactness criterion}\label{sec:appendix}

In the following, we state the variant of the generalized Aubin--Lions compactness criterion from \cite{MC:19} tailored to our setting.
It is a consequence of \cite[Theorem~3.1]{MC:19} together with \cite[Remark~3.1(2)]{MC:19} and \cite[Theorem~8.3]{KCM:24}, see also the discussion in the proof of \cite[Corollary~4.4]{TW:20} for the adjustments of the second inequality of Property~(C1).
We will comment more on the assumptions below.

\begin{lem}\label{lem:gen Aubin--Lions compactness crit}
Consider Hilbert spaces $V$ and $H$ with $V \Subset H$, and let $V_N^n$ and $Q_N^n$ be (isomorphic to) Hilbert spaces with~$(V_N^n,Q_N^n) \hookrightarrow V \times V$, where the embeddings are uniformly continuous with respect to~$N$ and $n$, and where additionally
\begin{equation*}
    V_N^n \Subset \overline{Q_N^n}^H \hookrightarrow (Q_N^n)'.
\end{equation*}
Moreover, let $\bfu_N^n \in V_N^n$, $n=1,\dots,N$, and define $\bfu_N$ via $\bfu_N^n$ by piecewise constant approximation as in~\eqref{eq:pw const approx}.
Assume that $(\bfu_N) \subset \rL^2(0,T;H)$ such that the following properties (A)--(C) are fulfilled:
\begin{enumerate}
    \item[(A)] There is a universal constant $C > 0$ so that
    \begin{enumerate}
        \item[(A1)] $\| \bfu_N \|_{\rL^2(0,T;V)} \le C$ and
        \item[(A2)] $\| \bfu_N \|_{\rL^\infty(0,T;H)} \le C$.
    \end{enumerate}
    \item[(B)] There exists a constant $C > 0$ independent of $n$ and $N$, an exponent $p \in [1,2)$ and a sequence of nonnegative numbers $\{a_N^n\}_{n=0}^{N-1}$ for each $N$, satisfying $(\Delta t) \sum_{n=0}^{N-1} |a_N^n|^2 \le C$ uniformly in $N$ such that
    \begin{equation*}
        \left\| P_N^n \frac{\bfu_N^{n+1}-\bfu_N^n}{\Delta t} \right\|_{(Q_N^n)'} \le C\bigl(a_N^n + \| \bfu_N^n \|_{V_N^n} + \| \bfu_N^{n+1} \|_{V_N^{n+1}}\bigr)^p
    \end{equation*}
    for all $n = 0,1,\dots,N-1$.
    \item[(C)] The function spaces $Q_N^n$ and $V_N^n$ depend in a smooth way on time.
    More precisely, we have the following.
    \begin{enumerate}
        \item[(C1)] For every $N \in \bN$, or equivalently, for each $\Delta t > 0$, and for all $l \in \{1,\dots,N\}$ as well as $n \in \{1,\dots,N-1\}$, there exist a space $Q_N^{n,l} \subset V$ and operators $J_{N,l,n}^i \colon Q_N^{n,l} \to Q_N^{n+i}$, $i = 0,1,\dots,l$, so that $\| J_{N,l,n}^i \bfv \|_{Q_N^{n+i}} \le C \cdot \| \bfv \|_{Q_N^{n,l}}$ for all $\bfv \in Q_N^{n,l}$.
        Furthermore, for $C > 0$ independent of $N$, $n$ and $l$, it holds that
        \begin{equation*}
            \begin{aligned}
                \left(J_{N,l,n}^{j+1} \bfv - J_{N,l,n}^{j} \bfv,\bfu_N^{n+j+1} \right)_H
                &\le C (\Delta t) \cdot \| \bfv \|_{Q_N^{n,l}} \cdot \| \bfu_N^{n+j+1} \|_{V_N^{n+j+1}}, &&\tfor j \in \{0,\dots,l-1\},\\
                \| J_{N,l,n}^i \bfv - \bfv \|_{H}
                &\le C (l \Delta t)^{\frac{1}{4}} \cdot \| \bfv \|_{Q_N^{n,l}}, &&\tfor i \in \{0,\dots,l\}.
            \end{aligned}
        \end{equation*}
        \item[(C2)] Set $V_N^{n,l} = \overline{Q_N^{n,l}}^V$.
        There are functions $I_{N,l,n}^i \colon V_N^{n+i} \to V_N^{n,l}$, $i = 0,1,\dots,l$,  as well as a universal constant $C > 0$ with 
        \begin{equation*}
            \begin{aligned}
                \| I_{N,l,n}^i \bfv \|_{V_N^{n,l}} 
                &\le C \cdot \| \bfv \|_{V_N^{n+i}}, &&\tfor i \in \{0,\dots,l\},\\
                \| I_{N,l,n}^i \bfv - \bfv \|_{H}
                &\le g(l \Delta t) \| \bfv \|_{V_N^{n+i}}, &&\tfor i \in \{0,\dots,l\},
             \end{aligned}
        \end{equation*}
        where $g \colon \bR_+ \to \bR_+$ is a universal, monotonically increasing function with $g(h) \to 0$ as~$h \to 0$.
        \item[(C3)] The following Ehrling property is valid:
        For every $\delta > 0$, there is a constant $C(\delta) > 0$ independent of $n$, $l$ and $N$ with
        \begin{equation*}
            \| \bfv \|_H \le \delta \cdot \| \bfv \|_{V_N^{n,l}} + C(\delta) \cdot \| \bfv \|_{(Q_N^{n,l})'}, \tfor \bfv \in V_N^{n,l}.
        \end{equation*}
    \end{enumerate}
\end{enumerate}
Then the sequence $(\bfu_N)$ is relatively compact in $\rL^2(0,T;H)$.
\end{lem}

A few remarks on the generalized Aubin--Lions compactness criterion in \autoref{lem:gen Aubin--Lions compactness crit} are in order now.

\begin{remark}\label{rem:assumptions of adjusted generalized Aubin--Lions}
\begin{enumerate}[(a)]
    \item The relaxation of $V_N^n$ and $Q_N^n$ merely being isomorphic to Hilbert spaces does not affect the arguments in \cite[Section~3]{MC:19}.
    \item In \cite[Theorem~3.1]{MC:19}, in~(A), for the aforementioned time shift operator $\tau_{\Delta t}$, it is additionally assumed that $\| \tau_{\Delta t} \bfu_N - \bfu_N \|_{\rL^2(\Delta t,T;H)}^2 \le C \Delta t$.
    However, as pointed out in \cite[Remark~3.1(2)]{MC:19}, this assumption is not necessary.
    \item The result in \cite[Theorem~3.1]{MC:19} is originally formulated for the following more specific property instead of~(B) in our case:
    There exists a universal constant $C > 0$ so that for the orthogonal projector $P_N^n$ onto $\overline{Q_N^n}^H$, it holds that
    \begin{equation*}
        \left\| P_N^n \frac{\bfu_N^{n+1}-\bfu_N^n}{\Delta t} \right\|_{(Q_N^n)'} \le C \bigl(\| \bfu_N^{n+1} \|_{V_N^{n+1}} + 1\bigr), \tfor n=0,\dots,N-1.
    \end{equation*}
    It is shown in \cite[Theorem~8.3]{KCM:24} that the above condition can be replaced by~(B) in \autoref{lem:gen Aubin--Lions compactness crit}.
    \item A close inspection of \cite[Section~3]{MC:19} shows that the adjustments of the statements in~(C1) still give rise to the desired estimates, see also the discussion in the proof of \cite[Corollary~4.4]{TW:20}.
\end{enumerate}
\end{remark}

\end{document}